\pdfoutput=1
\documentclass[paper=a4, english, final ]{scrartcl}

\usepackage{babel}

\usepackage[T1]{fontenc}
\usepackage[utf8]{inputenc}

\usepackage{lmodern}

\usepackage[final, babel ]{microtype}

\usepackage{amsfonts} \usepackage{amssymb}  \usepackage{mathtools} \mathtoolsset{showonlyrefs} 

\providecommand\given{} \newcommand\SetSymbol[1][]{
   \mathrel{}\mathclose{}#1|\allowbreak\mathopen{}\mathrel{}}
\DeclarePairedDelimiterX\Set[1]{\lbrace}{\rbrace}{ \renewcommand\given{\SetSymbol[\delimsize]} #1 }

\DeclarePairedDelimiter\abs{\lvert}{\rvert}\DeclarePairedDelimiter\norm{\lVert}{\rVert}

\makeatletter
\let\oldabs\abs
\def\abs{\@ifstar{\oldabs}{\oldabs*}}
\let\oldnorm\norm
\def\norm{\@ifstar{\oldnorm}{\oldnorm*}}

\let\oldSet\Set
\def\Set{\@ifstar{\oldSet}{\oldSet*}}
\makeatother

\newcommand{\MoveEqLeftR}{\MoveEqLeft[3.3]}
 
\usepackage{chngcntr}
\counterwithin{equation}{subsection}

\usepackage{amsthm} \usepackage{thmtools, thm-restate}
\declaretheorem[name=Theorem, numberlike=equation]{theorem}
\declaretheorem[name=Proposition, numberlike=equation]{prop}

\declaretheorem[name=Definition, numberlike=equation, style=definition]{defn}
\declaretheorem[name=Lemma, numberlike=equation, style=definition]{lemma}
\declaretheorem[name=Corollary, numberlike=equation, style=definition]{cor}

\declaretheorem[name=Remark, numberlike=equation, style=remark]{rem}
\declaretheorem[name=Example, numberlike=equation, style=remark]{ex}

\usepackage{tikz-cd} \tikzset{
	pf/.style={commutative diagrams/.cd, every arrow, every label},
	surj/.style=commutative diagrams/two heads,
	inj/.style=commutative diagrams/hook,
	gl/.style=commutative diagrams/equal,
	mat/.style={matrix of math nodes, commutative diagrams/.cd, every cell},
	dr/.style={matrix of math nodes, commutative diagrams/.cd, every cell, column sep=small},
	seq/.style={matrix of math nodes, commutative diagrams/.cd, every cell, column sep=small}
	}
\newenvironment{diag*}{\[\begin{tikzpicture}[commutative diagrams/.cd, every diagram, baseline=(current bounding box.center)]}{\end{tikzpicture}\]\ignorespacesafterend}
\newenvironment{diag}{\begin{equation}\begin{tikzpicture}[commutative diagrams/.cd, every diagram, baseline=(current bounding box.center)]}{\end{tikzpicture}\end{equation}\ignorespacesafterend}

\usepackage{csquotes} \usepackage[style=alphabetic, isbn = false,
	]{biblatex}
\author{\texorpdfstring{Enno Keßler \and Artan Sheshmani \and Shing-Tung Yau}{
Enno Keßler, Artan Sheshmani, Shing-Tung Yau
}
}
\title{Super \texorpdfstring{\(\targetACI\)}{J}-holomorphic Curves: Construction of the Moduli Space}

\usepackage{stensor}
\DeclareRobustCommand{\tensor}{\stensor}

\usepackage{slashed}
\newcommand{\Dirac}{\slashed{D}}

\usepackage{dsfont}
\usepackage{faktor}

\usepackage[final, pdfusetitle ]{hyperref}

\DeclareMathOperator{\ACI}{I}

\DeclareMathOperator{\Div}{div}

\DeclareMathOperator{\End}{End}

\DeclareMathOperator{\Hom}{Hom}
\DeclareMathOperator{\im}{im}
\DeclareMathOperator{\id}{id}
\DeclareMathOperator{\ind}{ind}

\DeclareMathOperator{\Tr}{Tr}
\DeclareMathOperator{\UGL}{U}

\newcommand{\ic}{\mathrm{i}}

\newcommand{\dual}[1]{{#1}^{\vee}}

\newcommand{\cat}[1]{\mathsf{#1}}

\newcommand{\Red}[1]{{#1}_{red}}
\newcommand{\Smooth}[1]{{|#1|}}

\renewcommand{\d}{\mathop{}\!d}

\newcommand{\p}[1]{{p(#1)}}

\newcommand{\VSec}[1]{\Gamma\left(#1\right)}

\newcommand{\tangent}[2][]{T_{#1}#2} \newcommand{\differential}[1]{\d{#1}} \newcommand{\cotangent}[1]{\dual{T}#1}

\newcommand{\NaturalNumbers}{\mathbb{N}}
\newcommand{\Integers}{\mathbb{Z}}
\newcommand{\Z}{\Integers}

\newcommand{\RealNumbers}{\mathbb{R}}
\newcommand{\R}{\RealNumbers}
\newcommand{\ComplexNumbers}{\mathbb{C}}
\newcommand{\C}{\ComplexNumbers}

\newcommand{\ProjectiveSpace}[2][]{\mathbb{P}_{#1}^{#2}}

\newcommand{\cD}{\mathcal{D}}
\newcommand{\cO}{\mathcal{O}}

\newcommand{\SuperUpperHalfPlane}{S\mathbb{H}}
\newcommand{\targetACI}{J}
\newcommand{\DJBar}{\overline{D}_\targetACI}
\newcommand{\DelJBar}{\overline{\partial}_\targetACI}

\newcommand{\Fieldspace}{\mathcal{H}}
\newcommand{\Targetbundle}{\mathcal{E}}

\DeclareMathOperator{\DLaplace}{\Delta^\cD}
\DeclareMathOperator{\DBarLaplace}{\bar{\Delta}^\cD}
\DeclareMathOperator{\nablabar}{\overline{\nabla}}

\begin{document}
\maketitle
\begin{abstract}
	Let \(M\) be a super Riemann surface with holomorphic distribution~\(\cD\) and \(N\) a symplectic manifold with compatible almost complex structure~\(\targetACI\).
	We call a map \(\Phi\colon M\to N\) a super \(\targetACI\)-holomorphic curve if its differential maps the almost complex structure on \(\cD\) to \(\targetACI\).
	Such a super \(\targetACI\)-holomorphic curve is a critical point for the superconformal action and satisfies a super differential equation of first order.
	Using component fields of this super differential equation and a transversality argument we construct the moduli space of super \(\targetACI\)-holomorphic curves as a smooth subsupermanifold of the space of maps \(M\to N\).
\end{abstract}

\counterwithin{equation}{section}
\section{Introduction}
In this article we introduce super \(\targetACI\)-holomorphic curves, a generalization of \(\targetACI\)-holomorphic curves to super Riemann surfaces.
\(\targetACI\)-holomorphic curves have been extensively studied in the past forty years in mathematics and theoretical physics.
Interest sparked when Mikhail Gromov~\cite{G-PHCSM} observed that \(\targetACI\)-holomorphic curves can be used to define invariants of symplectic manifolds.
Around the same time, it was discovered that the count of maps from Riemann surfaces to symplectic manifold expresses certain features of topological superstring theory, see~\cite{W-TDGISTMC}.
Subsequently the theory of \(\targetACI\)-holomorphic curves was systematically developed by several authors, leading to the moduli space of stable maps and Gromov--Witten invariants, see, for example,~\cites{CDGP-PCYMESST}{RT-MTQC}{K-ERCVTA}{BM-SSMGMI}.
For an overview, we refer to the textbooks~\cites{HKKTPVVZ-MS}{McDS-JHCST}.

Supergeometry is the mathematical theory to combine geometry with anti-commuting variables as needed for supersymmetry; a foundational article is~\cite{L-ITS}.
Super Riemann surfaces are supergeometric generalizations of Riemann surfaces playing an important role in superstring theory, see~\cites{F-NSTTDCFT}{dHP-GSP}.
The moduli space of super Riemann surfaces has been studied in~\cites{D-LaM}{LBR-MSRS}{CR-SRSUTT}{DW-SMNP} and more recently by~\cites{FKP-RSSSMSPM}{BR-SMSCRP}.

Here, we propose a natural combination of both theories:
Super \(\targetACI\)-holomorphic curves are maps from a super Riemann surface to a symplectic manifold satisfying a first-order super differential equation.
This definition of super \(\targetACI\)-holomorphic curve couples the differential equations for \(\targetACI\)-holomorphic curves with spinors and is, in analogy to \(\targetACI\)-holomorphic curves, a critical point of a supergeometric action.
We construct the moduli space of super \(\targetACI\)-holomorphic curves for fixed super Riemann surfaces and codomain.
This moduli space is a finite-dimensional supermanifold obtained as the zero locus of a suitable super differential operator \(\DJBar\), seen as a section of an appropriate bundle over the space of maps.
In upcoming work we hope to extend those results to obtain a moduli space of super stable maps and associated invariants.

To be more precise, recall that a \(\targetACI\)-holomorphic curve is a map \(\varphi\colon \Sigma \to N\) from a Riemann surface \(\Sigma\) to an almost Kähler manifold \((N, \omega, \targetACI)\) such that
\begin{equation}
	\DelJBar\varphi
	= \frac12\left(1+\ACI\otimes\targetACI\right)\differential{\varphi}
	= 0.
\end{equation}
A \(\targetACI\)-holomorphic curve is a generalization of holomorphic maps to the almost Kähler case.
In particular, if the target is Kähler, that is a holomorphic manifold, any \(\targetACI\)-holomorphic curve is holomorphic.
Similarly, every \(\targetACI\)-holomorphic map is harmonic with respect to the Riemannian metric \(n\) on \(N\) obtained from \(\omega\) and \(\targetACI\).
More specifically, \(\targetACI\)-holomorphic maps are absolute minimizers of the Dirichlet action
\begin{equation}
	A(\varphi, g) = \int_\Sigma \norm{\differential{\varphi}}^2_{\dual{g}\otimes\varphi^*n} \d{vol}_g.
\end{equation}
The moduli space of \(\targetACI\)-holomorphic curves can be constructed in several different ways.
One possibility to construct the moduli space of \(\targetACI\)-holomorphic curves is to use the implicit function theorem if the operator \(\DelJBar\) is transversal to the zero section, see~\cite{McDS-JHCST}.

Our definition of super \(\targetACI\)-holomorphic curve allows to generalize those results to supergeometry as follows:
A super Riemann surface is a complex supermanifold of dimension \(1|1\) together with a holomorphic distribution \(\cD\subset\tangent{M}\) of rank \(0|1\) such that \(\cD\otimes\cD\cong\faktor{\tangent{M}}{\cD}\).
We call a map \(\Phi\colon M\to N\) a super \(\targetACI\)-holomorphic curve if
\begin{equation}
	\DJBar\Phi = \left.\frac12\left(1+\ACI\otimes\targetACI\right)\differential\Phi\right|_{\cD} \in\VSec{\dual{\cD}\otimes\Phi^*\tangent{N}}
\end{equation}
vanishes.
Here \(\ACI\) is the almost complex structure on \(\tangent{M}\).
Notice how the condition \(\DJBar\Phi=0\) depends on both, the almost complex structure \(\ACI\) on \(M\) and the distribution~\(\cD\).
This definition of super \(\targetACI\)-holomorphic curves satisfies, like the classical one, that it is a holomorphic map if the target is holomorphic and it is a critical point of the superconformal action
\begin{equation}
	A(\Phi) = \int_M \norm{\left.\differential{\Phi}\right|_{\cD}}^2 [\d{vol}].
\end{equation}
When expressed in component fields this action is the action of the two-dimensional non-linear supersymmetric sigma model or spinning string action functional~\cites{dHP-GSP}{DZ-CASS}{BdVH-LSRIASS}.

To see how our definition of super \(\targetACI\)-holomorphic curve extends the classical one, and in order to apply analytical methods to construct the moduli space, we will need to express the map \(\Phi\) in terms of its component fields.
To this end, let \(i\colon \Smooth{M}\to M\) be a map from a Riemann surface \(\Smooth{M}\) to the super Riemann surface \(M\) which is the identity on the topology.
By the results in~\cite{EK-SGSRSSCA}, the super Riemann surface is completely determined by a Riemannian metric \(g\) on \(\Smooth{M}\), a spinor bundle \(S\) and a spinor-valued differential form \(\chi\in\VSec{\cotangent{\Smooth{M}}\otimes S}\).
The map \(\Phi\) decomposes in component fields \(\varphi\colon \Smooth{M}\to N\), \(\psi\in\VSec{\dual{S}\otimes \varphi^*\tangent{N}}\) and \(F\in\VSec{\varphi^*\tangent{N}}\).
The first main result is then, see Corollary~\ref{cor:SJCComponentEquations} that \(\DJBar\Phi=0\) is equivalent to
\begin{equation}\label{eq:IntroSJCComponentEquations}
	\begin{aligned}
		&\left(1+\ACI\otimes\targetACI\right)\psi = 0 \\
		&F = 0, \\
		&\DelJBar\varphi + \left<Q\chi, \psi\right> + \frac14\Tr_{\dual{g}_S}\left(\gamma\otimes\mathfrak{j}\targetACI\right)\psi = 0, \\
		&\Dirac\psi - 2\left<\vee Q\chi, \differential{\varphi}\right> + \norm{Q\chi}^2\psi - \frac13SR^N(\psi) = 0.
	\end{aligned}
\end{equation}
Here \(Q\chi\) is the projection of \(\chi\) to the irreducible representation of the spin group of type~\(\frac32\).
For the precise definition of the endomorphism \(\mathfrak{j}\) and the curvature term \(SR^N(\psi)\) we refer to Section~\ref{Sec:DJBar}.
One can already notice that one obtains a \(\targetACI\)-holomorphic curve \(\varphi\colon \Smooth{M}\to N\) in case both \(Q\chi\) and \(\psi\) vanish.
In the case of a Kähler target \(N\) and a holomorphic \(i\colon\Smooth{M}\to M\) the Equations~\eqref{eq:IntroSJCComponentEquations} yield a \(\targetACI\)-holomorphic curve \(\varphi\) together with a holomorphic spinor \(\psi\in\VSec{\dual{S}\otimes_\C\varphi^*\tangent{N}}\).

In the remainder of the article we construct a moduli space of super \(\targetACI\)-holomorphic curves using a super differential geometric approach inspired by the construction of the moduli space of \(\targetACI\)-holormorphic curves in~\cite{McDS-JHCST}.
To this end, the underlying Riemann surface and its spinor bundle have to be fixed, while the gravitino may vary in a finite-dimensional superdomain \(\underline{\mathcal{X}}\).
Using the component fields we can equip the space of all maps \(\Fieldspace=\Hom(M, N)\) with the structure of a Fréchet supermanifold in the categorical framework to supergeometry as developed by Molotkov and Sachse in~\cites{M-IDCSM}{S-GAASTS}.
The operator \(\DJBar\Phi\) gives a smooth section \(\underline{\mathcal{S}}\colon \underline{\Fieldspace}\to \underline{\Targetbundle}\) of a smooth infinite-dimensional super vector bundle \(\underline{\Targetbundle}\) over \(\underline{\Fieldspace}\).
Its zero locus is the moduli space of super \(\targetACI\)-holomorphic maps \(\Phi\colon M\to N\).
If the section \(\underline{\mathcal{S}}\) is transversal to the zero-section one can apply an implicit function theorem in a suitable Banach space completion to obtain the main theorem:
\begin{restatable*}{theorem}{MainThm}\label{thm:MainThm}
Let \((N, \omega, \targetACI)\) be an almost Kähler manifold such that the operators \(D_\phi\) and  \(\Red{\Dirac}^{1,0}\) are surjective for all \(\phi\colon \Red{M}\to N\).
Then \(\underline{\mathcal{S}}\) is transversal to the zero section and the moduli space \(\underline{\mathcal{M}(\mathcal{X}, A)}\) of super \(\targetACI\)-holomorphic curves \(\Phi\colon M\to N\) is a smooth supermanifold fibered over \(\underline{\mathcal{X}}\) of relative dimension
	\begin{equation}
		2n(1-p) + 2\left<c_1(\tangent{N}), A\right>|2\left<c_1(\tangent{N}), A\right>.
	\end{equation}
	Here \(p\) is the topological genus of \(M\) and \(A\in H_2(N)\) is the homology class of the image of \(\Red{\Phi}\).
	Reduction yields the moduli space of \(\targetACI\)-holomorphic curves \(\phi\colon \Red{M}\to N\).

	More generally, if the operators \(D_\phi\) and \(\Red{\Dirac}^{1,0}\) are surjective only on a subset \(T\) of the \(\targetACI\)-holomorphic curves \(\phi\colon \Red{M}\to N\), we obtain a supermanifold structure on the restriction \(\underline{\mathcal{M}(\mathcal{X}, A)}|_U\) to an open neighbourhood \(U\supset T\).
\end{restatable*}
Here \(D_\phi\) is the linearisation of \(\DelJBar\phi\) around \(\phi\) and \(\Red{\Dirac}^{1,0}\) is the holomorphic twisted Dirac operator along \(\phi\).
It is known that for generic \(\targetACI\) the operator \(D_\phi\) is surjective for all \(\phi\colon \Red{M}\to N\).
We give criteria for the surjectivity of \(\Red{\Dirac}^{1,0}\) based on Bochner's method and examples in the last section.

One should notice that our definition of super \(\targetACI\)-holomorphic curve differs from the one in~\cite{G-HSCSSM} where it is required that \((1+\ACI\otimes\targetACI)\differential{\Phi}\) vanishes on all of \(\tangent{M}\) and not only on \(\cD\).
Consequently, in the work of Groeger the resulting super \(\targetACI\)-holomorphic curves are not critical points of the superconformal action and their moduli space is a classical manifold.
A moduli space of stable maps from Riemann surfaces to complex superschemes has been constructed before, see~\cite{AG-MSMSM}.
Such maps do not have a spinorial component field, in contrast to the case of super \(\targetACI\)-holomorphic maps from super Riemann surfaces to symplectic manifolds treated here.

The outline of this article is as follows:
In Section~\ref{Sec:DJBar} we will, after a brief recall of supergeometry and super Riemann surfaces, define the notion of super \(\targetACI\)-holomorphic curves in an almost Kähler manifold \(N\) as the zero set of the first order differential operator \(\DJBar\).
We will also derive the component field equations that will be crucial in Section~\ref{Sec:ModuliSpace} to identify the moduli space of super \(\targetACI\)-holomorphic curves in the space of all maps \(M\to N\).
The space of maps is constructed in Section~\ref{Sec:SpaceOfMaps} as an infinite dimensional supermanifold \(\underline{\Fieldspace}\) in the Molotkov--Sachse approach.
The operator \(\DJBar\) is described as a section \(\underline{\mathcal{S}}\) of a vector bundle \(\underline{\Targetbundle}\to\underline{\Fieldspace}\).
In Section~\ref{Sec:ModuliSpace} we show that the section \(\underline{\mathcal{S}}\) is transversal to the zero-section under certain conditions on the almost complex structure~\(\targetACI\) on \(N\).
Then, given transversality, we use the implicit function theorem to show that the moduli space of \(\targetACI\)-holomorphic curves from a fixed super Riemann surface \(M\) in an almost Kähler manifold \(N\) is a finite-dimensional supermanifold given by the zero locus of \(\underline{\mathcal{S}}\) in \(\underline{\Fieldspace}\).

\subsection*{Acknowledgements}
Enno Keßler was supported by a Deutsche Forschungsgemeinschaft Research Fellowship, KE2324/1--1.
Artan Sheshmani was supported partially by the NSF~DMS-1607871, NSF~PHY-1306313, the Simons~38558, and Laboratory of Mirror Symmetry NRU HSE, RF Government grant, ag. No~14.641.31.0001.
Shing-Tung Yau was partially supported by NSF~DMS-1607871, NSF~PHY-1306313, and Simons~38558.
We thank Vladimir Baranovsky, Christopher Beasley, Roman Bezrukavnikov, Dennis Borisov, Ron Donagi, Jürgen Jost, Valentino Tosatti and Ruijun Wu for useful discussion and comments.
The second author would like to thank Professors Yuri Manin, Sheldon Katz, Michael Kapranov, Ron Donagi and Tony Pantev who advised him about the notions of supergeometry in the past years.
 
\counterwithin{equation}{subsection}
\section{Definition of super \texorpdfstring{\(\targetACI\)}{J}-holomorphic curves}\label{Sec:DJBar}

\subsection{Supergeometry}\label{SSec:DJBarSRS}
We use in this work the ringed space approach to supermanifolds, sometimes also called the Kostant-Berezin-Leites approach.
See, for example,~\cites{L-ITS}{DM-SUSY}{CCF-MFS}{EK-SGSRSSCA}.
In this approach, a supermanifold of dimension~\(m|n\) is a locally ringed space \((M, \cO_M)\) which is locally isomorphic to
\begin{equation}
	\R^{m|n}=\left(\R^m, C^\infty(\R^m, \R)\otimes{\bigwedge}_n\right),
\end{equation}
the ringed space consisting of the topological space \(\R^m\) together with the sheaf of smooth real valued functions twisted by a real Grassmann-algebra in \(n\) generators.
If \(x^a\), \(a=1,\dotsc, m\) are the standard coordinates on \(\R^m\) and \(\eta^\alpha\), \(\alpha=1,\dotsc, n\) the generators of \({\bigwedge}_n\) we call \((x^a, \eta^\alpha)\) supercoordinates.
Maps between supermanifolds are required to preserve only the \(\Z_2\)-grading of the sheaf (and not the \(\Z\)-grading).
Many differential geometric concepts such as vector bundles, tangent bundles, Lie groups, principal bundles, connections and complex manifolds generalize to supergeometry in a straightforward manner if one uses the right notion of \(\Z_2\)-grading.

For example, the tangent bundle \(\tangent{M}\) of \(M\) is determined by the locally free sheaf of vector fields of rank \(m|n\).
In the coordinates \((x^a, \eta^\alpha)\) every tangent field \(X\) can be expanded as \(X=X^a\partial_{x^a} + X^\alpha\partial_{\eta^\alpha}\), where \(X^a\), \(X^\alpha\in\cO_M\).
The cotangent bundle is locally spanned by \(\differential{x^a}\), \(\d\eta^\alpha\) satisfying
\begin{equation}
	\left<X, \differential{f}\right> = Xf.
\end{equation}
We use the conventions set in~\cites{DF-SM}{DM-SUSY}, where vector fields are applied from the left to differential forms.

It has often been argued that the study of supermanifolds is incomplete without the use of additional odd parameters, see, for example,~\cites{DM-SUSY}{EK-SGSRSSCA}.
Consequently, we will generalize the notion of supermanifolds which are locally \(\R^{m|n}\) to families of supermanifolds which are locally of the form \(\R^{m|n}\times B\) for a supermanifold \(B\).
Globally, families of supermanifolds are given by a submersion \(M\to B\).
In addition, we will assume that all geometric constructions are functorial under base change, \(M\times_B B'\to B'\).
Hence, every supermanifold \(M\) can be seen as the trivial family \(M\times B\to B\) over \(B\).
Let us demonstrate the effect of the additional parameters in the following example:
\begin{ex}
	Assume that \(y^a\) are coordinates on \(\R^m\) and \((x^a, \eta^\alpha)\) are coordinates on~\(\R^{m|n}\) and let \(i\colon \R^m\to\R^{m|n}\) be a map which is the identity on the topological spaces.
	By the Charts Theorem, see~\cite[Section~2.1.7]{L-ITS}, the map \(i\) can be given in terms of the coordinates and the only possibility is
	\begin{align}
		i^\#x^a &= y^a, &
		i^\#\eta^\alpha &= 0,
	\end{align}
	because the map of sheaves \(i^\#\colon \cO_{\R^{m|n}}\to \cO_{\R^m}\) is a \(\Z_2\)-preserving algebra-homomorphism and \(0\in\cO_{\R^m}\) is the only odd element.

	Suppose now that we have a map \(i\colon \R^m\times B\to \R^{m|n}\times B\) over \(B\) which is again the identity on topological spaces.
	By the Charts Theorem it is again sufficient to give the map in coordinates:
	\begin{align}
		i^\#x^a &= y^a + \xi^a, &
		i^\#\eta^\alpha &= \zeta^\alpha,
	\end{align}
	for nilpotent even functions \(\xi^a\in\cO_{\R^m\times B}\) and odd functions \(\zeta^\alpha\in\cO_{\R^m\times B}\).
	Notice that it is only possible to choose non-zero such functions if the base~\(B\) possesses odd dimensions.
\end{ex}
More generally, it has been shown in~\cite[Chapter~3.3]{EK-SGSRSSCA} that for every supermanifold~\(M\) over \(B\) locally isomorphic to \(\R^{m|n}\times B\) there exists a supermanifold \(\Smooth{M}\) over \(B\) locally isomorphic to \(\R^{m|0}\times B\) and a map \(i\colon\Smooth{M}\to M\) over \(B\) such that \(i\) is the identity on topological spaces.
Without loss of generality, one can assume \(\Smooth{M}=\Red{M}\times B\).
The map \(i\colon \Smooth{M}\to M\) is of great importance to understand in what way supergeometry extends classical geometry and to the understanding of supersymmetry in terms of supergeometry.

In the remainder of Chapter~\ref{Sec:DJBar} we will mostly suppress the \(B\)-dependence.
Instead we will assume that every supermanifold is a not necessarily trivial family of supermanifolds over \(B\).
Every map and all other geometric constructions are also understood to be relative over \(B\).

\subsection{Super Riemann surfaces}
The notion of Riemann surface does not generalize in a straightforward way to supergeometry.
For a list of possible candidates see~\cite[Chapter~4]{S-GAASTS}.
One of those possible candidates, called super Riemann surface or sometimes SUSY-curve, stands out as its theory allows for so many analogies to the theory of Riemann surfaces.
We use the following definition from~\cite{LBR-MSRS}, but the concept goes back at least to~\cite{F-NSTTDCFT}:
\begin{defn}
	A super Riemann surface is a complex supermanifold \(M\) of dimension~1|1 together with a holomorphic distribution \(\cD\subset TM\) of complex rank \(0|1\) such that the commutator of vector fields induces an isomorphism
	\begin{equation}\label{eq:IsoSRS}
			\frac12[\cdot, \cdot]\colon \cD\otimes \cD\to \faktor{TM}{\cD}.
	\end{equation}
\end{defn}

\begin{ex}\label{ex:LocalSRS}
	Let \((z, \theta)\) be holomorphic coordinates on \(\C^{1|1}\).
	The distribution \(\cD\) generated by \(D=\partial_\theta + \theta\partial_z\) defines a super Riemann surface structure on \(\C^{1|1}\).
	Every super Riemann surface is locally isomorphic to \(\C^{1|1}\) with this standard super Riemann surface structure, see~\cite{LBR-MSRS}.
	Holomorphic coordinates \((z, \theta)\) such that \(\cD=\left<\partial_\theta + \theta\partial_z\right>\) are called superconformal coordinates.
\end{ex}

Further examples of super Riemann surfaces include the completion of \(\C^{1|1}\) to the super projective space \(\ProjectiveSpace[\C]{1|1}\) and \(\SuperUpperHalfPlane\), the restriction of \(\C^{1|1}\) to the upper half-plane.
The Uniformization Theorem for super Riemann surfaces, see~\cites{CR-SRSUTT}[Theorem~9.4.1]{EK-SGSRSSCA} states that \(\ProjectiveSpace[\C]{1|1}\), \(\SuperUpperHalfPlane\) and \(\C^{1|1}\) are the only simply connected super Riemann surfaces.
Consequently, any super Riemann surface of genus zero is isomorphic to \(\ProjectiveSpace[\C]{1|1}\), whereas the universal cover of a super Riemann surface of genus one is \(\SuperUpperHalfPlane\) and the universal cover of super Riemann surfaces with genus two or higher is \(\C^{1|1}\).
Here the genus of a super Riemann surface \(M\) is the topological genus of its reduced space \(\Red{M}\).

In~\cites{JKT-SRSMG}{EK-SGSRSSCA} an analogue of harmonic maps on super Riemann surfaces has been studied.
To this end, one needs to introduce compatible Riemannian metrics on the super Riemann surface \(M\).
We call a supersymmetric bilinear form \(m\) on \(M\) a Riemannian metric compatible with the super Riemann surface structure, if it is compatible with the almost complex structure \(\ACI\) on \(\tangent{M}\), that is \(m(\ACI X, \ACI Y) = m(X, Y)\) and turns the isomorphism~\eqref{eq:IsoSRS} into an isomorphism of hermitian line bundles.
In addition one requires positivity of the reduction of \(m\).
The choice of such a metric reduces the structure group of \(M\) to \(\UGL(1)\).
With respect to an \(\UGL(1)\)-frame \(F_A\) the metric is given by
\begin{align}
	m(F_a, F_b) &= \delta_{ab}, &
	m(F_a, F_\beta) &= 0, &
	m(F_\alpha, F_\beta) &= \varepsilon_{\alpha\beta}.
\end{align}
Here we use the customary notation that capital latin letters, for example \(A, B, \dotsc\) refer to all, even and odd, indices, whereas small latin letters such as \(a, b,\dotsc\) only refer to the indices of the even part and small greek letters \(\alpha, \beta,\dotsc\) refer to indices in the odd part.
Furthermore, \(\delta_{ab}\) is the Kronecker-delta and \(\varepsilon_{\alpha\beta}\) is the completely anti-symmetric tensor.

\begin{ex}\label{ex:SCMetricsFrames}
	Let \(D=\partial_{\theta}+\theta\partial_z\) be the generator of \(\cD\) in some superconformal coordinates as in Example~\ref{ex:LocalSRS}.
	For any real function \(\lambda\) and complex function \(l\) the frames
	\begin{align}
		\lambda^2\partial_z + lD &&
		\lambda D
	\end{align}
	define a hermitian form on \(\C^{1|1}\) whose real form gives a Riemannian metric compatible with the super Riemann surface structure.
	Of particular importance is the case
	\begin{align}
		\lambda &= \sqrt{\Im z + \frac12\ic\theta\bar\theta} &
		l &= \frac\ic2\left(\bar\theta - \theta\right)
	\end{align}
	for the upper half plane in \(\C^{1|1}\) because this is an analogue the constant negative curvature metric on super Riemann surfaces.
	Different functions \(\lambda\) and \(l\) induce different metrics lying in the same superconformal class.
\end{ex}
Notice that every such metric \(m\) yields a orthogonal sum decomposition \(\tangent{M}=\cD\oplus\cD^\perp\), where \(\cD^\perp\) is isometrically isomorphic to \(\cD\otimes\cD\).
Another metric \(\tilde{m}\) compatible with the super Riemann surface structure can differ from \(m\) by a superconformal rescaling and a superconformal change of splitting.
A superconformal rescaling, locally described by a change of \(\lambda\) in Example~\ref{ex:SCMetricsFrames}, is a rescaling of the metric on \(\cD\).
A superconformal splitting is a change of the embedding \(\cD^\perp\hookrightarrow\tangent{M}\) and is locally described by a change of the function \(l\).

For any map \(\Phi\colon M\to N\) to a Riemannian supermanifold \((N,n)\) one defines the superconformal action
\begin{equation}
	A(\Phi) = \int_{M}\norm{\left.\differential{\Phi}\right|_{\cD}}^2_{\dual{m}\otimes\Phi^*n}[\d{vol}_g]
\end{equation}
using the theory of Berezin integrals on supermanifolds.
Its critical points satisfy a second order differential equation which can be expressed in terms of a \(\UGL(1)\)-frame \(F_A\) as
\begin{equation}
	0 = \DLaplace\Phi
	= m^{\alpha\beta}\left(\nabla_{F_\alpha}F_\beta\Phi + \left(\Div_m F_\alpha\right)F_\beta\Phi\right).
\end{equation}
The main motivation for the superconformal action is that it reduces to the spinning string action when expressed in component fields.
For more on component fields, see Section~\ref{SSec:DJBarComponentFields} below and~\cite{EK-SGSRSSCA}.

Metrics compatible with the super Riemann surface structure induce additional structure on \(\Smooth{M}\) for fixed \(i\colon \Smooth{M}\to M\).
The isomorphism \(\tangent{\Smooth{M}}=i^*\cD^\perp\) induces a Riemannian metric \(g\) and an almost complex structure \(\ACI\) on \(\tangent{M}\) which are compatible with each other.
The bundle \(S=i^*\cD\) is a spinor bundle on \(\Smooth{M}\) of real rank \(0|2\) and with induced metric~\(g_S\).
The gravitino \(\chi\in\VSec{\cotangent{M}\otimes S}\) is defined by \(\chi=-p_S\circ \differential{i}\), where \(p_S\colon i^*\tangent{M}\to S\) is the orthogonal projector.
The main result is now:
\begin{theorem}[{\cites{JKT-SRSMG}[Corollary~11.3.6]{EK-SGSRSSCA}}]\label{thm:BijectionSRSMetricSpinorBundleGravitino}
	There is a bijection of sets
	\begin{multline}
		\Set{\text{Super Riemann surfaces \(M\) with fixed \(i\colon \Smooth{M}\to M\)}}\\
		\longleftrightarrow
		\faktor{\Set{(g, S, \chi)\text{ on \(\Smooth{M}\)}}}{(g, S, \chi)\sim(\lambda^2g, S, \chi+\gamma s)}.
	\end{multline}
\end{theorem}
We will not repeat the full proof here and refer to~\cites{JKT-SRSMG}{EK-SGSRSSCA}.
Instead we discuss some consequences of Theorem~\ref{thm:BijectionSRSMetricSpinorBundleGravitino} and fix further notation relevant to this article.

Notice that metrics related by a Weyl transformation \(g\sim\lambda^2g\) and gravitinos related by a super Weyl transformation \(\chi\sim \chi+\gamma s\) for \(s\in\VSec{S}\) give the same super Riemann surface.
The Weyl and super Weyl transformations on metric and gravitino are induced by different choices of metric \(m\) on the super Riemann surface.
In order to make this more transparent let \(F_A\) be a real \(\UGL(1)\)-frame on \(M\) for a chosen metric \(m\).
Then \(s_\alpha=i^*F_\alpha\) is an orthonormal spinor frame (\(g_S(s_\alpha, s_\beta)=\varepsilon_{\alpha\beta}\)) and we can always write \(i^*F_a = f_a + \chi(f_a)\), where \(f_a\) is a \(g\)-orthonormal frame on \(\Smooth{M}\) (\(g(f_a, f_b)=\delta_{ab}\)).
One can directly see how the different functions \(\lambda\) and \(l\) in Example~\ref{ex:SCMetricsFrames} induce metrics \(g\) and gravitino related by Weyl and super Weyl transformations.
One notices also that in general \(F_A\) is not completely determined by \(f_a\), \(s_\alpha\), and \(\chi\).
However, by modifying only the nilpotent part of the metric~\(m\) and choosing appropriate coordinates on \(M\) one can bring \(F_A\) in the form of a Wess--Zumino frame which is completely determined by \(f_a\), \(s_\alpha\) and \(\chi\), see~\cite[Chapter~11]{EK-SGSRSSCA} and also the earlier~\cite{H-SWT}.
This is the main ingredient for the proof of Theorem~\ref{thm:BijectionSRSMetricSpinorBundleGravitino}.

We have denoted the Clifford multiplication on \(S\) by \(\gamma\colon \cotangent{\Smooth{M}}\to \End{S}\).
It satisfies \(\gamma(X)\gamma(Y)s + \gamma(Y)\gamma(X)s = 2g(X, Y) s\) and in dimension two it allows to define an isomorphism \(\Gamma\colon S\otimes S\to \tangent{M}\) using contraction with the metrics \(g\) and \(g_S\).
This implies that there is an orthogonal sum decomposition \(\tangent{\Smooth{M}}\otimes S = S\oplus S\otimes_\C S\otimes_\C S\), where the projectors on the factors are given by
\begin{align}
	P\chi &= \frac12 f_a \otimes \gamma(f^a)\gamma(f^b)\chi(f_b), &
	Q\chi &= \frac12 f_a\otimes \gamma(f^b)\gamma(f^a)\chi(f_b).
\end{align}
One checks that the image of \(P\) is isomorphic to \(S\) by using the map \(\delta_\gamma\colon X\otimes s\mapsto \gamma(X)s\).
Its kernel is given by \(\tangent{\Smooth{M}}\otimes_\C S= S\otimes_\C S\otimes_\C S\).
This gives a direct sum decomposition of the real tensor product \(\tangent{\Smooth{M}}\otimes S\) in irreducible representations of \(\UGL(1)\), one of type \(\frac12\) and one of type \(\frac32\).
Only the \(\frac32\)-part of the gravitino is relevant to the reconstruction of the super Riemann surface.
Its \(\frac12\)-part can be arbitrarily modified by a super Weyl-transformation.

\subsection{The operator \texorpdfstring{\(\DJBar\Phi\)}{DJBarPhi}}\label{SSec:DJBarOperator}
Let \((N, \omega)\) be a symplectic manifold of dimension \(2n\).
We assume that \(\targetACI\in\VSec{\End{\tangent{N}}}\) is a compatible almost complex structure, that is, \(\targetACI^2=-\id_{\tangent{N}}\), \(\omega(\targetACI X, \targetACI Y)=\omega(X, Y)\) and \(\omega(\targetACI X, X)>0\) for all \(X\), \(Y\in\VSec{\tangent{N}}\).
In that case we can define a Riemannian metric
\begin{equation}
	n(X, Y) = \omega(\targetACI X, Y)
\end{equation}
on \(N\), such that \(\targetACI\) is an isometry.
We denote the Levi-Civita covariant derivative on \(N\) by \(\nabla\).
We will furthermore need the covariant derivative
\begin{equation}\label{eq:DefinitionNablaBar}
	\nablabar_X Y = \nabla_X Y - \frac12\targetACI\left(\nabla_X\targetACI\right) Y.
\end{equation}
Notice that \(\nablabar\targetACI=0\).
In contrast, \(\nabla\targetACI=0\) holds only if the almost complex structure \(\targetACI\) is integrable, that is, the \((N, \omega, \targetACI)\) is Kähler.

\begin{defn}
	Let \(\Phi\colon M\to N\) be a map from a super Riemann surface \(M\) to \(N\).
	We define the operator \(\DJBar\Phi\in\VSec{\dual{\cD}\otimes \Phi^*\tangent{N}}\) by
	\begin{equation}
		\DJBar\Phi = \frac12\left.\left(\differential{\Phi} + \targetACI\circ \differential{\Phi}\circ \ACI\right)\right|_{\cD}.
	\end{equation}
	We will call maps \(\Phi\) such that \(\DJBar\Phi=0\) super \(\targetACI\)-holomorphic curves.
\end{defn}

We use the name \enquote{super \(\targetACI\)-holomorphic curve} because we think it generalizes \(\targetACI\)-holomorphic curves, where the domain is a Riemann surface, in an interesting way.
We will see in Corollary~\ref{Cor:JHolomorphicCurvesTrivialFamilySRS} below that for trivial families of super Riemann surfaces and underlying even manifold \(\Smooth{M}\to M\), a super \(\targetACI\)-holomorphic curves is equivalent to a \(\targetACI\)-holomorphic curve \(\varphi\colon \Smooth{M}\to M\) together with a holomorphic twisted spinor \(\psi\in\VSec{\dual{S}\otimes\varphi^*\tangent{N}}\).
In addition, super \(\targetACI\)-holomorphic curves satisfy the following supergeometric properties in analogy to \(\targetACI\)-holomorphic curves on Riemann surfaces.
First, super \(\targetACI\)-holomorphic curves are holomorphic if the target is Kähler, see Lemma~\ref{lemma:KählerJholomorphicCurvesAreHolomorphic} below.
Second, super \(\targetACI\)-holomorphic curves satisfy \(\DLaplace\Phi=0\), see Proposition~\ref{prop:SuperJholomorphicCurvesAreDHarmonic} below.
That is, they are critical points of the superconformal action functional, in analogy to the fact that \(\targetACI\)-holomorphic curves are harmonic.

\begin{ex}[Super \(\targetACI\)-holomorphic curves  from \(\C^{1|1}\) to \(\R^{2n}\)]\label{ex:SJCtoR2n}
	Let \((z, \theta)\) be the standard superconformal coordinates from Example~\ref{ex:LocalSRS}, that is, \(\cD\) is generated by \(D=\partial_\theta + \theta\partial_z\).
	Let \((x^a, \eta^\alpha)\), \(a=1, 2\), \(\alpha=3, 4\) be the real coordinates on \(\R^{2|2}\) such that
	\begin{align}
		z &= x^1 + \ic x^2, &
		\theta &= \eta^3 + \ic \eta^4.
	\end{align}
	The standard super Riemann surface structure is given by the frames \(\partial_{x^1}\), \(\partial_{x^2}\), \(D_3\) and \(D_4\) in \(\VSec{\tangent{\R^{2|2}}}\) such that
	\begin{align}
		\partial_z &= \frac12\left(\partial_{x^1} - \ic\partial_{x^2}\right) &
		\partial_{\overline{z}} &= \frac12\left(\partial_{x^1} + \ic\partial_{x^2}\right) \\
		D &= \frac12\left(D_3 - \ic D_4\right) &
		\overline{D} &= \frac12\left(D_3 + \ic D_4\right)
	\end{align}
	in \(\VSec{\tangent{\R^{2|2}}}\otimes\C\).
	Here
	\begin{align}
		D_3
		&= \partial_{\eta^3} + \eta^3\partial_{x^1} + \eta^4\partial_{x^2}
		=\partial_{\eta^3} + \Gamma_{3\nu}^k\eta^\nu \partial_{x^k}, \\
		D_4
		&= \partial_{\eta^4} + \eta^3\partial_{x^2} - \eta^4\partial_{x^1}
		= \partial_{\eta^4} + \Gamma_{4\nu}^k\eta^\nu \partial_{x^k}.
	\end{align}
	The almost complex structure \(\ACI\) on \(\R^{2|2}\) is given by
	\begin{align}
		\ACI\partial_{x^1} &= \partial_{x^2} &
		\ACI D_3 &= D_4
	\end{align}
	and the distribution \(\cD\) is generated by \(D_3\), \(D_4\).

	Let us call coordinates on \(\R^{2n}\) by \(Y^b\) and assume \(\targetACI\partial_{Y^a}=\tensor{\targetACI}{_b^c}\partial_{Y^c}\).
	Then, for any map \(\Phi\colon \C^{1|1}\to \R^{2n}\),
	\begin{equation}
		\begin{split}
			\frac12\left(1+\ACI\otimes\targetACI\right)\differential{\Phi}|_{\cD}
			={}& \frac12\left(1+\ACI\otimes \targetACI\right) D^\alpha\otimes \left(D_\alpha(\Phi^\#Y^b)\partial_{Y^b}\right) \\
			={}& \frac12 D^\alpha\left(D_\alpha\Phi^\#Y^b + \tensor{\ACI}{_\alpha^\beta}D_\beta\Phi^\#Y^c\tensor{\targetACI}{_c^b}\right)\partial_{Y^b}
		\end{split}
	\end{equation}
	That is, \(\Phi\) is a super \(\targetACI\)-holomorphic curve if and only if for all \(a=1, \dotsc, 2n\)
	\begin{equation}
		D_3\left(\Phi^\#Y^b\right) + D_4\left(\Phi^\#Y^c\right)\tensor{\targetACI}{_c^b} = 0.
	\end{equation}
\end{ex}
\begin{ex}[Super \(\targetACI\)-holomorphic curves from \(\C^{1|1}\) to \(\C^n\)]
	Let \((r^b, s^b)\), \(b=1,\dotsc, n\) be coordinates on \(\R^{2n}\).
	The standard symplectic form on \(\R^{2n}\) is given by
	\begin{equation}
		\omega = -\sum_{a=1}^n \d{r^b}\wedge \d{s^b}
	\end{equation}
	and the corresponding almost complex structure by
	\begin{align}
		\targetACI\partial_{r^b} &= \partial_{s^b}, &
		\targetACI\partial_{s^b} &= -\partial_{r^b}.
	\end{align}
	By Example~\ref{ex:SJCtoR2n}, a map \(\Phi\colon\C^{1|1}\to \C^n\) is super \(\targetACI\)-holomorphic, if for \(b=1, \dotsc, n\)
	\begin{align}\label{eq:DelJBarHolomCoord}
		\left(D_3\Phi^\#r^b - D_4\Phi^\#s^b\right) &= 0, &
		\left(D_3\Phi^\#s^b + D_4\Phi^\#r^b\right) &= 0.
	\end{align}

	The complex coordinates \(Z^b=r^b+\ic s^b\in\cO_{\R^{2n}}\otimes\C\) are also holomorphic coordinates for \(\C^n = (\R^{2n}, \targetACI)\).
	We show that the Equations~\eqref{eq:DelJBarHolomCoord} is equivalent to \(\overline{D}\Phi^\#Z^b=0\), where \(\overline{D}=\partial_{\overline{\theta}} + \overline{\theta}\partial_{\overline{z}}\):
	\begin{equation}
		\begin{split}
			\overline{D}\Phi^\#Z^b
			&= \frac12\left(D_3 + \ic D_4\right)\Phi^\#\left(r^b + \ic s^b\right) \\
			&= \frac12\left(D_3\Phi^\#r^b - D_4\Phi^\#s^b\right) + \ic\left(D_4\Phi^\#r^b + D_3\Phi^\#s^b\right)
		\end{split}
	\end{equation}
	Any smooth map \(\Phi\colon \C^{1|1}\to \C^n\) can be expanded with respect to the holomorphic coordinates \((z, \theta)\) on \(\C^{1|1}\) and \(Z^b\) on \(\C^n\) as
	\begin{equation}
		\Phi^\#Z^b = f^b(z, \overline{z}) + \theta g^b(z, \overline{z}) + \overline{\theta} h^b(z, \overline{z}) + \theta\overline{\theta} k^b(z, \overline{z}),
	\end{equation}
	where the functions \(f^b\) and \(k^b\) are even smooth functions of \(z\), \(\overline{z}\), whereas \(g^b\) and \(h^b\) are odd smooth functions of \(z\), \(\overline{z}\).
	Notice that all coefficient functions may depend on additional, possibility odd parameters from the base \(B\) which is suppressed in the notation here.
	Now \(\Phi\) is super \(\targetACI\)-holomorphic if \(\overline{D}\Phi^\#Z^b=0\).
	\begin{equation}
		\begin{split}
			\overline{D}\Phi^\#Z^b
			&= \left(\partial_{\overline{\theta}} + \overline{\theta}\partial_{\overline{z}}\right)\Phi^\#Z^b \\
			&= h^c(z, \overline{z}) - \theta k^b(z, \overline{z}) + \overline{\theta}\partial_{\overline{z}}f^b(z, \overline{z}) - \theta\overline{\theta}\partial_{\overline{z}}g^b(z, \overline{z})
		\end{split}
	\end{equation}
	Consequently, any map \(\Phi\) is super \(\targetACI\)-holomorphic if and only if it is holomorphic, that is,
	\begin{equation}
		\Phi^\#Z^b = f^b(z) + \theta g^b(z).
	\end{equation}
\end{ex}

As all coordinate charts of a Kähler manifold are of the form \(\C^n\), we have:
\begin{lemma}\label{lemma:KählerJholomorphicCurvesAreHolomorphic}
	Let \(\Phi\colon M\to N\) be a map from a super Riemann surface into a Kähler manifold~\((N, \omega, \targetACI)\).
	Then, \(\Phi\) is a super \(\targetACI\)-holomorphic curve if and only if \(\Phi\) is a holomorphic map.
\end{lemma}

\begin{rem}
	In a series of papers~\cites{G-HSCSSM}{G-THSC}{G-CHSC} a different notion of super \(\targetACI\)-holomorphic curves has been proposed.
	Despite the fact that the equations defining a super \(\targetACI\)-holomorphic curve in~\cite{G-HSCSSM} have some similarities with those in Corollary~\ref{cor:SJCComponentEquations}, the theory developed there is quite different from ours.
	First, in those papers a super \(\targetACI\)-holomorphic curve is required to satisfy
	\begin{equation}
		\differential{\Phi} + \targetACI\differential{\Phi}\ACI = 0
	\end{equation}
	for the whole of \(\tangent{M}\) and not only for \(\cD\subset\tangent{M}\), see~\cite[Definition~4.18]{G-HSCSSM}.
	Consequently, the resulting super \(\targetACI\)-holomorphic curves are shown to be critical points of an action that differs substantially from the superconformal action, see~\cite[Theorem~3.2]{G-HSCSSM}.
	In the work by~\citeauthor{G-HSCSSM} the resulting moduli space is a classical manifold with no odd directions.
\end{rem}
\begin{rem}
	In~\cite{AG-MSMSM} the moduli space of stable maps from a Riemann surface to a complex superscheme has been constructed as a finite-dimensional superstack.
	This work differs in several aspects from ours.
	First, in their work, the domain is a classical Riemann surface and the target is a superdomain, whereas here the domain is a super Riemann surface and the target a classical almost Kähler manifold.
	However, the additional structure of a super Riemann surface is necessary for the spinorial part of the map as we will see in the next section.
	Also the relation to the superconformal action depends crucially on the super Riemann surface structure.
	While the concept of Gromov--Witten invariants has been extended to supermanifolds in~\cite{AG-MSMSM}, our work might in contrast rather lead to finer invariants of classical almost Kähler manifolds.
	Finally, their work uses an algebro-geometric approach for complex superschemes as targets, whereas we use tools from differential geometry for almost Kähler targets.
\end{rem}
\begin{rem}
	Several generalizations of super \(\targetACI\)-holomorphic curves are possible:
	If the target \(N\) is only almost complex but not symplectic, the results of Section~\ref{SSec:DJBarComponentFields} still hold.
	The Energy Identities of Section~\ref{SSec:DJBarEnergyIdentities} and Section~\ref{Sec:ModuliSpace} fail.
	If the target \(N\) is a symplectic supermanifold the results of Section~\ref{Sec:DJBar} still hold, but the moduli space in Chapter~\ref{Sec:ModuliSpace} would be bigger.
	Section~\ref{Sec:SpaceOfMaps} depends on the properties of the exponential maps, possibly allowing a generalization to Riemannian supermanifolds.
\end{rem}

\subsection{Energy Identities}\label{SSec:DJBarEnergyIdentities}
In this Section, we discuss the relationship between super \(\targetACI\)-holomorphic curves and critical points of the superconformal action.
Recall, that classical \(\targetACI\)-holomorphic curves are absolute minimizers of the harmonic action.
This is seen, using the Energy Identities, compare~\cite[Lemma~2.2.1]{McDS-JHCST}.
The Energy Identities generalize easily to super \(\targetACI\)-holomorphic curves.

\begin{prop}[Energy Identities]
	Let \(m\) be a metric on \(M\) compatible with the super Riemann surface structure and let \(F_A\) be an \(\UGL(1)\)-frame on \(M\).
	Then for any map \(\Phi\colon M\to N\) to a Riemannian manifold \((N,n)\)
	\begin{equation}
		A(\Phi)
		= \int_M \norm{\left.\differential\Phi\right|_{\cD}}^2_{\dual{m}|_\cD\otimes \Phi^*n}[\d{vol}_m]
		= \int_{M} 2\norm{\DJBar\Phi}^2 - m^{\alpha\beta}\Phi^*n\left(\targetACI\tensor{\ACI}{_\alpha^\mu}F_\mu\Phi, F_\beta\Phi\right) [\d{vol}_m].
	\end{equation}
\end{prop}
\begin{proof}
	\begin{equation}
		\begin{split}
			2\norm{\DJBar\Phi}^2
			&= \frac12\varepsilon^{\alpha\beta}\Phi^*n\left(F_\alpha + \targetACI\left(\tensor{\ACI}{_\alpha^\mu}F_\mu\Phi\right), F_\beta + \targetACI\left(\tensor{\ACI}{_\beta^\nu}F_\nu\Phi\right)\right) \\
			&= \varepsilon^{\alpha\beta}\Phi^*n\left(F_\alpha\Phi, F_\beta\Phi\right) + \varepsilon^{\alpha\beta}\Phi^*n\left(\targetACI\tensor{\ACI}{_\alpha^\mu}F_\mu\Phi, F_\beta\Phi\right)
			\qedhere
		\end{split}
	\end{equation}
\end{proof}
In the classical case of a map from a Riemann surfaces to a symplectic manifold with compatible almost complex structure \(\targetACI\), the second summand is the pullback of the symplectic form along~\(\Phi\).
Consequently, the summand has a homological interpretation and the integral is constant in the homology class of~\(\Phi\).
This does fail in the super case.
Nevertheless, super \(\targetACI\)-holomorphic maps are critical for the superconformal action:
\begin{prop}\label{prop:SuperJholomorphicCurvesAreDHarmonic}
	Let \(m\) be a superconformal metric on \(M\), see~\cite[Chapter 9.6]{EK-SGSRSSCA} and \(\DLaplace\) the resulting \(\cD\)-Laplace operator.
	Then, for all \(m\)-orthonormal frames \(F_A\),
	\begin{equation}
		\begin{split}
			\DLaplace\Phi
			={}& \varepsilon^{\alpha\beta}\left(\nabla_{F_\alpha} + \Div_m F_\alpha - \targetACI\tensor{\ACI}{_\alpha^\sigma}\left(\nabla_{F_\sigma} + \Div_m{F_\sigma}\right)\right)\left<F_\beta, \DJBar\Phi\right> \\
				&+ \varepsilon^{\alpha\beta}\tensor{\ACI}{_\alpha^\sigma} \left<\left<F_\sigma, \DJBar\right>, \Phi^*\nabla\targetACI\right> F_\beta\Phi.
		\end{split}
	\end{equation}
	In particular, super \(\targetACI\)-holomorphic maps into such manifolds satisfy \(\DLaplace\Phi=0\).
\end{prop}
\begin{proof}
	Straightforward calculation yields
	\begin{equation}
		\begin{split}
			\MoveEqLeftR
			2\varepsilon^{\alpha\beta}\left(\nabla_{F_\alpha} + \Div_m F_\alpha - \targetACI\tensor{\ACI}{_\alpha^\sigma}\left(\nabla_{F_\sigma} + \Div_m{F_\sigma}\right)\right)\left<F_\beta, \DJBar\Phi\right> \\
			={}& \varepsilon^{\alpha\beta}\left(\nabla_{F_\alpha} + \Div_m F_\alpha - \targetACI\tensor{\ACI}{_\alpha^\sigma}\left(\nabla_{F_\sigma} + \Div_m{F_\sigma}\right)\right)\left(F_\beta\Phi + \tensor{\ACI}{_\beta^\tau}\targetACI F_\tau\Phi\right) \\
			={}& \varepsilon^{\alpha\beta}\left(\nabla_{F_\alpha}F_\beta\Phi + \left(\Div_m F_\alpha\right)F_\beta\Phi - \targetACI\tensor{\ACI}{_\alpha^\sigma}\nabla_{F_\sigma}F_\beta\Phi - \targetACI\tensor{\ACI}{_\alpha^\sigma}\left(\Div_m F_\sigma\right) F_\beta\Phi \right.\\
				&+ \tensor{\ACI}{_\beta^\tau}\left(\nabla_{F_\alpha}\targetACI\right)F_\tau\Phi + \tensor{\ACI}{_\beta^\tau}\targetACI\nabla_{F_\alpha}F_\tau\Phi + \tensor{\ACI}{_\beta^\tau}\targetACI\left(\Div_m F_\alpha\right)F_\tau\Phi \\
				&- \left.\targetACI\tensor{\ACI}{_\alpha^\sigma}\tensor{\ACI}{_\beta^\tau}\left(\nabla_{F_\sigma} \targetACI\right)F_\tau\Phi + \tensor{\ACI}{_\alpha^\sigma}\tensor{\ACI}{_\beta^\tau}\nabla_{F_\sigma}F_\tau\Phi + \tensor{\ACI}{_\alpha^\sigma}\tensor{\ACI}{_\beta^\tau}\left(\Div_m F_\sigma\right)F_\tau\Phi\right) \\
			={}& 2\varepsilon^{\alpha\beta}\left(\nabla_{F_\alpha}F_\beta\Phi + \left(\Div_m F_\alpha\right)F_\beta\Phi\right) - 2\targetACI\varepsilon^{\alpha\beta}\tensor{\ACI}{_\alpha^\sigma}\left(\nabla_{F_\sigma}F_\beta\Phi + \left(\Div_m F_\sigma\right) F_\beta\Phi\right) \\
				&- \varepsilon^{\alpha\beta}\tensor{\ACI}{_\alpha^\sigma}\left(\nabla_{F_\sigma}\targetACI\right)\left(F_\beta\Phi - \targetACI\tensor{\ACI}{_\beta^\tau}F_\tau\Phi\right).
		\end{split}
	\end{equation}
	The connection on \(N\) is torsionfree and by~\cite[Lemma~10.2.13]{EK-SGSRSSCA}:
	\begin{equation}
		\begin{split}
			0
			&= \varepsilon^{\alpha\beta}\tensor{\ACI}{_\alpha^\sigma}\Phi^*T\left(F_\sigma\Phi, F_\beta\Phi\right) \\
			&= \varepsilon^{\alpha\beta}\tensor{\ACI}{_\alpha^\sigma}\left(\nabla_{F_\sigma}F_\beta\Phi + \nabla_{F_\beta}F_\sigma\Phi - [F_\sigma, F_\beta]\Phi\right) \\
			&= 2\varepsilon^{\alpha\beta}\tensor{\ACI}{_\alpha^\sigma}\left(\nabla_{F_\sigma}F_\beta\Phi + \left(\Div_m F_\sigma\right)F_\beta\Phi\right).
		\end{split}
	\end{equation}
	By~\cite[Prop.~IX.4.2 and Theorem~IX.4.3]{KN-FDG}, the condition \(\left(\nabla_X \targetACI\right)\targetACI Y = \left(\nabla_{\targetACI X}\targetACI\right) Y\) is equivalent to \(n(\targetACI\cdot, \cdot)\) being symplectic in the case of ordinary manifolds.
	Hence,
	\begin{equation}\label{eq:NablaTargetACI}
		\varepsilon^{\alpha\beta}\tensor{\ACI}{_\alpha^\sigma}\left(\nabla_{F_\sigma}\targetACI\right)\left(F_\beta\Phi - \targetACI\tensor{\ACI}{_\beta^\tau}F_\tau\Phi\right)
		= 2\varepsilon^{\alpha\beta}\tensor{\ACI}{_\alpha^\sigma} \left<\left<F_\sigma, \DJBar\Phi\right>, \Phi^*\nabla\targetACI\right> F_\beta\Phi
	\end{equation}
	and the result follows.
\end{proof}
\begin{rem}
	Let \(\tilde{m}\) be the superconformal metric arising from \(m\) by a superconformal rescaling by \(\lambda\) and a superconformal change of splitting.
	In that case the \(\cD\)-Laplace operator obtained from \(\tilde{m}\) is given by
	\begin{equation}
		\widetilde{\Delta}^{\cD} = \lambda^{-2}\DLaplace.
	\end{equation}
	Consequently, fact that super \(\targetACI\)-holomorphic curves are satisfy \(\DLaplace\Phi=0\) is independent of the chosen metric in the superconformal class.
	This is consistent with the fact that the operator \(\DJBar\) is defined independently of the superconformal metric \(m\).
\end{rem}

\subsection{Component Fields}\label{SSec:DJBarComponentFields}
Let \(i\colon \Smooth{M}\to M\) be an underlying even manifold of the super Riemann surface \(M\).
In~\cites{JKT-SRSMG}{EK-SGSRSSCA} it was shown that the map \(\Phi\colon M\to N\) is completely determined by its component fields
\begin{align}
	\varphi &= \Phi\circ i\colon \Smooth{M}\to N, \\
	\psi &= i^*\left.\differential{\Phi}\right|_{\cD}\in\VSec{\dual{S}\otimes\varphi^*\tangent{N}}, \\
	F &= -\frac12i^*\DLaplace\Phi\in\VSec{\varphi^*\tangent{N}}.
\end{align}
The proof reduces to choose good local coordinates \((x^a, \eta^\alpha)\) on \(M\) such that
\begin{equation}
	\Phi = \varphi(x) + \eta^\mu\tensor[_\mu]{\psi}{} + \eta^3\eta^4 F.
\end{equation}

The goal of this subsection is to obtain differential equations for the fields \(\varphi\), \(\psi\) and~\(F\) that are equivalent to \(\DJBar\Phi=0\).
In the remainder of this article, we will use those differential equations to construct the moduli space of super \(\targetACI\)-holomorphic curves.

\begin{lemma}\label{lemma:DefinitionComponentEquations}
	We define the component fields of \(A\in\VSec{\dual{\cD}\otimes\Phi^*\tangent{N}}^{0,1}\) as
	\begin{align}
		&i^*A \in\VSec{\dual{S}\otimes\varphi^*\tangent{N}}^{0,1}, \\
		&\frac12\Tr_{\dual{g}_S}i^*\nablabar^{\cotangent{M}\otimes\Phi^*\tangent{N}}A \in\VSec{\varphi^*\tangent{N}}, \\
		&\frac12\Tr_{\dual{g}_S}\left(\id_{\dual{S}}\otimes\gamma\otimes\id_{\varphi^*\tangent{N}}\right)i^*\nablabar^{\cotangent{M}\otimes\Phi^*\tangent{N}}A \in\VSec{\cotangent{\Smooth{M}}\otimes\varphi^*\tangent{N}}^{0,1}, \\
		&-\frac12i^*\DBarLaplace A \in \VSec{\dual{S}\otimes \varphi^*\tangent{N}}^{0,1}.
	\end{align}
	Here, \(\nablabar^{\cotangent{M}\otimes\Phi^*\tangent{N}}=\nabla^{\tangent{M}}\otimes\nablabar\) is the connection induced by \(\nabla^{\cotangent{M}}\) and the almost complex covariant derivative \(\nablabar\) defined in Equation~\eqref{eq:DefinitionNablaBar}, whereas \(\DBarLaplace\) is the \(\cD\)-Laplace operator obtained from the connection \(\nablabar^{\cotangent{M}\otimes\Phi^*\tangent{N}}\):
	\begin{equation}
		\DBarLaplace A
		= \varepsilon^{\alpha\beta}\left(\nablabar^{\cotangent{M}\otimes\Phi^*\tangent{N}}_{F_\alpha}\nablabar^{\cotangent{M}\otimes\Phi^*\tangent{N}}_{F_\beta} + \left(\Div_m F_\alpha\right)\nablabar^{\cotangent{M}\otimes\Phi^*\tangent{N}}_{F_\beta}\right)A.
	\end{equation}
	The map
	\begin{multline}
		\VSec{\dual{\cD}\otimes\Phi^*\tangent{N}}^{0,1} \to \\
		\VSec{\dual{S}\otimes\varphi^*\tangent{N}}^{0,1}\oplus\VSec{\varphi^*\tangent{N}}\oplus\VSec{\cotangent{\Smooth{M}}\otimes\varphi^*\tangent{N}}^{0,1}\oplus\VSec{\dual{S}\otimes\varphi^*\tangent{N}}^{0,1}
	\end{multline}
	that maps a section to its components fields is an isomorphism of \(\cO_B\)-modules.
\end{lemma}
\begin{proof}
	It is an easy variant of~\cite[Proposition~12.2.2]{EK-SGSRSSCA} that for sections \(A\in\VSec{\dual{\cD}\otimes\Phi^*\tangent{N}}\) the component fields
	\begin{align}
		&i^*A \in\VSec{\dual{S}\otimes\Phi^*\tangent{N}}, \\
		&s^\alpha\otimes i^*\nablabar^{\cotangent{M}\otimes\Phi^*\tangent{N}}_{F_\alpha}A \in\VSec{\dual{S}\otimes\dual{S}\otimes\varphi^*\tangent{N}}, \\
		&-\frac12i^*\DBarLaplace A \in \VSec{\dual{S}\otimes\varphi^*\tangent{N}},
	\end{align}
	yield an isomorphism of \(\cO_B\)-modules
	\begin{equation}
		\VSec{\dual{\cD}\otimes\Phi^*\tangent{N}} \cong
		\VSec{\dual{S}\otimes\varphi^*\tangent{N}}\oplus\VSec{\dual{S}\otimes\dual{S}\otimes\varphi^*\tangent{N}}\oplus\VSec{\dual{S}\otimes\varphi^*\tangent{N}}.
	\end{equation}
	The advantage of using \(\nablabar\) instead of \(\nabla\) is that it is compatible with the almost complex structure \(\targetACI\).
	It follows that for \(A\in\VSec{\dual{\cD}\otimes\Phi^*\tangent{N}}^{0,1}\) the fields \(i^*A\) and \(-\frac12i^*\DBarLaplace A\) are in \(\VSec{\dual{S}\otimes\varphi^*\tangent{N}}^{0,1}\).
	For
	\begin{equation}
		W = \Set{C\in\VSec{\dual{S}\otimes\dual{S}\otimes\varphi^*\tangent{N}}\given \left(1 - \id_{\dual{S}}\otimes\ACI\otimes\targetACI\right)C=0}
	\end{equation}
	notice that there is an isomorphism
	\begin{equation}
		\begin{split}
			\VSec{\varphi^*\tangent{N}}\oplus\VSec{\cotangent{\Smooth{M}}\otimes\varphi^*\tangent{N}}^{0,1} &\to W \\
			(X, f^k\otimes Y_k) &\mapsto s^\alpha\otimes s^\beta\left(\delta_{\alpha\beta}X + \Gamma_{\alpha\beta}^k Y_k + \varepsilon_{\alpha\beta}\targetACI X\right)
			\qedhere
		\end{split}
	\end{equation}
\end{proof}
We will be slightly imprecise and denote the almost complex structures \(\Phi^*\targetACI\) and \(\varphi^*\targetACI\) on \(\Phi^*\tangent{N}\) and \(\varphi^*\tangent{N}\) also by \(\targetACI\).
Similarly, we also denote the pullback connections on \(\Phi^*\tangent{N}\) and \(\varphi^*\tangent{N}\) induced from \(\nabla\) and \(\nablabar\) by the same symbol.
In order to express the component fields of \(\DJBar\Phi\) we also need the component fields of \(\Phi^*\targetACI\):
\begin{align}
	\mathfrak{j} &= s^\mu\otimes \tensor[_\mu]{\mathfrak{j}}{} = i^*\left.\left(\nabla\Phi^*\targetACI\right)\right|_{\cD} \in\VSec{\dual{S}\otimes\End{\varphi^*\tangent{N}}}, \\
		\tensor[_{34}]{\mathfrak{j}}{} &= -\frac12i^*\DLaplace\Phi^*\targetACI \in\VSec{\End{\varphi^*\tangent{N}}}
\end{align}
As for the map \(\Phi\), this allows to write in good coordinates \((x^a, \eta^\alpha)\) on \(M\)
\begin{equation}
	\Phi^*\targetACI = \varphi^*\targetACI + \eta^\mu\tensor[_\mu]{\mathfrak{j}}{} + \eta^3\eta^4\tensor[_{34}]{\mathfrak{j}}{}.
\end{equation}
Notice that the \(\Phi\)-dependence of \(\mathfrak{j}\) can be rewritten as \(\mathfrak{j}=\left<\psi, \varphi^*\nabla\targetACI\right>\).
By the properties of almost complex structures, we have for all spinors \(s\in\VSec{S}\) that \(\targetACI\left<s,\mathfrak{j}\right> = - \left<s,\mathfrak{j}\right>\targetACI\) and \(\targetACI\tensor[_{34}]{\mathfrak{j}}{} = \left(\Tr_{\dual{g}_S}\mathfrak{j}\mathfrak{j}\right) -\tensor[_{34}]{\mathfrak{j}}{}\targetACI\).
In the case that \(N\) is Kähler, both \(\mathfrak{j}\) and \(\tensor[_{34}]{\mathfrak{j}}{}\) vanish.

\begin{prop}\label{prop:ComponentsOfDJBar}
	The component fields of \(\DJBar\Phi\) in the sense of Lemma~\ref{lemma:DefinitionComponentEquations} are given by
	\begin{align}
		\MoveEqLeftR
		i^*\DJBar\Phi
		= \frac12\left(1+\ACI\otimes\targetACI\right)\psi \\
		\MoveEqLeftR
		\frac12\Tr_{\dual{g}_S}i^*\nablabar^{\cotangent{M}\otimes\Phi^*\tangent{N}}\DJBar\Phi
		= \frac14F + \frac18\Tr_{\dual{g}_S}\left(\id_{\dual{S}}\otimes\mathfrak{j}\targetACI + \ACI\otimes\mathfrak{j}\right)\psi \\
		\begin{split}
			\MoveEqLeftR
			\frac12\Tr_{\dual{g}_S}\left(\id_{\dual{S}}\otimes\gamma\otimes\id_{\varphi^*\tangent{N}}\right)i^*\nablabar^{\cotangent{M}\otimes\Phi^*\tangent{N}}\DJBar\Phi \\
			={}& -\frac12\left(1+\ACI\otimes\targetACI\right)\left(\differential{\varphi} + \left<\chi, \psi\right> - \frac14\Tr_{\dual{g}_S}\left(\gamma\otimes\mathfrak{j}\targetACI\right)\psi\right)
		\end{split} \\
		\begin{split}
			\MoveEqLeftR
			-\frac12i^*\DBarLaplace\DJBar\Phi
			= -\frac12\left(1 + \ACI\otimes\targetACI\right)\left(\Dirac\psi - 2\left<\vee Q\chi,\differential{\varphi}\right> + \norm{Q\chi}^2\psi + \delta_\gamma\chi\otimes F \right.\\
				&- \frac16SR^N(\psi) + \frac12\left(\id_{\dual{S}}\otimes\left(\targetACI\tensor[_{34}]{\mathfrak{j}}{} -\frac14\Tr_{\dual{g}_S}\mathfrak{j}\mathfrak{j}\right)\right)\psi \\
				&+ \left.\frac12\Tr_g\left<\left(\gamma\otimes\id_{\varphi^*\tangent{N}}\right)\psi, \targetACI\varphi^*\nabla\targetACI\right>\left(\differential{\varphi} + \left<\chi, \psi\right>\right) - \frac12\mathfrak{j}\targetACI F\right)
		\end{split}
	\end{align}
	Here \(\Dirac\) is the twisted Dirac-operator on \(\dual{S}\otimes\phi^*\tangent{N}\).
	In the frames \(f_k\) of \(\tangent{\Smooth{M}}\), \(s^\alpha\) of~\(\dual{S}\) we write for the lifted Levi-Civita connection \(\nabla^{\dual{S}}_{f_k}s^\alpha = -\frac12\omega^{LC}_k s^\beta\tensor{\ACI}{_\beta^\alpha}\) and \(\psi = s^\alpha\otimes \tensor[_\alpha]{\psi}{}\).
	Then,
	\begin{equation}
		\Dirac\psi
		= \gamma^k\nabla^{\dual{S}\otimes\varphi^*\tangent{N}}_{f^k}\psi
		= s^\beta \otimes \left(\frac12\omega^{LC}_k\tensor{\gamma}{^k_\beta^\mu}\tensor{\ACI}{_\mu^\alpha}\tensor[_\alpha]{\psi}{} - \tensor{\gamma}{^k_\beta^\alpha}\nabla^{\varphi^*\tangent{N}}_k\tensor[_\alpha]{\psi}{}\right).
	\end{equation}
	The curvature term \(SR^N(\psi)\) is contraction of the curvature tensor \(R^N\) of \(N\) with \(\psi\) of order three:
	\begin{equation}
		SR^N(\psi)
		= s^\alpha\otimes\left(\tensor{g}{_S^\mu^\nu}R^N(\tensor[_\alpha]{\psi}{}, \tensor[_\mu]{\psi}{})\tensor[_\nu]{\psi}{}\right).
	\end{equation}
\end{prop}
\begin{proof}
	We choose a local Wess--Zumino-frame \(F_A\) on \(M\) and write \(\DJBar\Phi = F^\alpha\otimes \left<F_\alpha, \DJBar\Phi\right>\).
	Then
	\begin{equation}
		\begin{split}
			i^*\DJBar\Phi
			&= i^*F^\alpha\otimes i^*\left<F_\alpha, \DJBar\Phi\right>
			= \frac12s^\alpha\otimes i^*\left(F_\alpha\Phi + \targetACI\left(\ACI F_\alpha\Phi\right)\right) \\
			&= \frac12s^\alpha\otimes \tensor[_\alpha]{\psi}{} + \frac12s^\alpha\tensor{\ACI}{_\alpha^\beta}\otimes \targetACI\tensor[_\beta]{\psi}{}
			= \frac12\left( 1+\ACI\otimes\targetACI \right)\psi.
		\end{split}
	\end{equation}

	For the component fields of first order, we use that for Wess--Zumino frames \(i^*\nabla^{\cotangent{M}}_{F_\mu}F^\alpha=0\), and furthermore,~\cite[Lemma~13.6.2]{EK-SGSRSSCA}.
	Then
	\begin{equation}
		\begin{split}
			\MoveEqLeftR
			s^\mu\otimes i^*\nablabar^{\cotangent{M}\otimes\Phi^*\tangent{N}}_{F_\mu}\DJBar\Phi \\
			={}& s^\mu\otimes s^\alpha\otimes i^*\nablabar_{F_\mu}\left<F_\alpha, \DJBar\Phi\right> \\
			={}& \frac12s^\mu\otimes s^\alpha\otimes i^*\nablabar_{F_\mu}\left(F_\alpha\Phi + \targetACI\left(\tensor{\ACI}{_\alpha^\beta}F_\beta\Phi\right)\right) \\
			={}& \frac12 s^\mu\otimes s^\alpha\otimes \left(\delta_\alpha^\beta + \tensor{\ACI}{_\alpha^\beta}\targetACI\right)\left(i^*\nabla_{F_\mu}F_\beta\Phi - \frac12i^*\targetACI \left(\nabla_{F_\mu}\targetACI\right)F_\beta\Phi\right) \\
			={}& \frac12 s^\mu\otimes s^\alpha\otimes \left(\delta_\alpha^\beta + \tensor{\ACI}{_\alpha^\beta}\targetACI\right)\left(\Gamma_{\mu\beta}^k\left(f_k\varphi + \tensor{\chi}{_k^\kappa}\tensor[_\kappa]{\psi}{}\right) - \varepsilon_{\mu\beta}F - \frac12i^*\targetACI \left(\nabla_{F_\mu}\targetACI\right)\tensor[_\beta]{\psi}{}\right).
		\end{split}
	\end{equation}
	Consequently,
	\begin{equation}
		\begin{split}
			\frac12\Tr_{\dual{g}_S}i^*\nablabar^{\cotangent{M}\otimes\Phi^*\tangent{N}}\DJBar\Phi
			={}& \frac14\varepsilon^{\alpha\mu}\left(\delta_\alpha^\beta + \tensor{\ACI}{_\alpha^\beta}\targetACI\right)\left(- \varepsilon_{\mu\beta}F - \frac12i^*\targetACI \left(\nabla_{F_\mu}\targetACI\right)\tensor[_\beta]{\psi}{}\right) \\
			={}& \frac14\left(F - \frac12i^*\varepsilon^{\alpha\mu}\left(\delta_\alpha^\beta + \tensor{\ACI}{_\alpha^\beta}\targetACI\right)\targetACI \left(\nabla_{F_\mu}\targetACI\right)\tensor[_\beta]{\psi}{}\right) \\
			={}& \frac14F + \frac18\Tr_{\dual{g}_S}\left(\id_{\dual{S}}\otimes\mathfrak{j}\targetACI + \ACI\otimes\mathfrak{j}\right)\psi
		\end{split}
	\end{equation}
	\begin{equation}
		\begin{split}
			\MoveEqLeftR
			\frac12\Tr_{\dual{g}_S}\left(\id_{\dual{S}}\otimes\gamma\otimes\id_{\varphi^*\tangent{N}}\right)i^*\nablabar^{\cotangent{M}\otimes\Phi^*\tangent{N}}\DJBar\Phi \\
			={}& \frac14 f^l\otimes \varepsilon^{\nu\mu}\tensor{\gamma}{_l_\nu^\alpha}\left(\delta_\alpha^\beta + \tensor{\ACI}{_\alpha^\beta}\targetACI\right)\left(\Gamma_{\mu\beta}^k\left(f_k\varphi + \tensor{\chi}{_k^\kappa}\tensor[_\kappa]{\psi}{}\right) - \frac12i^*\targetACI \left(\nabla_{F_\mu}\targetACI\right)\tensor[_\beta]{\psi}{}\right) \\
			={}& \frac12 f^l\otimes \left(\delta_l^k + \tensor{\ACI}{_l^k}\targetACI\right)\left(\left(f_k\varphi + \tensor{\chi}{_k^\kappa}\tensor[_\kappa]{\psi}{}\right) - \frac14i^*\varepsilon^{\nu\mu}\tensor{\gamma}{_k_\nu^\alpha}\targetACI \left(\nabla_{F_\mu}\targetACI\right)\tensor[_\alpha]{\psi}{}\right) \\
			={}& -\frac12\left(1+\ACI\otimes\targetACI\right)\left(\differential{\varphi} + \left<\chi, \psi\right> - \frac14\Tr_{\dual{g}_S}\left(\gamma\otimes\mathfrak{j}\targetACI\right)\psi\right)
		\end{split}
	\end{equation}

	For the second order component of \(\DJBar\Phi\), we use again~\cite[Lemma~13.6.2]{EK-SGSRSSCA} and \(i^*\left(\Div_m F_\alpha\right)=0\), see~\cite[Lemma~13.7.2]{EK-SGSRSSCA}.
	The first step is to express \(i^*\DBarLaplace\DJBar\Phi\) in terms of \(i^*\DLaplace\DJBar\Phi\) and component fields of \(\targetACI\).
	\begin{equation}
		\begin{split}
			\MoveEqLeftR
			-\frac12\DBarLaplace\DJBar\Phi
			= -\frac12\varepsilon^{\mu\nu}\left(\nablabar^{\cotangent{M}\otimes\Phi^*\tangent{N}}_{F_\mu}\nablabar^{\cotangent{M}\otimes\Phi^*\tangent{N}}_{F_\nu} + \left(\Div_m F_\mu\right)\nablabar^{\cotangent{M}\otimes\Phi^*\tangent{N}}_{F_\nu}\right)\DJBar\Phi \\
			={}& -\frac14 s^\alpha\otimes \left(\delta_\alpha^\beta + \tensor{\ACI}{_\alpha^\beta}\targetACI\right)\varepsilon^{\mu\nu}i^*\left(\nabla_{F_\mu} - \frac12\targetACI\left(\nabla_{F_\mu}\targetACI\right)\right)\left(\nabla_{F_\nu} - \frac12\targetACI\left(\nabla_{F_\nu}\targetACI\right)\right)F_\beta\Phi \\
			={}& -\frac14 s^\alpha\otimes \left(\delta_\alpha^\beta + \tensor{\ACI}{_\alpha^\beta}\targetACI\right)\varepsilon^{\mu\nu}i^*\left(\nabla_{F_\mu}\nabla_{F_\nu} - \frac14\left(\nabla_{F_\mu}\targetACI\right)\left(\nabla_{F_\nu}\targetACI\right) - \frac12\targetACI\left(\nabla_{F_\mu}\nabla_{F_\nu}\targetACI\right) \right.\\
				&- \left.\vphantom{\frac12}\targetACI\left(\nabla_{F_\mu}\targetACI\right)\nabla_{F_\nu}\right)F_\beta\Phi \\
			={}& -\frac12\left(1 + \ACI\otimes\targetACI\right)\left(\Dirac\psi - 2\left<\vee Q\chi,\differential{\varphi}\right> + \norm{Q\chi}^2\psi + \delta_\gamma\chi\otimes F - \frac16SR^N(\psi) \right.\\
				&+ \frac12\left(\id_{\dual{S}}\otimes\left(\targetACI\tensor[_{34}]{\mathfrak{j}}{} -\frac14\Tr_{\dual{g}_S}\mathfrak{j}\mathfrak{j}\right)\right)\psi \\
				&+ \left.\frac12\Tr_g\left<\left(\gamma\otimes\id_{\varphi^*\tangent{N}}\right)\psi, \targetACI\varphi^*\nabla\targetACI\right>\left(\differential{\varphi} + \left<\chi, \psi\right>\right) - \frac12\mathfrak{j}\targetACI F\right)
		\end{split}
	\end{equation}
	This shows the claim.
\end{proof}

We will see that the component expressions simplify considerably if we assume \(\DJBar\Phi=0\).
Recall that the map \(\varphi\colon \Smooth{M}\to N\) is called a \(\targetACI\)-holomorphic curve if
\begin{equation}
	\DelJBar\varphi = \frac12\left(1+\ACI\otimes\targetACI\right)\differential{\varphi}
\end{equation}
vanishes.

\begin{cor}\label{cor:SJCComponentEquations}
	The equation \(\DJBar\Phi=0\) is equivalent to
	\begin{align}
		0 &= \left(1+\ACI\otimes\targetACI\right)\psi \\
		0 &= F, \\
		\begin{split}
			0
			&= \frac12\left(1 + \ACI\otimes\targetACI\right)\left(\differential{\varphi} + \left<\chi, \psi\right> +\frac14\Tr_{\dual{g}_S}\left(\gamma\otimes\mathfrak{j}\targetACI\right)\psi\right), \\
			&= \DelJBar\varphi + \left<Q\chi, \psi\right> + \frac14\Tr_{\dual{g}_S}\left(\gamma\otimes\mathfrak{j}\targetACI\right)\psi.
		\end{split} \\
		\begin{split}
			0
			&= \frac12\left(1 + \ACI\otimes\targetACI\right)\left(\Dirac\psi - 2\left<\vee Q\chi,\differential{\varphi}\right> - \frac13SR^N(\psi)\right) \\
			&= \Dirac\psi - 2\left<\vee Q\chi, \differential{\varphi}\right> + \norm{Q\chi}^2\psi - \frac13SR^N(\psi)
		\end{split}
	\end{align}
\end{cor}
\begin{proof}
	For the proof we use the component expression of \(\DJBar\Phi\) obtained in Proposition~\ref{prop:ComponentsOfDJBar}, set the components to zero and simplify.
	Obviously the equation \(\frac12\left(1+ \ACI\otimes\targetACI\right)\psi=0\) is equivalent to \(\tensor{\ACI}{_\alpha^\beta}\tensor[_\beta]{\psi}{}=\targetACI\tensor[_\alpha]{\psi}{}\) or \(\psi\in\VSec{\dual{S}\otimes_\C\varphi^*\tangent{N}}\).

	Inserting \(\left(1 + \ACI\otimes\targetACI\right)\psi=0\) into the second component we obtain:
	\begin{equation}
		\begin{split}
			0
			&= \frac14F + \frac18\Tr_{\dual{g}_S}\left(\id_{\dual{S}}\otimes\mathfrak{j}\targetACI + \ACI\otimes\mathfrak{j}\right)\psi \\
			&= \frac14F - \frac18\varepsilon^{\mu\nu}\targetACI\left<\tensor[_\mu]{\psi}{}, \varphi^*\nabla\targetACI\right>\left(\delta_\nu^\tau - \tensor{\ACI}{_\nu^\tau}\targetACI\right)\tensor[_\tau]{\psi}{} \\
			&= \frac14F - \frac18\varepsilon^{\mu\nu}\targetACI\left<\left(\delta_\mu^\tau + \tensor{\ACI}{_\mu^\tau}\targetACI\right)\tensor[_\tau]{\psi}{}, \varphi^*\nabla\targetACI\right>\tensor[_\nu]{\psi}{}
			= \frac14F
		\end{split}
	\end{equation}

	For the first-order equation in \(\varphi\) note that
	\begin{align}
		\tensor{\ACI}{_k^l}\tensor{\left(Q\chi\right)}{_l^\lambda}
		&= -\tensor{\left(Q\chi\right)}{_k^\kappa}\tensor{\ACI}{_\kappa^\lambda}, &
		\tensor{\ACI}{_k^l}\tensor{\left(P\chi\right)}{_l^\lambda}
		&= \tensor{\left(P\chi\right)}{_k^\kappa}\tensor{\ACI}{_\kappa^\lambda}. &
	\end{align}
	Hence,
	\begin{equation}
		\begin{split}
			0
			={}& \frac12\left(1+\ACI\otimes\targetACI\right)\left(\differential{\varphi} + \left<\chi, \psi\right> + \frac14\Tr_{\dual{g}_S}\left(\gamma\otimes\mathfrak{j}\targetACI\right)\psi\right) \\
			={}& \DelJBar\varphi + \left<Q\chi, \psi\right> + \frac14\Tr_{\dual{g}_S}\left(\gamma\otimes\mathfrak{j}\targetACI\right)\psi.
		\end{split}
	\end{equation}

	The second order component of \(\DJBar\Phi=0\) simplifies considerably using \((1+\ACI\otimes\targetACI)\psi=0\) and \(F=0\):
	\begin{equation}
		\begin{split}
			0
			={}& -\frac12\left(1 + \ACI\otimes\targetACI\right)\left(\Dirac\psi - 2\left<\vee Q\chi,\differential{\varphi}\right> + \norm{Q\chi}^2\psi + \delta_\gamma\chi\otimes F \right.\\
				&- \frac16SR^N(\psi) + \frac12\left(\id_{\dual{S}}\otimes\left(\targetACI\tensor[_{34}]{\mathfrak{j}}{} -\frac14\Tr_{\dual{g}_S}\mathfrak{j}\mathfrak{j}\right)\right)\psi \\
				&+ \left.\frac12\Tr_g\left<\left(\gamma\otimes\id_{\varphi^*\tangent{N}}\right)\psi, \targetACI\varphi^*\nabla\targetACI\right>\left(\differential{\varphi} + \left<\chi, \psi\right>\right) - \frac12\mathfrak{j}\targetACI F\right) \\
			={}& -\frac12\left(1 + \ACI\otimes\targetACI\right)\left(\Dirac\psi - 2\left<\vee Q\chi,\differential{\varphi}\right> - \frac16SR^N(\psi) \right.\\
				&+ \left.\frac12\left(\id_{\dual{S}}\otimes\targetACI\tensor[_{34}]{\mathfrak{j}}{}\right)\psi + \frac12\Tr_g\left<\left(\gamma\otimes\id_{\varphi^*\tangent{N}}\right)\psi, \targetACI\varphi^*\nabla\targetACI\right>\left(\differential{\varphi} + \left<\chi, \psi\right>\right) \right) \\
			={}& -\frac12\left(1 + \ACI\otimes\targetACI\right)\left(\Dirac\psi - 2\left<\vee Q\chi,\differential{\varphi}\right> - \frac13SR^N(\psi)\right)
		\end{split}
	\end{equation}
	Here the last step, showing that the terms involving higher component fields of \(\targetACI\) equal a curvature term, needs further justification.
	The following calculation uses \(\tensor{\ACI}{_\alpha^\beta}\tensor[_\beta]{\psi}{}=\targetACI\tensor[_\alpha]{\psi}{}\), \(F=0\), as well as
	\begin{equation}
		2i^*\left(\nabla_{F_\alpha}\targetACI\right)F_\beta\Phi
		= \Gamma_{\alpha\beta}^k  \varepsilon^{\kappa\tau}\tensor{\gamma}{_k_\tau^\lambda}i^*\left(\nabla_{F_\kappa}\targetACI\right) F_\lambda\Phi.
	\end{equation}
	Then,
	\begin{equation}
		\begin{split}
			\MoveEqLeftR
			\frac12\left(\id_{\dual{S}}\otimes\targetACI\tensor[_{34}]{\mathfrak{j}}{}\right)\psi + \frac12\Tr_g\left<\left(\gamma\otimes\id_{\varphi^*\tangent{N}}\right)\psi, \targetACI\varphi^*\nabla\targetACI\right>\left(\differential{\varphi} + \left<\chi, \psi\right>\right) \\
			={}& -\frac14s^\alpha\otimes i^*\targetACI\varepsilon^{\mu\nu}\left(\left(\nabla_{F_\mu}\nabla_{F_\nu}\targetACI\right)F_\alpha\Phi + 2\left(\nabla_{F_\mu}\targetACI\right)\nabla_{F_\nu}F_\alpha\Phi\right) \\
			={}& -\frac14s^\alpha\otimes i^*\targetACI\varepsilon^{\mu\nu}\left(\nabla_{F_\mu}\left(\nabla_{F_\nu}\targetACI\right)F_\alpha\Phi + \left(\nabla_{F_\mu}\targetACI\right)\nabla_{F_\nu}F_\alpha\Phi\right) \\
			={}& \frac18s^\alpha\otimes \targetACI i^*\left(\tensor{\gamma}{^k_\alpha^\mu}\varepsilon^{\kappa\tau}\tensor{\gamma}{_k_\tau^\lambda}\nabla_{F_\mu}\left(\nabla_{F_\kappa}\targetACI\right)F_\lambda\Phi - 2\varepsilon^{\mu\nu}\left(\nabla_{F_\mu}\targetACI\right)\nabla_{F_\nu}F_\alpha\Phi\right) \\
			={}& \frac18s^\alpha\otimes \targetACI i^*\left(\tensor{\gamma}{^k_\alpha^\mu}\varepsilon^{\kappa\tau}\tensor{\gamma}{_k_\tau^\lambda}\left(\nabla_{F_\mu}\nabla_{F_\kappa}\targetACI\right)F_\lambda\Phi - 2\varepsilon^{\mu\nu}\left(\nabla_{F_\mu}\targetACI\right)\nabla_{F_\nu}F_\alpha\Phi\right) \\
			={}& \frac18s^\alpha\otimes \targetACI i^*\left(\tensor{\gamma}{^k_\alpha^\mu}\varepsilon^{\kappa\tau}\tensor{\gamma}{_k_\tau^\lambda}\left( \left(R(F_\mu, F_\kappa)\targetACI\right)F_\lambda\Phi - \left(\nabla_{F_\kappa}\nabla_{F_\mu}\targetACI\right)F_\lambda\Phi\right) \right.\\
				&- \left.2\varepsilon^{\mu\nu}\left(\nabla_{F_\mu}\targetACI\right)\nabla_{F_\nu}F_\alpha\Phi\right) \\
			={}& \frac18s^\alpha\otimes \targetACI i^*\left(\tensor{\gamma}{^k_\alpha^\mu}\varepsilon^{\kappa\tau}\tensor{\gamma}{_k_\tau^\lambda}\left( \targetACI R(F_\mu, F_\kappa)F_\lambda\Phi - R(F_\mu, F_\kappa)\targetACI F_\lambda\Phi \right.\right.\\
				&- \left.\left.\nabla_{F_\kappa}\left(\nabla_{F_\mu}\targetACI\right)F_\lambda\Phi - \left(\nabla_{F_\mu}\targetACI\right)\nabla_{F_\kappa}F_\lambda\Phi\right) - 2\varepsilon^{\mu\nu}\left(\nabla_{F_\mu}\targetACI\right)\nabla_{F_\nu}F_\alpha\Phi\right) \\
			={}& -\frac1{12}\left(1+\ACI\otimes\targetACI\right)SR^N(\psi)
		\end{split}
	\end{equation}
	The last step uses symmetry properties of the curvature tensor, see Lemma~\ref{lemma:SRIdentities} in the Appendix.
	This completes the proof of the first variant of the Dirac equation.
	We now rewrite the Dirac equation by applying \(1+\ACI\otimes\targetACI\) to all summands.
	For the Dirac operator it holds
	\begin{equation}\label{eq:CommuteIJDirac}
		\begin{split}
			\left(1\pm\ACI\otimes\targetACI\right)\Dirac\psi
			&= \left(1\pm\ACI\otimes\targetACI\right) \left(s^\beta\otimes\left(\frac12\omega^{LC}_k\tensor{\gamma}{^k_\beta^\mu}\tensor{\ACI}{_\mu^\alpha}\tensor[_\alpha]{\psi}{} - \tensor{\gamma}{^k_\beta^\alpha}\nabla^{\varphi^*\tangent{N}}_k\tensor[_\alpha]{\psi}{}\right)\right) \\
			&=  \left(s^\beta\otimes\left(\frac12\omega^{LC}_k\tensor{\gamma}{^k_\beta^\mu}\tensor{\ACI}{_\mu^\alpha}\tensor[_\alpha]{\psi}{} \pm \frac12\omega^{LC}_k\tensor{\ACI}{_\beta^\nu}\tensor{\gamma}{^k_\nu^\mu}\tensor{\ACI}{_\mu^\alpha}\targetACI\tensor[_\alpha]{\psi}{} \right.\right. \\
				&\quad - \left.\left.\vphantom{\frac12}\tensor{\gamma}{^k_\beta^\alpha}\nabla^{\varphi^*\tangent{N}}_k\tensor[_\alpha]{\psi}{} \mp \tensor{\ACI}{_\beta^\nu}\tensor{\gamma}{^k_\nu^\alpha}\targetACI\nabla^{\varphi^*\tangent{N}}_k\tensor[_\alpha]{\psi}{}\right)\right) \\
			&= \Dirac\left(1\mp\ACI\otimes\targetACI\right)\psi \mp s^\beta\otimes\tensor{\gamma}{^k_\beta^\nu}\tensor{\ACI}{_\nu^\alpha}\left(\nabla_k^{\phi^*\tangent{N}}\targetACI\right)\tensor[_\alpha]{\psi}{}.
		\end{split}
	\end{equation}
	For the term with gravitino:
	\begin{equation}\label{eq:CommuteIJGravitino}
		\begin{split}
			\MoveEqLeft
			-2\left(1 + \ACI\otimes\targetACI\right)\left<\vee Q\chi,\differential{\varphi}\right>
			= -2\left<\vee Q\chi, \left(1 - \ACI\otimes\targetACI\right)\differential{\varphi}\right>
			= -4\left<\vee Q\chi, \differential{\varphi} - \DelJBar\varphi\right> \\
			&= -4\left<\vee Q\chi, \differential{\varphi} + \left<Q\chi, \psi\right> - \frac14\Tr_{\dual{g}_S}\left<\psi, \varphi^*\targetACI\right>\ACI\gamma\psi\right> \\
			&= -4\left<\vee Q\chi, \differential{\varphi}\right> + 2\norm{Q\chi}^2\psi + \left<\vee Q\chi, \Tr_{\dual{g}_S}\left<\psi, \varphi^*\targetACI\right>\ACI\gamma\psi\right>
		\end{split}
	\end{equation}
	To apply \(1+\ACI\otimes\targetACI\) to the curvature term we need some preparation:
	\begin{equation}\label{eq:SRComplexMult}
		\begin{split}
			\MoveEqLeftR
			\tensor{\ACI}{_\alpha^\beta}\targetACI\tensor[_\beta]{{SR(\psi)}}{}
			= i^*\varepsilon^{\mu\nu}\tensor{\ACI}{_\alpha^\beta}\targetACI R(F_\beta\Phi, F_\mu\Phi)F_\nu\Phi \\
			={}& i^*\varepsilon^{\mu\nu}\tensor{\ACI}{_\alpha^\beta}\targetACI\left(\nabla_{F_\beta}\nabla_{F_\mu}F_\nu\Phi + \nabla_{F_\mu}\nabla_{F_\beta}F_\nu\Phi - \nabla_{[F_\beta, F_\mu]}F_\nu\Phi\right) \\
			={}& i^*\varepsilon^{\mu\nu}\tensor{\ACI}{_\alpha^\beta}\left(\nabla_{F_\beta}\left(\targetACI\nabla_{F_\mu}F_\nu\Phi\right) - \left(\nabla_{F_\beta}\targetACI\right)\nabla_{F_\mu}F_\nu\Phi \right.\\
				&+ \left.\nabla_{F_\mu}\left(\targetACI\nabla_{F_\beta}F_\nu\Phi\right) - \left(\nabla_{F_\mu}\targetACI\right)\nabla_{F_\beta}F_\nu\Phi - \nabla_{[F_\beta, F_\mu]}\targetACI F_\nu\Phi + \left(\nabla_{[F_\beta, F_\mu]}\targetACI\right)F_\nu\Phi \right) \\
			={}& i^*\varepsilon^{\mu\nu}\tensor{\ACI}{_\alpha^\beta}\left(\nabla_{F_\beta}\nabla_{F_\mu}\targetACI F_\nu\Phi - \nabla_{F_\beta}\left(\left(\nabla_{F_\mu}\targetACI\right)F_\nu\Phi\right) - \left(\nabla_{F_\beta}\targetACI\right)\nabla_{F_\mu}F_\nu\Phi \right.\\
				&+ \nabla_{F_\mu}\nabla_{F_\beta}\targetACI F_\nu\Phi - \nabla_{F_\mu}\left(\left(\nabla_{F_\beta}\targetACI\right)F_\nu\Phi\right) - \left(\nabla_{F_\mu}\targetACI\right)\nabla_{F_\beta}F_\nu\Phi \\
				&- \left.\nabla_{[F_\beta, F_\mu]}\targetACI F_\nu\Phi + \left(\nabla_{[F_\beta, F_\mu]}\targetACI\right)F_\nu\Phi \right) \\
			={}& i^*\varepsilon^{\mu\nu}\tensor{\ACI}{_\alpha^\beta}\left(R(F_\beta\Phi, F_\mu\Phi)\targetACI F_\nu\Phi - \nabla_{F_\mu}\left(\left(\nabla_{F_\beta}\targetACI\right)F_\nu\Phi\right) - \left(\nabla_{F_\mu}\targetACI\right)\nabla_{F_\beta}F_\nu\Phi \right.\\
				&+ \left.\left(\nabla_{[F_\beta, F_\mu]}\targetACI\right)F_\nu\Phi \right) \\
			={}& \frac13\tensor[_\alpha]{{SR(\psi)}}{} + \frac13\targetACI \tensor{\ACI}{_\alpha^\beta}\tensor[_\beta]{{SR(\psi)}}{} + \frac13\tensor[_\alpha]{{SR(\psi)}}{} - \tensor{\ACI}{_\alpha^\beta}\tensor{\gamma}{^t_\beta^\mu}\left(\nabla_{\tensor[_\mu]{\psi}{}}\targetACI\right)\left(f_t\varphi + \tensor{\chi}{_t^\tau}\tensor[_\tau]{\psi}{}\right) \\
				&+ \tensor{\ACI}{_\alpha^\beta}\tensor{\gamma}{^t_\beta^\mu}\left(\nabla_{\tensor[_\mu]{\psi}{}}\targetACI\right)\left(f_t\varphi + \tensor{\chi}{_t^\tau}\tensor[_\tau]{\psi}{}\right) + 2\tensor{\ACI}{_\alpha^\beta}\tensor{\gamma}{^t_\beta^\nu}\left(\nabla_{f_t\varphi + \tensor{\chi}{_t^\tau}\tensor[_\tau]{\psi}{}}\targetACI\right)\tensor[_\nu]{\psi}{} \\
			={}& \frac23\tensor[_\alpha]{{SR(\psi)}}{} + \frac13\targetACI \tensor{\ACI}{_\alpha^\beta}\tensor[_\beta]{{SR(\psi)}}{} + 2\tensor{\ACI}{_\alpha^\beta}\tensor{\gamma}{^t_\beta^\nu}\left(\nabla_{f_t\varphi + \tensor{\chi}{_t^\tau}\tensor[_\tau]{\psi}{}}\targetACI\right)\tensor[_\nu]{\psi}{}
		\end{split}
	\end{equation}
	Consequently,
	\begin{equation}\label{eq:CommuteIJSR}
		\begin{split}
			\MoveEqLeftR
			-\frac13\left(1 + \ACI\otimes\targetACI\right)SR^N(\psi)
			= -\frac23SR^N(\psi) - s^\alpha\otimes \tensor{\ACI}{_\alpha^\beta}\tensor{\gamma}{^t_\beta^\nu}\left(\nabla_{f_t\varphi + \tensor{\chi}{_t^\tau}\tensor[_\tau]{\psi}{}}\targetACI\right)\tensor[_\nu]{\psi}{} \\
			={}& -\frac23SR^N(\psi) - s^\alpha\otimes \tensor{\ACI}{_\alpha^\beta}\tensor{\gamma}{^t_\beta^\nu}\left(\nabla_{f_t\varphi}\targetACI\right)\tensor[_\nu]{\psi}{} \\
				&- \frac12s^\alpha\otimes\tensor{\ACI}{_\alpha^\beta}\tensor{\gamma}{^t_\beta^\nu}\Gamma^s_{\nu\tau}\tensor{\chi}{_t^\tau}\varepsilon^{\kappa\sigma}\tensor{\gamma}{_s_\sigma^\lambda}\left(\nabla_{\tensor[_\kappa]{\psi}{}}\targetACI\right)\tensor[_\lambda]{\psi}{} \\
			={}& -\frac23SR^N(\psi) - s^\alpha\otimes \tensor{\ACI}{_\alpha^\beta}\tensor{\gamma}{^t_\beta^\nu}\left(\nabla_{f_t\varphi}\targetACI\right)\tensor[_\nu]{\psi}{} - \left<\vee Q\chi, \Tr_{\dual{g}_S}\left<\psi, \varphi^*\nabla\targetACI\right>\ACI\gamma\psi\right>
		\end{split}
	\end{equation}
	Summing up~\eqref{eq:CommuteIJDirac},~\eqref{eq:CommuteIJGravitino} and~\eqref{eq:CommuteIJSR} we obtain
	\begin{equation}
		\begin{split}
			0
			={}& \left(1 + \ACI\otimes\targetACI\right)\left(\Dirac\psi - 2\left<\vee Q\chi,\differential{\varphi}\right> - \frac13SR^N(\psi)\right) \\
			={}& 2\Dirac\psi + s^\alpha\otimes\tensor{\ACI}{_\alpha^\nu}\tensor{\gamma}{^k_\nu^\beta}\left(\nabla_{f_k}\targetACI\right)\tensor[_\beta]{\psi}{} \\
				&-4\left<\vee Q\chi, \differential{\varphi}\right> + 2\norm{Q\chi}^2\psi + \left<\vee Q\chi, \Tr_{\dual{g}_S}\left<\psi, \varphi^*\targetACI\right>\ACI\gamma\psi\right> \\
				&- \frac23SR^N(\psi) - s^\alpha\otimes \tensor{\ACI}{_\alpha^\beta}\tensor{\gamma}{^t_\beta^\nu}\left(\nabla_{f_t\varphi}\targetACI\right)\tensor[_\nu]{\psi}{} - \left<\vee Q\chi, \Tr_{\dual{g}_S}\left<\psi, \varphi^*\nabla\targetACI\right>\ACI\gamma\psi\right> \\
			={}& 2\left(\Dirac\psi - 2\left<\vee Q\chi, \differential{\varphi}\right> + \norm{Q\chi}^2\psi - \frac13SR^N(\psi)\right)
		\end{split}
	\end{equation}
	This completes the proof of the Corollary~\ref{cor:SJCComponentEquations}.
\end{proof}

\begin{rem}\label{rem:ConformalVarianceComponentsDJBar}
	In contrast to \(\DJBar\Phi\), its component fields depend on the superconformal metric \(m\).
	Indeed, suppose that the superconformal metric \(\tilde{m}\) is obtained from \(m\) by a superconformal rescaling by \(\lambda\in\cO_M\) and a superconformal change of splitting \(l\in\cD\).
	Then the induced change of the metric \(g\) and the gravitino \(\chi\) are given by
	\begin{align}
		\tilde{g} &= {\left(i^\#\lambda\right)}^4 g, &
		\tilde{g}_S &= {\left(i^\#\lambda\right)}^2 g_S, &
		\tilde{\chi} = \chi + \gamma i^*l.
	\end{align}
	Let us denote the standard isometries
	\begin{align}
		b\colon (\tangent{\Smooth{M}}, g) &\to (\tangent{\Smooth{M}}, \tilde{g}) &
		\beta\colon S_g &\to S_{\tilde{g}} \\
		X &\mapsto {(i^\#\lambda)}^{-2}X &
		s &\mapsto {(i^\#\lambda)}^{-1}s
	\end{align}
	The almost complex structure \(\ACI\) on \(\tangent{\Smooth{M}}\) and \(S\) are preserved under \(b\) and \(\beta\).

	The component fields of \(\Phi\) calculated with respect to \(\tilde{m}\) are given by
	\begin{align}
		\tilde{\varphi} &= \varphi, &
		\tilde{\psi} &= {(i^\#\lambda)}^{-1}\left({(\dual{\beta})}^{-1}\otimes \id_{\varphi^*\tangent{N}}\right)\psi, &
		\tilde{F} &= {(i^\#\lambda)}^{-2}F.
	\end{align}
	The components of \(\DJBar\Phi\) transform also homogenously.
	That is, the component of degree zero transforms by \({(i^\#\lambda)}^{-1}\left({(\dual{\beta})}^{-1}\otimes\id_{\varphi^*\tangent{N}}\right)\), the first order component involving \(F\) transforms by \({(i^\#\lambda)}^{-2}\) and the other one is invariant.
	The second order term transforms by \({(i^\#\lambda)}^{-3}\left({(\dual{\beta})}^{-1}\otimes\id_{\varphi^*\tangent{N}}\right)\).
	Hence, in particular the solutions of the component equations of \(\DJBar\Phi=0\) do not depend on the chosen superconformal metric~\(m\).
\end{rem}

\begin{cor}\label{Cor:JHolomorphicCurvesTrivialFamilySRS}
	Let \(M\) be a trivial family of super Riemann surfaces over \(B\), \(i\colon\Smooth{M}\to M\) be holomorphic and \(N\) Kähler.
	There is a bijection between super \(\targetACI\)-holomorphic curves \(\Phi\colon M\to N\) and tuples \((\varphi, \psi)\), where \(\varphi\colon \Smooth{M}\to N\) is a \(\targetACI\)-holomorphic curve and \(\psi\) a holomorphic section of \(\dual{S}\otimes_\C\varphi^*\tangent{N}\).
	In the case \(B=\R^{0|0}\) the set of super \(\targetACI\)-holomorphic curves on \(M\) is in bijection to the set of \(\targetACI\)-holomorphic curves \(\varphi\colon \Red{M}\to N\).
\end{cor}
\begin{proof}
	If the embedding \(i\colon \Smooth{M}\to M\) is holomorphic, we have \(Q\chi=0\), see~\cite[Proposition~11.1.10]{EK-SGSRSSCA}.
	As the target~\(N\) is Kähler it holds \(\nabla\targetACI=0\) and \(\targetACI R^N(X, Y) = R^N(X, Y)\targetACI\).
	Together with \(\left(1+\ACI\otimes\targetACI\right)\psi=0\) and Lemma~\ref{lemma:SRIdentities} this yields
	\begin{equation}
		\begin{split}
			SR^N(\psi)
			&= \left(\ACI\otimes \targetACI\right)\left(\ACI\otimes\targetACI\right) SR^N(\psi)
			= \left(\ACI\otimes\targetACI\right) s^\alpha\otimes\tensor{\ACI}{_\alpha^\beta}\varepsilon^{\mu\nu}\targetACI R^N(\tensor[_\beta]{\psi}{}, \tensor[_\mu]{\psi}{})\tensor[_\nu]{\psi}{} \\
			&= \left(\ACI\otimes\targetACI\right) s^\alpha\otimes\tensor{\ACI}{_\alpha^\beta}\varepsilon^{\mu\nu}\tensor{\ACI}{_\nu^\sigma} R^N(\tensor[_\beta]{\psi}{}, \tensor[_\mu]{\psi}{})\tensor[_\sigma]{\psi}{}
			= \frac13\left(\ACI\otimes\targetACI\right) SR^N(\psi)
			= \frac19 SR^N(\psi)
		\end{split}
	\end{equation}
	and consequently \(SR^N(\psi)=0\).
	The components of \(\DJBar\Phi=0\) simplify considerably to:
	\begin{align}
		\DelJBar\varphi &= 0, &
		\left(1+\ACI\otimes\targetACI\right)\psi &= 0, &
		\Dirac\psi &= 0, &
		F &= 0.
	\end{align}
	The two conditions on \(\psi\) are equivalent to \(\psi\) being a holomorphic section of \(\dual{S}\otimes_\C\varphi^*\tangent{N}\).
	Hence, for given \(\varphi\) satisfying \(\DelJBar\varphi = 0\) and every holomorphic section of \(\dual{S}\otimes_\C\varphi^*\tangent{N}\) the tuple \((\varphi, \psi, F=0)\) determines a super \(\targetACI\)-holomorphic curve \(\Phi\).

	To see why the equations on \(\psi\) are equivalent to holomorphicity, let \(x^a\) be local conformal coordinates on \(\Smooth{M}\), \(\partial_{z}=\frac12\left(\partial_{x^1}-\ic\partial_{x^2}\right)\) and let \(g(\partial_{x^a}, \partial_{x^b}) = \delta_{ab}\lambda^4\).
	Then \(f_a=\frac1{\lambda^2}\partial_{x^a}\) is an orthonormal frame.
	There is a spinor frame \(t=t_3+\ic t_4\in\VSec{S}\) such that \(t\otimes t=\partial_z\) and \(s_\alpha=\lambda t_\alpha\) is an orthonormal spinor frame.
	For the dual frame \(s^\alpha\) of \(s_\alpha\) it holds that
	\begin{equation}
		\nabla^{\dual{S}}_X s^\alpha
		= -\frac12 \left<X, \omega^{LC}\right> s^\beta\tensor{\ACI}{_\beta^\alpha},
	\end{equation}
	where \(\left<f_k, \omega^{LC}\right> = \lambda^2\tensor{\ACI}{_k^l}\partial_{x^l}\frac1{\lambda^2}\).
	If we now assume that \(\tangent{N}\) is trivialized by a set of holomorphic frames \(e_B\), the field \(\psi\) can be written locally as
	\begin{equation}
		\psi = s^\alpha \tensor[_\alpha]{\psi}{^B}\otimes \varphi^*e_B
		= \left(\tensor[_3]{\psi}{^B} + \ic\tensor[_4]{\psi}{^B}\right)\frac1\lambda \dual{t}\otimes \varphi^*e_B
	\end{equation}
	It remains to show that the Dirac-equation implies the holomorphicity of the coefficients \(\left(\tensor[_3]{\psi}{^B} + \ic\tensor[_4]{\psi}{^B}\right)\frac1\lambda\):
	\begin{equation}
		\begin{split}
			\Dirac \psi
			&= s^\beta\left(\frac12 \omega_k^{LC}\tensor{\gamma}{^k_\beta^\mu}\tensor{\ACI}{_\mu^\alpha}\tensor[_\alpha]{\psi}{^B} + \tensor{\gamma}{^k_\beta^\alpha} f_k\tensor[_\alpha]{\psi}{^B}\right) \varphi^*e_B - s^\beta\tensor{\gamma}{^k_\beta^\alpha}\tensor[_\alpha]{\psi}{^B}\otimes \nabla_{f_k}\varphi^*e_B \\
			&= \left(s^3 \frac1{\lambda}\left(\partial_{x^1}(\tensor[_4]{\psi}{^B}\frac1\lambda) - \partial_{x^2}(\tensor[_3]{\psi}{^B}\frac1\lambda)\right)
			+ s^4 \frac1{\lambda}\left(\partial_{x^1}(\tensor[_3]{\psi}{^B}\frac1\lambda) + \partial_{x^2}(\tensor[_4]{\psi}{^B}\frac1\lambda)\right)\right)\otimes\varphi^*e_B
		\end{split}
	\end{equation}
	Here we have used in the second step that \(N\) is Kähler and hence the terms with covariant derivatives vanish:
	\begin{equation}
		\begin{split}
			s^\beta\tensor{\gamma}{^k_\beta^\alpha}\tensor[_\alpha]{\psi}{^b}\otimes \nabla_{f_k}\varphi^*e_b
			&= \frac12s^\beta\tensor{\gamma}{^k_\beta^\alpha}\left(\tensor[_\alpha]{\psi}{^b} - \tensor{\ACI}{_\alpha^\mu}\tensor[_\mu]{\psi}{^c}\tensor{\targetACI}{_c^b}\right)\otimes \nabla_{f_k}\varphi^*e_b \\
			&= s^\beta\tensor{\gamma}{^k_\beta^\alpha}\tensor[_\alpha]{\psi}{^b}\otimes \left(\bar\partial \varphi^*e_b\right)(f_k)
			= 0
			\qedhere
		\end{split}
	\end{equation}
\end{proof}

Corollary~\ref{Cor:JHolomorphicCurvesTrivialFamilySRS} gives a rough idea of the expected dimension of the moduli space of super \(\targetACI\)-holomorphic curves.
By~\cite[Example~15.2.1]{F-IT}, we have
\begin{equation}
	\begin{split}
		\chi(\Smooth{M}, \dual{S}\otimes_\C\varphi^*\tangent{N})
		&= \dim_\C H^0(\Smooth{M}, \dual{S}\otimes_\C\varphi^*\tangent{N}) - \dim_\C H^1(\Smooth{M}, \dual{S}\otimes\varphi^*\tangent{N}) \\
		&= \left<c_1(\dual{S}\otimes_\C\varphi^*\tangent{N}), \Smooth{M}\right> + n(1-p) \\
		&= \left<c_1(\tangent{N}), A\right> + n(p-1) + n(1-p)
		= \left<c_1(\tangent{N}), A\right>
	\end{split}
\end{equation}
Here \(A\) denotes the homology class of \(\im \Red{\Phi}\), \(2n\) is the real dimension of \(N\) and \(p\) the genus of \(\Smooth{M}\).
We have used that \(\dual{S}\otimes_\C\dual{S} = \cotangent{\Smooth{M}}\) and hence \(c_1(\dual{S})=p-1\).

The holomorphic fields \(\psi\in\VSec{\dual{S}\otimes_\C\varphi^*\tangent{N}}\) constitute the odd fiber of the moduli space of super \(\targetACI\)-holomorphic curves at the point given by \(\varphi\).
Consequently, it would be expected, that the moduli space of super \(\targetACI\)-holomorphic curves is of real dimension
\begin{equation}
	2n(1-p) + 2\left<c_1(\tangent{N}), A\right>| 2\left<c_1(\tangent{N}), A\right>.
\end{equation}
We will show in Section~\ref{Sec:ModuliSpace} that under certain conditions there is indeed a moduli space of super \(\targetACI\)-holomorphic curves as supermanifold of that dimension.

\section{The space of maps}\label{Sec:SpaceOfMaps}
The goal of this section is to recall the necessary theory of infinite-dimensional supermanifolds and setup the needed infinite dimensional supermanifolds of maps.
To this end, we use the Molotkov--Sachse approach to infinite-dimensional supermanifolds, as proposed in~\cite{M-IDCSM} and worked out in~\cite{S-GAASTS}.
The supermanifold structure on
\begin{equation}
	\underline{\Fieldspace}
	=\Set{\Phi\colon M\to N}
\end{equation}
has been studied before in~\cite{H-SSSMSM}.
However, we propose a new way to construct charts on \(\underline{\Fieldspace}\) for the case of maps from a closed compact super Riemann surface to a Riemannian manifold using the component field formalism and the exponential map on the target.
Furthermore, we construct the space of maps \(\underline{\Fieldspace}_{\mathcal{X}}\) where the super Riemann surface structure of the domain varies for gravitinos given in a finite-dimensional superdomain~\(\underline{\mathcal{X}}\).
Using parallel transport we construct the vector bundle \(\underline{\Targetbundle}_{\mathcal{X}}\to\underline{\Fieldspace}_{\mathcal{X}}\), whose fiber over~\(\Phi\) is given by \(\VSec{\dual{\cD}\otimes \Phi^*\tangent{N}}\).
The operator \(\DJBar\) yields a section \(\underline{\mathcal{S}}\colon \underline{\Fieldspace}_{\mathcal{X}}\to\underline{\Targetbundle}_{\mathcal{X}}\).
This construction is tailored to allow to determine the moduli space of super \(\targetACI\)-holomorphic curves \(\underline{\mathcal{M}}(\mathcal{X}, A)\) as the zero-locus of the section~\(\underline{\mathcal{S}}\) in the next Section~\ref{Sec:ModuliSpace}.

\subsection{Inifinite dimensional supermanifolds}
Let \(C\) and \(M\) be supermanifolds.
A \(C\)-point of \(M\) is a map \(p\colon C\to M\) and all \(C\)-points form a set \(\Hom(C, M)\).
By Yoneda-Lemma, the supermanifold \(M\) is completely determined by the functor
\begin{equation}
	\begin{split}
		\cat{SFinMan}^{op}&\to \cat{Set} \\
		C&\mapsto \Hom(C, M)
	\end{split}
\end{equation}
from the opposite of the category of finite-dimensional super manifolds to the category of sets.

The idea of the Molotkov--Sachse approach to supermanifolds is to restrict to \(\R^{0|k}\)-points of \(M\) but to enrich the resulting sets to carry the structure of a smooth manifold.
Let us denote by \(\cat{SPoint}\) the restriction of the category of supermanifolds to the superpoints~\(\R^{0|k}\) for all \(k\in\NaturalNumbers\).
The opposite category \(\cat{SPoint}^{op}\) is isomorphic to the category \(\cat{Gr}\) of finitely generated Grassmannn algebras.
In the Molotkov--Sachse approach, a finite-dimensional supermanifold \(\underline{M}\) is then a covariant functor with values in the category of finite-dimensional smooth manifolds
\begin{equation}
	\underline{M}\colon \cat{SPoint}^{op}\to \cat{FinMan}
\end{equation}
with a Grothendieck topology and locally isomorphic to the point functor of a superdomain.

\begin{ex}[Point functor of \(\R^{m|n}\)]\label{ex:PointFunctorRmn}
	Notice that the superdomains \(\R^{0|k}\) are by definition a topological point with the sheaf of functions isomorphic to the Grassmann algebra \(\bigwedge_k\).
	Let us denote the coordinates of \(\R^{0|k}\), that is, a set of generators for \(\bigwedge_k\), by \((\lambda^\kappa)\).

	Furthermore, let \((x^a, \eta^\alpha)\) be the standard coordinates of the superdomain \(\R^{m|n}\).
	We call a map \(p\colon \R^{0|k}\to \R^{m|n}\) an \(\R^{0|k}\)-point of \(\R^{m|n}\).
	By the charts theorem every such \(\R^{0|k}\)-point is given by
	\begin{align}
		p^\#x^a &= \sum_{\underline{\kappa}\text{ even}} \lambda^{\underline{\kappa}}\tensor[_{\underline{\kappa}}]{p}{^a}, &
		p^\#\eta^\alpha &= \sum_{\underline{\kappa}\text{ odd}} \lambda^{\underline{\kappa}}\tensor[_{\underline{\kappa}}]{p}{^\alpha},
	\end{align}
	where the sum runs over even resp.\ odd \(\Z_2\)-multiindices and \(\tensor[_{\underline{\kappa}}]{p}{^A}\in\R\).
	Denote the standard basis of the super vector space \(\R^{m|n}\) by \(e_A\), where \(e_a\) is a basis for the even part and \(e_\alpha\) a basis of the odd part.
	Now, \(\R^{0|k}\)-points of \(\R^{m|n}\) are in one-to-one correspondence to elements
	\begin{equation}
		\lambda^{\underline{\kappa}}\otimes \tensor[_{\underline{\kappa}}]{p}{^A} e_A
		= \sum_{\underline{\kappa}\text{ even}} \lambda^{\underline{\kappa}}\otimes \tensor[_{\underline{\kappa}}]{p}{^a} e_a + \sum_{\underline{\kappa}\text{ odd}} \lambda^{\underline{\kappa}}\otimes \tensor[_{\underline{\kappa}}]{p}{^\alpha} e_\alpha
	\end{equation}
	of \({\left(\R^{m|n}\otimes \bigwedge_k\right)}_0\), the even part of the tensor product of the super vector space \(\R^{m|n}\) with the Grassmann algebra~\(\bigwedge_k\).
	In particular, \(\Hom(\R^{0|k}, \R^{m|n})={\left(\R^{m|n}\otimes \bigwedge_k\right)}_0\) is equipped with the structure of a topological vector space and a manifold structure.
	In the Molotkov--Sachse approach the supermanifold \(\R^{m|n}\) is given by the covariant functor
	\begin{equation}
		\begin{split}
			\underline{\R}^{m|n}\colon \cat{SPoint}^{op}&\to \cat{FinMan} \\
			\R^{0|k}&\mapsto {\left(\R^{m|n}\otimes {\bigwedge}_k\right)}_0.
		\end{split}
	\end{equation}

	Smooth maps \(f\colon \R^{m|n}\to \R^{p|q}\) between superdomains correspond in the Molotkov--Sachse approach to natural transformations \(\underline{\R}^{m|n}\to\underline{\R}^{p|q}\).
	The \(\R^{0|k}\)-point \(p\colon \R^{0|k}\to \R^{m|n}\) is mapped under \(f\) to the \(\R^{0|k}\)-point \(f\circ p\) of \(\R^{p|q}\).
	One can work out the explicit formula to show that the map \(p\mapsto f\circ p\) is a smooth map from \({\left(\R^{m|n}\otimes {\bigwedge}_k\right)}_0\) to \({\left(\R^{p|q}\otimes {\bigwedge}_k\right)}_0\).
\end{ex}

In~\cites{M-IDCSM}{S-GAASTS} it is shown that the category of finite-dimensional supermanifolds as ringed spaces can equivalently be reformulated as functors \(\cat{SPoint}^{op}\to\cat{FinMan}\).
The main idea to generalize to infinite-dimensional supermanifolds is to replace the category~\(\cat{FinMan}\) of finite-dimensional manifold by a suitable category of infinite-dimensional manifolds, such as Banach or Fréchet manifolds.
In the following, we are interested in functors \(\underline{M}\colon\cat{Spoint}^{op}\to\cat{Man}\) and natural transformations between them.
In this work, the category \(\cat{Man}\) will be either the category of finite-dimensional manifolds \(\cat{FinMan}\), the category of Fréchet-manifolds \(\cat{FrMan}\) or the category \(\cat{BMan}\) of Banach-manifolds with their respective smooth maps.
In principle, other categories of locally convex spaces such as the more general convenient manifolds, see~\cites{KM-CSGA}{H-SSSMSM}, are also suitable.

\begin{defn}[{see~\cites[Chapter~3.5, 3.6]{S-GAASTS}{M-IDCSM}}]\label{defn:SManFunctor}
A functor \(\underline{U}\colon\cat{SPoint}^{op}\to\cat{Man}\) is called an open subset of \(\underline{M}\colon\cat{SPoint}^{op}\to\cat{Man}\), or \(\underline{U}\subset\underline{M}\) if there is a functor morphism \(\underline{U}\to\underline{M}\) such that for every \(C\in\cat{SPoint}\) the manifold \(\underline{U}(C)\) is an open subset of \(\underline{M}(C)\).

	Let \(V\) be a topological super vector space, that is \(V=V_0\oplus V_1\) where \(V_i\in\cat{Man}\) and define \(\underline{V}\colon \cat{SPoint}^{op}\to\cat{Man}\) by \(\underline{V}(C)={\left(V\otimes\cO_C\right)}_0\).
	For an open subset \(U\subset V\), we define the restriction \(\underline{V}|_{U}\) of \(\underline{V}\) to \(U\) by
	\begin{equation}
		\underline{V}|_U(\R^{0|k})
		= \Set{ p = \sum_{\underline{\kappa}}p_{\underline{\kappa}}\otimes \lambda^{\underline{\kappa}}\in{\left(V\otimes{\bigwedge}_k\right)}_0\given p_0\in U}.
	\end{equation}
	One can show that all open subfunctors of \(\underline{V}\) are of the form \(\underline{V}|_U\).
	An open subfunctor of \(\underline{V}\) is called a superdomain.

	A functor morphism \(f\colon\underline{U}\to\underline{U}'\) is called supersmooth if for every \(C\in\cat{SPoint}\) the maps \(f(C)\colon\underline{U}(C)\to\underline{U}'(C)\) are smooth and their differentials are \({\left(\cO_C\right)}_0\)-linear at every point.

	A collection of open subfunctors \(\Set{\underline{U}_i}\) of \(\underline{M}\) is called a supersmooth atlas if
	\begin{itemize}
		\item
			Every \(\underline{U}_i\) is isomorphic to a superdomain.
		\item
			The open subsets \(\Set{\underline{U}_i}\) cover \(\underline{M}\), that is, for every \(C\in\cat{SPoint}\) the open sets \(\underline{U}_i(C)\) cover \(\underline{M}(C)\).
		\item
			For all \(i\) and \(j\) the fiber product \(\underline{U}_i\times_{\underline{M}}\underline{U}_j\) has the structure of a superdomain such that the projections \(\underline{U}_i\times_{\underline{M}}\underline{U}_j\to\underline{U}_i\) and \(\underline{U}_i\times_{\underline{M}}\underline{U}_j\to\underline{U}_j\) are supersmooth.
	\end{itemize}

	As usual, two supersmooth atlases are equivalent if their union is also an atlas.
	A functor \(\underline{M}\) with an equivalence class of supersmooth atlases is called a supermanifold.
	A functor morphism \(f\colon \underline{M}\to\underline{M}'\) is called supersmooth if it restricts to supersmooth maps of superdomains using atlases.

	We denote the resulting category by \(\cat{SMan}\), or more specifically by \(\cat{SFinMan}\) for finite-dimensional supermanifolds.
\end{defn}

For further comments on those definitions we refer to the original works.
Let us just note that it has been proven in~\cite{M-IDCSM} that the category \(\cat{SFinMan}\) is equivalent to the category of supermanifolds in the ringed space approach.
Working with the more general notions of possibly infinite dimensional manifolds allows to define infinite dimensional supermanifold, extending the notion of supermanifold in the ringed space approach.

For the remainder of this work we will need a slight generalization of \(\cat{SMan}\): the category \(\cat{SMan}_B\) of families of supermanifolds over \(B\).
In order not to deal with the topological complications arising from a topologically non-trivial base, we will assume \(B\in\cat{SPoint}\).

Let \(\cat{SPoint}_B\) denote the category of superpoints over \(B\), that is superpoints \(C\to B\) together with smooth maps \(C\to C'\) over \(B\).
Then we define the category \(\cat{SMan}_B\) repeating Definition~\ref{defn:SManFunctor} where we replace \(\cat{SPoint}\) by \(\cat{SPoint}_B\).
In order to justify why we call the category \(\cat{SMan}_B\) the category of supermanifolds fibered over \(B\) let us briefly give an argument why the category \(\cat{SFinMan}_B\) is equivalent to the category of supermanifolds over \(B\) in the ringed space formalism:
First notice that there is an inclusion \(\cat{SPoint}\to\cat{SPoint}_B\) that sends every \(C\in\cat{SPoint}\) to the trivial \(C\to \R^{0|0}\to B\) in \(\cat{SPoint}_B\).
Consequently, every \(\underline{M}\in\cat{SFinMan}_B\) can also be seen as a supermanifold in \(\cat{SFinMan}\) and hence as a supermanifold \(M\) in the ringed space formalism.
The object \(\underline{B}\in\cat{SFinMan}_B\) given by all \(C\)-points of \(B\) is a final object in the category \(\cat{SFinMan}\) hence we obtain a map \(M\to B\) in the ringed space formalism.
It now remains to show that \(M\to B\) is a submersion which can be seen in the local picture, analogously to Example~\ref{ex:PointFunctorRmn}.

\subsection{A supermanifold-structure on \texorpdfstring{\(\Hom(M, N)\)}{Hom(M, N)}}
In this section we want to equip the sets
\begin{equation}
	\Fieldspace
	=\Set{\Phi\colon M\to N}
\end{equation}
of maps from a closed compact super Riemann surface to a Riemannian manifold \(N\) with the structure of an infinite-dimensional Fréchet supermanifold.
To this end, we construct local charts around a given \(\Phi\) using the exponential map on \(N\) and the component field formalism for super Riemann surfaces.
In a first step we will focus on the super Riemann surface \(M\) on \(\Red{M}\) determined by \(\Red{g}\), \(\Red{S}\) and \(\chi=0\).
In a second step we will allow the gravitino to vary in a finite-dimensional superdomain \(\underline{\mathcal{X}}\).

The work~\cite{H-SSSMSM} has achieved the goal of constructing a supermanifold structure on \(\Hom(M, N)\) in a different setting.
In particular, there, \(M\) is not required to be compact or a super Riemann surface and \(N\) can have odd dimensions and does not need a Riemannian metric.
However, the case of non-trivial families \(M\) over \(B\) is not treated in~\cite{H-SSSMSM}.

The exponential map of the classical Riemannian manifold \((N, n)\) allows us to identify maps in the neighbourhood of a given map \(\Phi\colon M\to N\) with a neighbourhood of zero in the \({\left(\cO_B\right)}_0\)-module \({\VSec{\Phi^*\tangent{N}}}_0\) of even sections of \(\Phi^*\tangent{N}\).
This is in analogy to the case of maps between ordinary manifolds as treated, for example, in~\cite[Theorem~42.1]{KM-CSGA}.
The exponential map \(\exp\) on \(N\) is not defined on all \(\tangent{N}\) but we can assume that there is a neighbourhood \(U\) of the zero-section which is diffeomorphic under \((\pi_N, \exp)\) to an open neighbourhood \(V\) of the diagonal in \(N\times N\).
For a given \(\Phi\in\Hom_B(M, N)\) we denote by
\begin{equation}
	V_\Phi = \Set{\tilde{\Phi}\in\Hom_B(M, N)\given (\Red{\Phi}(\Red{M}), \Red{\tilde{\Phi}}(\Red{M}))\subset V}
\end{equation}
and
\begin{equation}
	\begin{split}
		v_\Phi\colon V_\Phi &\to {\VSec{\Phi^*\tangent{N}}}_0 \\
		\tilde{\Phi} &\mapsto {(\pi_N, \exp)}^{-1}\circ(\Phi, \tilde{\Phi})
	\end{split}
\end{equation}
Notice that \({(\pi_N, \exp)}^{-1}\circ(\Phi, \tilde{\Phi})\) is a map from \(M\) to \(\tangent{N}\) that projects down to \(\Phi\) and hence corresponds to an even section of \(\Phi^*\tangent{N}\).
The inverse can be taken, whenever \((\Red{\Phi}(\Red{M}), \Red{\tilde{\Phi}}(\Red{M}))\subset V\).
Hence, \(v_\Phi\) can be restricted to an bijection of sets \(v_\Phi\colon V_\Phi\to U_\Phi\subset{\VSec{\Phi^*\tangent{N}}}_0\).
We denote the inverse of \(v_\Phi\) by \(\exp_\Phi\):
\begin{equation}
	\begin{split}
		\exp_\Phi\colon U_\Phi\subset{\VSec{\Phi^*\tangent{N}}}_0&\to V_\Phi\subset\Hom_B(M, N)\\
		Z &\mapsto {(v_\Phi)}^{-1}Z
	\end{split}
\end{equation}

Let now \(M\) be the super Riemann surface over \(\R^{0|0}\) with \(\Red{i}\colon \Red{M}\to M\) determined by the metric \(\Red{g}\), the spinor bundle \(\Red{S}\) and the gravitino \(\chi=0\) on \(\Red{M}\).
Consider the functor
\begin{equation}
	\begin{split}
		\underline{\Fieldspace}\colon \cat{SPoint}^{op} &\to \cat{Set}\\
		C &\mapsto \underline{\Fieldspace}(C)=\Hom_C(M\times C, N\times C)
	\end{split}
\end{equation}
We will show below that we can use the sets \(V_\Phi\) to construct charts for \(\underline{\Fieldspace}\).
This will lead to an infinite-dimensional Fréchet-supermanifold \(\underline{\Fieldspace}\) such that its \(C\)-points are given by families of maps \(M\to N\) parametrized by \(C\).

Recall that any map \(\Phi\in\Hom_C(M\times C, N)\) is determined by its component fields \(\varphi\colon \Smooth{M}\to N\), \(\psi\in\VSec{\cotangent{\Smooth{M}}\otimes S}\) and \(F\in\VSec{\Smooth{\varphi}^*\tangent{N}}\).
We denote the base extension of \(\Red{i}\colon \Red{M}\to M\) by \(i_C\colon \Smooth{M}=\Red{M}\times C\to M\times C\) and fix a map \(\phi\colon \Red{M}\to N\).
For every \(C\), the map \(\phi\) determines maps \(\Phi_C\in\Hom(M\times C, N)\) whose component fields are given by \(\varphi_C=\phi\times\id_C\), \(\psi_C=0\) and \(F_C=0\).
Setting
\begin{equation}
	\underline{V}_\phi(C) = \Set{\tilde{\Phi}\in\Hom_C(M\times C, N\times C)\given (\phi(\Red{M}), \Red{\tilde{\Phi}}(\Red{M}))\subset V}
\end{equation}
and
\begin{equation}
	\begin{split}
		\underline{v}_\phi(C)\colon \underline{V}_\phi(C) &\to \underline{U}_\phi(C)\subset {\VSec{\Phi_C^*\tangent{N}}}_0 \\
		\tilde{\Phi} &\mapsto {(\pi_N, \exp)}^{-1}\circ(\Phi_C, \tilde{\Phi})
	\end{split}
\end{equation}
one obtains a chart of \(\underline{\Fieldspace}\) around \(\phi\).
Every map \(\tilde{\Phi}\in\Hom_C(M\times C, N)\) is contained in \(\underline{V}_{\Red{\tilde{\Phi}}}\).
Consequently, the set of charts \(\Set{\underline{V}_\phi}\) covers \(\underline{\Fieldspace}\).

In order to give \(\underline{\Fieldspace}\) the structure of a supermanifold we equip  \(\underline{U}_\phi=\underline{v}_\phi(\underline{V}_\phi)\) with the structure of an infinite-dimensional Fréchet superdomain using the component field formalism.
By definition, the component field \(\varphi_C\) of \(\Phi_C\) is of the form \(\varphi_C=\phi\times\id_C\).
It follows \(\VSec{\varphi_C^*\tangent{N}}=\VSec{\phi^*\tangent{N}}\otimes\cO_C\) and for \(S_C=i_C^*\cD\) that \(\VSec{S_C}=\VSec{\Red{S}}\otimes\cO_C\).

We denote the component fields of \(Y\in{\VSec{\Phi_C^*\tangent{N}}}_0\) by
\begin{equation}\label{eq:DefnComponentFieldsY}
	\begin{aligned}
		\xi &= i_C^*Y \in{\VSec{\varphi_C^*\tangent{N}}}_0, \\
		\zeta &= s^\alpha\otimes i_C^*\nabla_{F_\alpha}Y \in{\VSec{\dual{S}\otimes\varphi_C^*\tangent{N}}}_0, \\
		\sigma &= -\frac12i_C^*\DLaplace Y \in{\VSec{\varphi_C^*\tangent{N}}}_0.
	\end{aligned}
\end{equation}
As those component fields determine \(Y\) completely, we obtain isomorphisms
\begin{equation}
	\begin{split}
		{\VSec{\Phi_C^*\tangent{N}}}_0
		&\cong {\VSec{\varphi_C^*\tangent{N}}}_0 \oplus {\VSec{\dual{S_C}\otimes\varphi_C^*\tangent{N}}}_0 \oplus {\VSec{\varphi_C^*\tangent{N}}}_0 \\
		&\cong {\left(\left(\VSec{\phi^*\tangent{N}} \oplus \VSec{\Red{\dual{S}}\otimes\phi^*\tangent{N}} \oplus \VSec{\phi^*\tangent{N}}\right) \otimes \cO_C\right)}_0 \\
		&\cong \VSec{\phi^*\tangent{N}}\otimes{\left(\cO_C\right)}_0 \oplus \VSec{\Red{\dual{S}}\otimes\phi^*\tangent{N}}\otimes{\left(\cO_C\right)}_1 \oplus \VSec{\phi^*\tangent{N}}{\otimes\left(\cO_C\right)}_0 \\
	\end{split}
\end{equation}
We call the space on the last line \(\underline{W}_\phi(C)\).
The isomorphism \({\VSec{\Phi_C^*\tangent{N}}}\cong\underline{W}_\phi(C)\) can be made more explicit by choosing coordinates \(\lambda^{\kappa}\), \(\kappa=1, \dotsc, k\) on \(C=\R^{0|k}\).
Then, the even component fields \(\xi\), \(\zeta\) and \(\sigma\) of \(Y\in{\VSec{\Phi_C^*\tangent{N}}}_0\) can be expanded as follows:
\begin{align}
	\xi &= \sum_{\underline{\kappa}\text{ even}}\lambda^{\underline{\kappa}}\tensor[_{\underline{\kappa}}]{\xi}{} &
	\zeta &= \sum_{\underline{\kappa}\text{ odd}}\lambda^{\underline{\kappa}}\tensor[_{\underline{\kappa}}]{\zeta}{} &
	\sigma &= \sum_{\underline{\kappa}\text{ even}}\lambda^{\underline{\kappa}}\tensor[_{\underline{\kappa}}]{\sigma}{}
\end{align}
Here, \(\tensor[_{\underline{\kappa}}]{\xi}{}\) and \(\tensor[_{\underline{\kappa}}]{\sigma}{}\) are sections of \(\phi^*\tangent{N}\) and \(\tensor[_{\underline{\kappa}}]{\zeta}{}\) are sections of \(\Red{\dual{S}}\otimes\phi^*\tangent{N}\).
Notice that \(Y\in\underline{U}_\phi(C)=\im\underline{v}_\phi(C)\) if and only if \(\tensor[_0]{\xi}{}\in U\), where \(U\subset \tangent{N}\) is the open subset isomorphic to \(V\subset N\times N\) under \((\pi_N, \exp)\).

The vector spaces of sections \(\VSec{\phi^*\tangent{N}}\) and \(\VSec{\Red{\dual{S}}\otimes\phi^*\tangent{N}}\) have the structure of Fréchet spaces induced by the \(C^k\)-seminorms given by the metrics \(g_S\) and \(\phi^*\tangent{N}\).
Consequently also the spaces \(\underline{W}_\phi(C)\) are Fréchet spaces as a finite sum of Fréchet spaces is again Fréchet.
Setting
\begin{equation}
	\underline{U}_\phi(C) = \Set{(\xi,\zeta, \sigma)\in\underline{W}_\phi(C)\given \tensor[_0]{\xi}{}\in U}
\end{equation}
we obtain an open Fréchet superdomain \(\underline{U}_\phi\) and bijective functor morphism
\begin{equation}
	\underline{v}_\phi\colon \underline{V}_\phi\to \underline{U}_\phi.
\end{equation}
In addition, a change of charts is a smooth map of superdomain, as they can be obtained through the exponential map on \(N\).
Hence we have proven:
\begin{prop}
	The functor \(\underline{\Fieldspace}\) carries an infinite-dimensional Fréchet supermanifold structure given by the atlas \(\Set{\underline{V}_\phi}\).
\end{prop}

Up to now we have constructed the supermanifold of maps \(M\to N\) for a fixed super Riemann surface \(M\) constructed by on \(\Red{M}\) from \(\Red{g}\), \(\Red{S}\) and \(\chi=0\).
We will now construct a fibered version \(\underline{\Fieldspace}_{\mathcal{X}}\) of \(\underline{\Fieldspace}\) for a finite-dimensional space \(\mathcal{X}\) of gravitinos.
More precisely, let \(\mathcal{X}\subset\VSec{\cotangent{\Red{M}}\otimes\Red{S}}\) a finite-dimensional sub-vector space and set
\begin{equation}
	\begin{split}
		\underline{\mathcal{X}}\colon \cat{SPoint}^{op} &\to \cat{Man} \\
		C &\mapsto \mathcal{X}\otimes{\left(\cO_C\right)}_1
	\end{split}
\end{equation}
That is, \(\underline{\mathcal{X}}\) is a superdomain of dimension \(0|\dim \mathcal{X}\).
For every \(\chi\in \underline{\mathcal{X}}(B)\) we denote by \(M_\chi\) the super Riemann surface over \(B\) with embedding \(i_\chi\colon \Red{M}\times B\to M_\chi\) constructed from the triple \((\Red{g}, \Red{S}, \chi)\) on \(\Smooth{M}=\Red{M}\times B\).
The space of maps from \(M_\chi\) to \(N\) parametrized by \(\chi\) is given by the functor
\begin{equation}
	\begin{split}
		\underline{\Fieldspace}_{\mathcal{X}}\colon \cat{SPoint}^{op}&\to \cat{Set} \\
		C &\mapsto \Set{(\chi, \Phi_\chi)\given \chi\in\mathcal{X}(C)\text{, }\Phi_\chi\in\Hom_C(M_\chi, N)}
	\end{split}
\end{equation}
Using the exponential map as before allows to equip \(\underline{\Fieldspace}_{\mathcal{X}}\) with the structure of an infinite-dimensional supermanifold with surjection to \(\underline{\mathcal{X}}\).
The projection \(\underline{\Fieldspace}_{\mathcal{X}}\to \underline{\mathcal{X}}\) is given by \((\chi, \Phi_\chi)\mapsto \chi\).

To equip \(\underline{\Fieldspace}_{\mathcal{X}}\) with an atlas, notice that any map \(\phi\colon\Red{M}\to N\) determines for every \(C\in\cat{SPoint}\) and every \(\chi\in\underline{\mathcal{X}}(C)\) a map \(\Phi_{C, \chi}\colon M_\chi\to N\) whose component fields are given by \(\varphi_{C, \chi}=\phi\times\id_C\), \(\psi_{C,\chi}=0\) and \(F_{C, \chi}=0\).
For this map \(\Phi_{C, \chi}\) it holds
\begin{equation}
	\begin{split}
		{\VSec{\Phi_{C,\chi}^*\tangent{N}}}_0
		&\cong \VSec{\phi^*\tangent{N}}\otimes{\left(\cO_C\right)}_0 \oplus \VSec{\Red{\dual{S}}\otimes\phi^*\tangent{N}}\otimes{\left(\cO_C\right)}_1 \oplus \VSec{\phi^*\tangent{N}}{\otimes\left(\cO_C\right)}_0 \\
		&= \underline{W}_\phi(C)
	\end{split}
\end{equation}
We can now again use the exponential map to identify
\begin{equation}
	\begin{split}
		\underline{V}_{\mathcal{X}, \phi}(C)
		&=\Set{(\chi, \tilde{\Phi})\in \underline{\mathcal{X}}(C)\times\Hom_C(M_\chi, N)\given (\phi(\Red{M}), \Red{\tilde{\Phi}}(\Red{M}))\subset V} \\
		&\subset \underline{\Fieldspace}_{\mathcal{X}}(C)
	\end{split}
\end{equation}
with
\begin{equation}
	\underline{\mathcal{X}}(C)\times \underline{U}_{\phi}(C)
	= \Set{(\chi, \xi, \zeta, \sigma)\in\underline{\mathcal{X}}(C)\times\underline{W}_\phi(C)\given \tensor[_0]{\xi}{}\in U}.
\end{equation}
via
\begin{equation}
	\begin{split}
		\underline{v}_{\mathcal{X}, \phi}(C)\colon \underline{V}_{\mathcal{X}, \phi}(C) &\to \underline{\mathcal{X}}(C)\times\underline{U}_\phi(C) \\
		(\chi, \tilde{\Phi}) &\mapsto \left(\chi, {(\pi_N, \exp)}^{-1}\circ(\Phi_{C, \chi},\tilde{\Phi})\right)
	\end{split}
\end{equation}
Analogously to the case without gravitino, we obtain
\begin{prop}
	The functor \(\underline{\Fieldspace}_{\mathcal{X}}\) carries an infinite-dimensional Fréchet supermanifold structure given by the atlas \(\Set{v_{\underline{\mathcal{X}},\phi}\colon \underline{V}_{\mathcal{X}, \phi}\to \underline{\mathcal{X}}\times\underline{U}_\phi}\) such that the projection \(\underline{\Fieldspace}_{\mathcal{X}}\to \underline{\mathcal{X}}\) is a surjection.
\end{prop}

In general, the mapping space \(\underline{\Fieldspace}_{\mathcal{X}}\) is not a global product, because the maps \(\Phi_{C, \chi}\) and hence also the coordinate changes \(\underline{\mathcal{X}}\times\underline{U}_\phi\to\underline{\mathcal{X}}\times\underline{U}_{\phi'}\) depend on the gravitino.
For a given \(\chi\in\underline{\mathcal{X}}(B)\) we set
\begin{equation}
	\begin{split}
		\underline{\Fieldspace}_\chi\colon \cat{SPoint}^{op}_B&\to \cat{Man} \\
		\left(b\colon C\to B\right)&\mapsto \Set{(b^*\chi, \Phi)\given \Phi\colon M_{b^*\chi}\to N}
	\end{split}
\end{equation}
Then \(\underline{\Fieldspace}_\chi\) is the infinite-dimensional supermanifold over \(B\) of maps \(M_\chi\to N\).

\subsection{The vector bundle of antiholomorphic forms}
In this section we want to construct a vector bundle \(\underline{\Targetbundle}_{\mathcal{X}}\to\underline{\Fieldspace}_{\mathcal{X}}\) such that the fiber over \((\chi, \Phi)\in\underline{\Fieldspace}_{\mathcal{X}}(C)\) is given by the even complex anti-linear differential forms on \(M_\chi\) with values in~\(\Phi^*\tangent{N}\):
\begin{equation}
	\begin{split}
		\underline{\Targetbundle}_{\mathcal{X}}(C)|_{(\chi, \Phi)}
		&= {\VSec{\dual{\cD}\otimes \Phi^*\tangent{N}}}^{0,1}_0. \\
	\end{split}
\end{equation}
We are interested in the vector bundle \(\underline{\Targetbundle}_{\mathcal{X}}\) because \(\DJBar\Phi\) induces a section \(\underline{\mathcal{S}}\colon \underline{\Fieldspace}_{\mathcal{X}}\to \underline{\Targetbundle}_{\mathcal{X}}\).

To see that \(\underline{\Targetbundle}_{\mathcal{X}}\) is indeed a vector bundle over \(\underline{\Fieldspace}_{\mathcal{X}}\), we will give a trivialization above the open subfunctors \(\underline{V}_{\mathcal{X},\phi}\subset\underline{\Fieldspace}_{\mathcal{X}}\) using parallel transport \(P^{\nablabar}_{\exp_\Phi}\).
Roughly, for every \(X\in \underline{U}_\phi\) the vector bundle isomorphism
\begin{equation}
	\id_{\dual{\cD}}\otimes P^{\nablabar}_{\exp_\Phi tX} \colon \dual{\cD}\otimes \Phi^*\tangent{N} \to\dual{\cD}\otimes {\left(\exp_\Phi X\right)}^*\tangent{N}
\end{equation}
gives a trivialization of \(\underline{\Targetbundle}_{\mathcal{X}}\).
As \(\nablabar\) preserves the almost complex structure \(\targetACI\), this map preserves the \(0,1\)-part of both vector bundles.
In order to make this precise we need to define parallel transport for supermanifolds and properly use the formalism of infinite-dimensional vector bundles for \(\underline{\Targetbundle}_\mathcal{X}\).

Recall that in ordinary, non-super, differential geometry vectors can be transported along a curve, see, for example,~\cite[Volume~1, Chapter II.3]{KN-FDG}:
Let \(E\to N\) be a vector bundle with connection \(\nabla^E\) and \(c\colon [0,1]\to N\) a curve in \(N\).
For every vector \(X\in E_{c(0)}\) the theory of ordinary differential equations guarantees the existence of a time-dependent vector field \(X_t\) such that
\begin{align}
	\nabla_{\partial_t}^{c^*E} X_t &=  0 &
	X_0 &= X
\end{align}
The vector \(X_1\in E_{c(1)}\) is called parallel transport of \(X\) along \(c\).
We will use the notation \(P_c^\nabla X = X_1\).
Analogously, we obtain for supermanifolds:
\begin{prop}
	Let \(f_\bullet\colon M\times [0,1]\to N\) be a time-indexed family of maps from a supermanifold \(M\) to the supermanifold \(N\).
	For every connection \(\nabla^E\) on a vector bundle \(E\to N\) and every vector field \(X\in\VSec{f_0^*E}\) there is a unique vector field \(X_\bullet\in\VSec{f_\bullet^*E}\) such that
	\begin{align}
		\nabla_{\partial_t}^E X_\bullet &=0 &
		X_0 = X
	\end{align}
	We call \(P^{\nabla^E}_{f_\bullet} X = X_1\in \VSec{f_1^*E}\) the parallel transport of \(X\) with respect to \(\nabla\) along \(f_\bullet\).

	The parallel transport \(P^{\nabla^E}_{f_\bullet}\colon f_0^*E\to f_1^*E\) is an isomorphism of vector bundles over~\(M\).
	If the connection \(\nabla^E\) preserves a \(G\)-structure on \(E\) the parallel transport preserves the \(G\)-structure.
\end{prop}

For the precise definition of super vector bundles in the categorical reformulation of supermanifolds we refer to~\cite[Chapter~3.6]{S-GAASTS}.
Roughly, a functor \(\underline{E}\colon \cat{SPoint}^{op}\to \cat{VBun}\) with values in the category of vector bundles is a super vector bundle with typical fiber~\(F\) if
\begin{itemize}
	\item
		\(F\) is a super vector space and \(\underline{F}(C)={\left(F\otimes\cO_C\right)}_0\).
	\item
		\(\underline{E}(C)\) is a vector bundle over \(\underline{M}(C)\) with typical fiber \(\underline{F}(C)\) for every \(C\in\cat{SPoint}\).
	\item
		\(\underline{E}\) can be covered by a set of open subfunctors of the form \(\underline{U}\times\underline{F}\) where \(\underline{U}\) is an open subfunctor of \(\underline{M}\) with fiberwise linear glueing over \(\underline{M}\).
\end{itemize}

The vector bundle \(\underline{\Targetbundle}_{\mathcal{X}}\) is given by the functor
\begin{equation}
	\begin{split}
		\underline{\Targetbundle}_{\mathcal{X}}\colon \cat{SPoint}^{op}&\to \cat{Set} \\
		C &\mapsto \underline{\Targetbundle}_{\mathcal{X}}(C)
	\end{split}
\end{equation}
where
\begin{equation}
	\underline{\Targetbundle}_{\mathcal{X}}(C)
	= \Set{(\chi, \tilde{\Phi}, \Xi)\in \underline{\mathcal{X}}(C)\times \Hom_C(M_\chi, N)\times {\VSec{\dual{\cD}\otimes\tilde{\Phi}^*\tangent{N}}}^{0,1}_0}.
\end{equation}
It is obvious that there is a functor morphism \(\underline{\pi}\colon \underline{\Targetbundle}_{\mathcal{X}}\to\underline{\Fieldspace}_{\mathcal{X}}\), sending \((\chi, \tilde{\Phi}, \Xi)\) to \((\chi,\tilde{\Phi})\).
For \(\phi\colon \Red{M}\to N\) and \(\Phi_{C,\chi}\colon M_\chi\to N\) as in the previous section, the preimage
\begin{equation}
	{\underline{\pi}(C)}^{-1}(\underline{V}_{\mathcal{X}, \phi}(C))
	= \Set{(\chi, \tilde{\Phi}, \Xi)\given (\chi, \tilde{\Phi})\in \underline{V}_{\mathcal{X}, \phi}(C), \Xi\in{\VSec{\dual{\cD}\otimes\tilde{\Phi}^*\tangent{N}}}^{0,1}_0}
\end{equation}
can be mapped bijectively
\begin{equation}
	\begin{split}
		\underline{e}_{\mathcal{X}, \phi}(C)\colon {\underline{\pi}(C)}^{-1}(\underline{V}_{\mathcal{X}, \phi}(C)) &\to \underline{\mathcal{X}}(C)\times\underline{U}_\phi(C)\times{\VSec{\dual{\cD}\otimes\Phi_{C,\chi}^*\tangent{N}}}_0^{0,1} \\
		(\chi, \tilde{\Phi}, \Xi) &\mapsto \left(\underline{v}_{\mathcal{X},\phi}(C)(\chi, \tilde{\Phi}), {\left(\id_{\dual{\cD}}\otimes P^{\nablabar}_{\exp_{\Phi_{C,\chi}}tX}\right)}^{-1}\Xi\right)
	\end{split}
\end{equation}
using the map \(\underline{v}_{\mathcal{X},\phi}(C)\) and parallel transport.

In order to obtain a Fréchet vector bundle  over the Fréchet supermanifold \(\underline{\Fieldspace}\) we have to give the structure of a Fréchet vector space to \({\VSec{\dual{\cD}\otimes\Phi_{C,\chi}^*\tangent{N}}}_0^{0,1}\).
This is done analogously to the case of \({\VSec{\Phi_{C,\chi}^*\tangent{N}}}_0\) and using Lemma~\ref{lemma:DefinitionComponentEquations}, which gives \({\left(\cO_C\right)}_0\)-linear isomorphisms
\begin{equation}
	\begin{split}
		\MoveEqLeft
		{\VSec{\dual{\cD}\otimes\Phi_{C,\chi}^*\tangent{N}}}_0^{0,1} \\
		\cong{}& \VSec{\dual{S}\otimes\varphi_C^*\tangent{N}}_0^{0,1}\oplus\VSec{\varphi_C^*\tangent{N}}_0\oplus\VSec{\cotangent{\Smooth{M}}\otimes\varphi_C^*\tangent{N}}_0^{0,1}\oplus\VSec{\dual{S}\otimes\varphi_C^*\tangent{N}}_0^{0,1} \\
		\cong{}& \left(
		\VSec{\dual{\Red{S}}\otimes\phi^*\tangent{N}}^{0,1}\oplus\VSec{\phi^*\tangent{N}}\oplus\VSec{\cotangent{\Red{M}}\otimes\phi^*\tangent{N}}^{0,1} \right.\\
		& \left.\oplus\VSec{\dual{\Red{S}}\otimes\phi^*\tangent{N}}^{0,1}\otimes\cO_C\right)_0 \\
		\cong{}& \VSec{\dual{\Red{S}}\otimes\phi^*\tangent{N}}^{0,1}\otimes{\left(\cO_C\right)}_1 \oplus \left(\VSec{\phi^*\tangent{N}}\oplus\VSec{\cotangent{\Red{M}}\otimes\phi^*\tangent{N}}^{0,1}\right)\otimes{\left(\cO_C\right)}_0 \\
		& \oplus\VSec{\dual{\Red{S}}\otimes\phi^*\tangent{N}}^{0,1}\otimes{\left(\cO_C\right)}_1.
	\end{split}
\end{equation}
We denote the space after the last isomorphism by \(\underline{F}_\phi(C)\).
Notice how the decomposition into component fields allows to identify \({\VSec{\dual{\cD}\otimes\Phi_{C, \chi}^*\tangent{N}}}_0^{0,1}\) and \({\VSec{\dual{\cD}\otimes\Phi_{C, \chi'}^*\tangent{N}}}_0^{0,1}\) for different gravitinos \(\chi\) and \(\chi'\).
The spaces \(\underline{F}_\phi(C)\) can be equipped with the Fréchet structure induced by the \(C^k\)-seminorms and consequently \(\underline{F}_\phi\) is an infinite-dimensional Fréchet-superdomain.
Hence, \(\underline{\Targetbundle}_\mathcal{X}\) is locally isomorphic to \(\underline{\mathcal{X}}\times\underline{U}_\phi\times\underline{F}_\phi\) where \(\underline{\pi}\colon \underline{\Targetbundle}_{\mathcal{X}}\to \underline{\Fieldspace}_{\mathcal{X}}\) is locally the projection on \(\underline{\mathcal{X}}\times\underline{U}_\phi\).
For different maps \(\phi\), \(\phi'\) the change of trivialization
\begin{equation}
	\underline{\mathcal{X}}(C)\times\underline{U}_\phi(C)\times \underline{F}_\phi(C)\to \underline{\mathcal{X}}(C)\times\underline{U}_{\phi'}(C)\times\underline{F}_{\phi'}(C)
\end{equation}
is fiberwise \(\cO_C\)-linear because the parallel transport is \(\cO_C\)-linear.
Consequently:
\begin{prop}
	\(\underline{\Targetbundle}_{\mathcal{X}}\to\underline{\Fieldspace}_{\mathcal{X}}\) is a super vector bundle of infinite rank over \(\underline{\Fieldspace}_{\mathcal{X}}\) with atlas
	\begin{equation}
		\Set{\underline{e}_{\mathcal{X}, \phi}\colon \underline{\pi}^{-1}\left(\underline{V}_{\mathcal{X}, \phi}\right)\to \underline{\mathcal{X}}\times \underline{U}_\phi\times \underline{F}_\phi}.
	\end{equation}
\end{prop}

The maps \(\underline{\mathcal{S}}(C)\colon (\chi, \Phi)\mapsto (\chi, \Phi, \DJBar\Phi)\) yield a functor morphism \(\underline{\mathcal{S}}\colon\underline{\Fieldspace}_{\mathcal{X}}\to\underline{\Targetbundle}_{\mathcal{X}}\).
It satisfies \(\underline{\pi}\circ\underline{\mathcal{S}}=\id_{\underline{\Fieldspace}_{\mathcal{X}}}\), that is, \(\underline{\mathcal{S}}\) is a section of \(\underline{\Targetbundle}_{\mathcal{X}}\).
\begin{prop}\label{prop:SmoothSection}
	The functor morphism \(\underline{\mathcal{S}}\colon \underline{\Fieldspace}_{\mathcal{X}}\to \underline{\Targetbundle}_{\mathcal{X}}\) is a smooth section.
\end{prop}
\begin{proof}
	It only remains to show that \(\underline{\mathcal{S}}\) is smooth.
	Locally the section \(\underline{\mathcal{S}}\) is given by a map \(\underline{\mathcal{F}}_\phi(C)\colon \underline{\mathcal{X}}(C)\times\underline{U}_\phi(C)\to \underline{F}_\phi(C)\).
	Let \(X\in\VSec{\Phi_{C, \chi}^*\tangent{N}}\) be the vector field determined by the component fields \((\xi, \zeta, \sigma)\in\underline{U}_\phi(C)\) with respect to the embedding \(i_\chi\colon \Red{M}\times C\to M_\chi\).
	Then \(\underline{\mathcal{F}}_\phi(C)\) maps the tuple \((\chi, \xi, \zeta, \sigma)\) to the component fields of
	\begin{equation}
		{\left(\id_{\dual{\cD}}\otimes P_{\exp_{\Phi_{C, \chi}} tX}^{\nablabar}\right)}^{-1}\DJBar \exp_{\Phi_{C, \chi}} X.
	\end{equation}
	The parallel transport along \(\exp_{\Phi_{C, \chi}} X\) is \(\cO_M\)-linear and depends smoothly on \(X\).
	Hence there exist for every vector field \(Y\) smooth linear maps
	\begin{equation}
		\nablabar P(Y)\colon \VSec{\Phi_{C, \chi}^*\tangent{N}}\to \VSec{{\left(\exp_{\Phi_{C, \chi}} X\right)}^*\tangent{N}}
	\end{equation}
	such that
	\begin{equation}
		\nablabar_Y \left(P_{\exp_{\Phi_{C, \chi}} tX} Z\right) = P_{\exp_{\Phi_{C, \chi}} tX}\left(\nabla_Y Z\right) + \left(\nablabar P(\nabla_Y X)\right) Z.
	\end{equation}
	The dependence of \(\nablabar P(Y)\) on \(Y\) is smooth.
	We denote by \(\nablabar P^{-1}(Y)\) the derivative of \({\left(P_{\exp_{\Phi_{C, \chi}} tX}^{\nablabar}\right)}^{-1}\) and higher derivatives by \(\nablabar^k P^{-1}(Y)\).
	Then, \(\underline{\mathcal{F}}_\phi(C)(\chi, \xi, \zeta, \sigma)\) is given by the following component fields
	\begin{align}
		\MoveEqLeftR
		i_\chi^*{\left(\id_{\dual{\cD}}\otimes P_{\exp_{\Phi_{C, \chi}} tX}^{\nablabar}\right)}^{-1}\DJBar \exp_{\Phi_{C, \chi}} X
		= {\left(\id_{\dual{S}}\otimes P_{\exp_{\varphi_{C, \chi}} ti_\chi^*X}^{\nablabar}\right)}^{-1}i_\chi^*\DJBar\exp_{\Phi_{C, \chi}} X \\
		\begin{split}
			\MoveEqLeftR
			\frac12\Tr_{\dual{g}_S}i_\chi^*\nablabar^{\cotangent{M}\otimes\Phi_{C,\chi}^*\tangent{N}}{\left(\id_{\dual{\cD}}\otimes P_{\exp_{\Phi_{C, \chi}} tX}^{\nablabar}\right)}^{-1}\DJBar \exp_{\Phi_{C, \chi}} X \\
			={}& \frac12 {\left(P_{\exp_{\varphi_{C, \chi}} ti_\chi^*X}^{\nablabar}\right)}^{-1}i_\chi^*\Tr_{\dual{g}_S}i_\chi^*\nablabar^{\cotangent{M}\otimes\Phi_\chi^*\tangent{N}}\DJBar\exp_{\Phi_{C, \chi}} X \\
				&+ \frac14\varepsilon^{\alpha\beta}i_\chi^*\left(\nablabar P^{-1}(\nabla_{F_\alpha} X)\right)i_\chi^*\left<F_\beta, \DJBar\exp_{\Phi_{C, \chi}}X\right>
		\end{split} \\
		\begin{split}
			\MoveEqLeftR
			\frac12\Tr_{\dual{g}_S}\left(\id_{\dual{S}}\otimes\gamma\otimes\id_{\varphi_{C,\chi}^*\tangent{N}}\right)i_\chi^*\nablabar^{\cotangent{M}\otimes\Phi_{C,\chi}^*\tangent{N}}{\left(\id_{\dual{\cD}}\otimes P_{\exp_{\Phi_{C, \chi}} tX}^{\nablabar}\right)}^{-1}\DJBar \exp_{\Phi_{C, \chi}} X \\
			={}& \frac12 \left(\id_{\cotangent{\Smooth{M}}}\otimes{\left(P_{\exp_{\varphi_{C, \chi}} ti_\chi^*X}^{\nablabar}\right)}^{-1}\right) \\
				&\cdot i_\chi^*\Tr_{\dual{g}_S}\left(\id_{\dual{S}}\otimes\gamma\otimes\id_{\varphi_{C,\chi}^*\tangent{N}}\right)i_\chi^*\nablabar^{\cotangent{M}\otimes\Phi_{C,\chi}^*\tangent{N}}\DJBar\exp_{\Phi_{C, \chi}} X \\
				&+ f^k\otimes \frac14\varepsilon^{\alpha\mu}\tensor{\gamma}{_k_\mu^\beta}i_\chi^*\left(\nablabar P^{-1}(\nabla_{F_\alpha} X)\right)i_\chi^*\left<F_\beta, \DJBar\exp_{\Phi_{C, \chi}}X\right>
		\end{split} \displaybreak[0]\\
		\begin{split}
			\MoveEqLeftR
			-\frac12i_\chi^*\DBarLaplace{\left(\id_{\dual{\cD}}\otimes P_{\exp_{\Phi_{C, \chi}} tX}^{\nablabar}\right)}^{-1}\DJBar \exp_{\Phi_{C, \chi}} X \\
			={}& -\frac12 {\left(\id_{\dual{S}}\otimes P_{\exp_{\varphi_{C, \chi}} ti_\chi^*X}^{\nablabar}\right)}^{-1}i_\chi^*\DBarLaplace\DJBar\exp_{\Phi_{C, \chi}} X \\
				&+ \frac14\varepsilon^{\alpha\beta}i_\chi^*\left(\id_{\dual{S}}\otimes \nablabar P^{-1}(\nabla_{F_\alpha} X)\right)i_\chi^*\nablabar_{F_\beta} \DJBar\exp_{\Phi_{C, \chi}}X \\
				&+ \frac14\varepsilon^{\alpha\beta}i_\chi^*\left(\id_{\dual{S}}\otimes \nablabar^2 P^{-1}(\DLaplace X)\right)i_\chi^* \DJBar\exp_{\Phi_{C, \chi}}X
		\end{split}
	\end{align}
	Component fields of \(\exp_{\Phi_{C, \chi}} X\) can be obtained from the component fields of \(X\) and smooth derivatives of the exponential maps.
	Together with the expressions in Proposition~\ref{prop:ComponentsOfDJBar} for the components of \(\DJBar\) and the derivatives of \(P_{\exp_{\Phi_{C, \chi}} tX}\) the above equations yield expressions for \(\underline{\mathcal{F}}_\phi(C)(\chi, \xi, \zeta, \sigma)\).

	In order to obtain the coefficients of in the expansion of the component fields in terms of the generators \(\lambda^\kappa\) of \(C\) one needs to take further derivatives in direction of \(\lambda^\kappa\).
	The resulting expressions are of the same form:
	a combination of differentials of the smooth maps \(\exp_{\Phi_{C, \chi}}\) and \(P_{\exp_{\Phi_{C, \chi}} tX}^{\nablabar}\) and at most one differential operator in directions tangential to \(\Red{M}\) of order one.
	Consequently, the resulting expressions for \(\underline{\mathcal{F}}_\phi(C)\) are smooth.
\end{proof}
 
\section{The moduli space of super \texorpdfstring{\(\targetACI\)}{J}-holomorphic curves}\label{Sec:ModuliSpace}
In this section we fix a closed compact Riemann surface \((\Red{M}, \Red{g})\) with spinor bundle~\(\Red{S}\) and an almost Kähler manifold \((N, \omega, \targetACI)\).
Recall that for any gravitino \(\chi\in\underline{\mathcal{X}}(C)\) the triple \((\Red{M}, \Red{g}, \chi)\) determines a super Riemann surface \(M_\chi\) over \(C\) together with an embedding \(i_\chi\colon \Red{M}\times C\to M_\chi\) which is the identity on the topological spaces.
The super \(\targetACI\)-holomorphic curves from \(M_\chi\) to \(N\) form a subset of \(\underline{\Fieldspace}_{\mathcal{X}}(C)\).
This leads to the following definition
\begin{defn}
	Let \(\mathcal{X}\subset\VSec{\cotangent{\Red{M}}\otimes \Red{S}}\) and \(A\in H_2(N)=H_2(N, \R)\).
	We call the subfunctor
	\begin{equation}
		\begin{split}
			\underline{\mathcal{M}(\mathcal{X}, A)}\colon \cat{SPoint}^{op}&\to \cat{Set} \\
			C & \mapsto \Set{(\chi, \Phi)\in \underline{\mathcal{X}}(C)\times\Hom_C(M_\chi, N)\given \DJBar\Phi=0 \text{ and } [\im\Red{\Phi}] = A}
		\end{split}
	\end{equation}
	of \(\underline{\Fieldspace}_{\mathcal{X}}\) the moduli space of \(\targetACI\)-holomorphic curves with fixed homology class \(A\) of the image of \(\Red{\Phi}\colon\Red{M}\to N\).
\end{defn}
Notice that \(\underline{\mathcal{M}(\mathcal{X}, A)}(\R)\) is the moduli space of \(\targetACI\)-holomorphic curves from \((\Red{M}, \Red{g})\) to the almost Kähler manifold \((N, \omega, \targetACI)\) whose homology class is given by \(A\).
We will  show that under certain conditions on the target \((N, \omega, \targetACI)\) the moduli space \(\underline{\mathcal{M}(\mathcal{X}, A)}\) is a finite dimensional supermanifold.
The proof proceeds by extending the construction of the moduli space of \(\targetACI\)-holomorphic curves \(\phi\colon \Red{M}\to N\) as presented, for example, in~\cite{McDS-JHCST} to super Riemann surfaces.
We use that \(\underline{\mathcal{M}(\mathcal{X}, A)}\) is the zero locus of the section \(\underline{\mathcal{S}}\) induced by \(\DJBar\).
Intuitively, if \(\underline{\mathcal{S}}\) is transversal to the zero section at a map~\(\Phi\), the preimage of zero should inherit a supermanifold structure in the neighbourhood of \(\Phi\) from the one on \(\underline{\Fieldspace}_{\mathcal{X}}\).
Using the component field approach, we will show that transversality of \(\underline{\mathcal{S}}\) is equivalent to the surjectivity of the differential operators~\(D_\phi\) and \(\Red{\Dirac}^{1,0}\) along \(\phi=\Red{\Phi}\), which can be interpreted as a condition on the target \((N, \omega, \targetACI)\).
Completing the charts of \(\underline{\Fieldspace}_{\mathcal{X}}\) allows then to calculate the dimension of the moduli space applying an index theorem to \(D_\phi\) and \(\Red{\Dirac}^{1,0}\).
In this Sobolev completion we then apply an implicit function theorem to a local expression of \(\underline{\mathcal{S}}\) and show smoothness of the solutions using elliptic regularity theory.

\subsection{Transversality}
In this Section we want to reduce the question whether \(\underline{\mathcal{S}}\) is transversal to the zero-section to an analytical condition for the almost Kähler target \((N, \omega, \targetACI)\).
To this end we calculate \(\tangent{\underline{\mathcal{S}}}(C)\) at points \(\Phi_C\), where \(\Phi_C\) is the extension of a \(\targetACI\)-holomorphic curve \(\phi\colon \Red{M}\to N\) to the super Riemann surface \(M\) determined by \((\Red{g}, \Red{S}, \chi=0)\).

The zero-section of \(\underline{\Targetbundle}_{\mathcal{X}}\) is the smooth functor morphism \(\underline{\mathcal{Z}}\colon \underline{\Fieldspace}_{\mathcal{X}}\to \underline{\Targetbundle}_{\mathcal{X}}\) given by the maps \(\underline{\mathcal{Z}}(C)\colon (\chi, \tilde{\Phi})\mapsto (\chi, \tilde{\Phi}, 0)\).
For every \(C\in\cat{SPoint}\) we have maps
\begin{align}
	\tangent{\underline{\mathcal{S}}(C)}\colon \tangent{\left(\underline{\Fieldspace}_{\mathcal{X}}(C)\right)}&\to \tangent{\left(\underline{\Targetbundle}_{\mathcal{X}}(C)\right)}, &
	\tangent{\underline{\mathcal{Z}}(C)}\colon \tangent{\left(\underline{\Fieldspace}_{\mathcal{X}}(C)\right)} &\to \tangent{\left(\underline{\Targetbundle}_{\mathcal{X}}(C)\right)}.
\end{align}
We say that \(\underline{\mathcal{S}}\) is transversal to the zero section at the \(\targetACI\)-holomorphic curve \(\phi\colon \Red{M}\to N\) if for all \(C\) the union of the subspaces
\begin{align}
	&\im \tangent[(0,\Phi_C)]{\underline{\mathcal{S}}(C)} &
	&\im \tangent[(0,\Phi_C)]{\underline{\mathcal{Z}}(C)}
\end{align}
generate \(\tangent[(0,\Phi_C, 0)]{\left(\underline{\Targetbundle}_{\mathcal{X}}(C)\right)}\).

We use the charts \(\underline{\mathcal{X}}\times\underline{U}_\phi\) of \(\underline{\Fieldspace}_{\mathcal{X}}\) and the trivialization \(\underline{\mathcal{X}}\times\underline{U}_\phi\times\underline{F}_\phi\) of \(\underline{\Targetbundle}_{\mathcal{X}}\).
Recall that
\begin{equation}
	\underline{U}_\phi(C)
	\subset \underline{W}_\phi(C)
	= \VSec{\phi^*\tangent{N}}{\otimes\left(\cO_C\right)}_0 \oplus \VSec{\Red{\dual{S}}\otimes\phi^*\tangent{N}}\otimes{\left(\cO_C\right)}_1\oplus\VSec{\phi^*\tangent{N}}{\otimes\left(\cO_C\right)}_0
\end{equation}
and
\begin{equation}
	\begin{split}
		\underline{F}_\phi(C)
		={}& \VSec{\dual{\Red{S}}\otimes\phi^*\tangent{N}}^{0,1}\otimes{\left(\cO_C\right)}_1 \oplus \VSec{\phi^*\tangent{N}}\otimes{\left(\cO_C\right)}_0 \\
			& \oplus\VSec{\cotangent{\Red{M}}\otimes\phi^*\tangent{N}}^{0,1}\otimes{\left(\cO_C\right)}_0\oplus\VSec{\dual{\Red{S}}\otimes\phi^*\tangent{N}}^{0,1}\otimes{\left(\cO_C\right)}_1.
	\end{split}
\end{equation}
The section \(\underline{\mathcal{S}}\) is locally given by \(\underline{\mathcal{F}}_\phi\colon \underline{\mathcal{X}}\times\underline{U}_\phi\to \underline{F}_\phi\).
We denote the derivative of \(\underline{\mathcal{F}}_\phi(C)\) at \((0,0,0,0)\) by
\begin{equation}
	\underline{\mathcal{F}}'_\phi(C)\colon \underline{\mathcal{X}}(C)\times\underline{W}_\phi(C) \to \underline{F}_\phi(C).
\end{equation}
Consequently, \(\underline{\mathcal{S}}\) is transversal to the zero-section at the \(\targetACI\)-holomorphic curve \(\phi\colon \Red{M}\to N\) if and only if \(\underline{\mathcal{F}}'_\phi(C)\) is surjective for all \(C\).

In order to calculate \(\underline{\mathcal{F}}'_\phi(C)\) we need some more notation.
For \(\varphi_C=\phi\times\id_C\colon \Smooth{M}\to N\) let us define
\begin{equation}
	\begin{split}
		D_{\varphi_C}\colon \VSec{\varphi_C^*\tangent{N}}&\to {\VSec{\cotangent{\Smooth{M}}\otimes\varphi_C^*\tangent{N}}}^{0,1} \\
		\xi&\mapsto\frac12\left(1+\ACI\otimes\targetACI\right)\left(\nabla\xi - \frac12\id_{\cotangent{\Smooth{M}}}\otimes\left(\targetACI\left(\nabla_\xi \targetACI\right)\right)\differential\varphi_C\right)
	\end{split}
\end{equation}
Notice that \(D_{\varphi_C}\) is the linearisation of \(\DelJBar\) at \(\varphi_C\), see~\cite[Proposition~3.1.1]{McDS-JHCST}.

For \(\zeta\in\VSec{\dual{S}\otimes\varphi_C^*\tangent{N}}\) let
\begin{align}
	\zeta^{1,0}&=\frac12\left(1-\ACI\otimes\targetACI\right)\zeta
	\in\VSec{\dual{S}\otimes_\C\varphi_C^*\tangent{N}}, \\
	\zeta^{0,1}&=\frac12\left(1+\ACI\otimes\targetACI\right)\zeta
	\in{\VSec{\dual{S}\otimes\varphi_C^*\tangent{N}}}^{0,1}.
\end{align}
Now define
\begin{align}
	\Dirac^{1,0}\colon \VSec{\dual{S}\otimes_\C\varphi_C^*\tangent{N}} &\to {\VSec{\dual{S}\otimes\varphi_C^*\tangent{N}}}^{0,1} \\
	\zeta^{1,0} &\mapsto\frac12\left(1 + \ACI\otimes\targetACI\right)\Dirac\zeta^{1,0} \\
	\Dirac^{0,1}\colon {\VSec{\dual{S}\otimes\varphi_C^*\tangent{N}}}^{0,1} &\to \VSec{\dual{S}\otimes_\C\varphi_C^*\tangent{N}} \\
	\zeta^{0,1} &\mapsto \frac12\left(1 - \ACI\otimes\targetACI\right)\Dirac\zeta^{0,1}
\end{align}
By Equation~\eqref{eq:CommuteIJDirac}, we know that
\begin{equation}
\label{eq:DiracOperatorsProjectors}
	\begin{pmatrix}
		{\left(\Dirac\zeta\right)}^{1,0} \\
		{\left(\Dirac\zeta\right)}^{0,1}
	\end{pmatrix}
	=
	\begin{pmatrix}
		\frac12 \ACI\gamma^k\otimes\left(\nabla_k\targetACI\right) & \Dirac^{0,1} \\
		\Dirac^{1,0} & -\frac12\ACI\gamma^k\otimes\left(\nabla_k\targetACI\right)
	\end{pmatrix}
	\begin{pmatrix}
		\zeta^{1,0} \\
		\zeta^{0,1}
	\end{pmatrix}
\end{equation}
where the diagonal terms vanish if \(N\) is Kähler.

\begin{prop}\label{prop:DifferentialOfF}
	The differential of \(\underline{\mathcal{F}}_\phi(C)\) at \((0,0,0,0)\) is given by
	\begin{equation}
		\begin{split}
			\underline{\mathcal{F}}'_\phi(C)\colon \underline{\mathcal{X}}(C)\times\underline{W}_\phi(C) &\to \underline{F}_\phi(C) \\
			(\rho, \xi, \zeta, \sigma) &\mapsto
			\left(\zeta^{0,1},
			\frac14\sigma,
			-D_{\varphi_C}\xi,
			-\hat\Dirac_{\varphi_C}\zeta + 2\left<\vee Q\rho,\differential\varphi_C\right>\right)
		\end{split}
	\end{equation}
	where
	\begin{equation}
		\hat\Dirac_{\varphi_C} \zeta
		= \Dirac^{1,0}\zeta^{1,0} - \frac12\left(\ACI \gamma^k\otimes\left(\nabla_k\targetACI\right)\zeta^{0,1} - \Tr_g\left<\left(\gamma\otimes\id_{\varphi_C^*\tangent{N}}\right)\zeta^{0,1}, \targetACI\varphi_C^*\nabla\targetACI\right>\differential{\varphi_C}\right).
	\end{equation}
\end{prop}
\begin{proof}
	We will treat the summands of
	\begin{equation}
		\underline{\mathcal{F}}'_\phi(C)(\rho, \xi, \zeta, \sigma)
		= \underline{\mathcal{F}}'_\phi(C)(0, \xi, \zeta, \sigma) + \underline{\mathcal{F}}'_\phi(C)(\rho, 0, 0, 0)
	\end{equation}
	separately.
	Both summands can be obtained from variation of the component fields of \(\DJBar\Phi\) obtained in~\ref{prop:ComponentsOfDJBar} with respect to \((\varphi, \psi, F)\) and \(\chi\) respectively.

	For the first summand, let \(M\) be the family of super Riemann surface over \(C\) determined by \((\Red{g}, \Red{S}, \chi=0)\) and \(\Phi_C\colon M\to N\) the map whose component fields are given by \(\varphi_C=\phi\times C\), \(\psi_C=0\) and \(F_C=0\).
	We denote by \(Y\in\VSec{\Phi_C^*\tangent{N}}\) the vector field whose component fields are given by \((\xi, \zeta, \sigma)\in\underline{U}_\phi(C)\) with respect to the embedding \(i\colon \Red{M}\times C\to M_\chi\).
	Furthermore, denote the component fields of \(\Phi_t=\exp_{\Phi_C} tY\) by \(\varphi_t\), \(\psi_t\) and \(F_t\), where \(\Phi_0=\Phi_C\), \(\varphi_0 = \phi\times\id_C\), \(\psi_0=0\), \(F_0=0\).
	Then
	\begin{align}
		\left.\frac{\d}{\d{t}}\right|_{t=0} \varphi_t
		&= \left.\frac{\d}{\d{t}}\right|_{t=0} \left(\exp_{\Phi_C}tY\right)\circ i
		= i^*Y
		= \xi, \\
		\begin{split}
			\left.\nabla^{\dual{S}\otimes\varphi_t^*\tangent{N}}_{\partial_t}\right|_{t=0}\psi_t
			&= s^\alpha\otimes i^*\left(\left.\nabla_{\partial_t}\right|_{t=0}F_\alpha\Phi_t\right)
			= s^\alpha\otimes i^*\left(\left.\nabla_{F_\alpha}\partial_t\left(\exp_{\Phi_C} tY\right)\right|_{t=0}\right) \\
			&= s^\alpha\otimes i^*\nabla_{F_\alpha}Y
			= \zeta,
		\end{split} \\
		\begin{split}
			\left.\nabla_{\partial_t}\right|_{t=0}F_t
			&= \left.\nabla_{\partial_t}\right|_{t=0}i^*\DLaplace \Phi_t
			= \left.\nabla_{\partial_t}\right|_{t=0}i^*\varepsilon^{\alpha\beta}\nabla_{F_\alpha}F_\beta\Phi_t \\
			&= i^*\varepsilon^{\alpha\beta}\nabla_{F_\alpha}\nabla_{F_\beta}\left.\partial_t\left(\exp_{\Phi_C} tY\right)\right|_{t=0}
			= i^*\varepsilon^{\alpha\beta}\nabla_{F_\alpha}\nabla_{F_\beta}Y
			= \sigma.
		\end{split}
	\end{align}
	We have used here that the connection \(\nabla\) on \(N\) is the Levi-Civita connection and \(\psi_0=0\).
	Notice that \(\psi_0=0\) also implies
	\begin{align}
		\left.\nabla^{\dual{S}\otimes\varphi_t^*\tangent{N}}_{\partial_t}\right|_{t=0}\psi^{1,0}_t
		&= \zeta^{1,0}, &
		\left.\nabla^{\dual{S}\otimes\varphi_t^*\tangent{N}}_{\partial_t}\right|_{t=0}\psi^{0,1}_t
		&= \zeta^{0,1}.
	\end{align}
	We will use the same arguments for the summands of \(\underline{\mathcal{F}}'_\phi(C)(0, \xi, \zeta, \sigma)\in\underline{F}_\phi(C)\).
	For the first summand we have:
	\begin{equation}
		\begin{split}
			\MoveEqLeft
			\left.\frac{\d}{\d{t}}\right|_{t=0} i^*{\left(\id_{\dual{\cD}}\otimes P_{\exp_{\Phi_C} tY}^{\nablabar}\right)}^{-1}\DJBar\exp_{\Phi_C} tY
			= \left.i^*\nablabar^{\dual{\cD}\otimes\Phi_t^*\tangent{N}}_{\partial_t}\right|_{t=0}\DJBar\Phi_t \\
			&= \left.\nablabar^{\dual{S}\otimes\varphi_t^*\tangent{N}}_{\partial_t}\right|_{t=0}i^*\DJBar\Phi_t
			= \left.\nabla^{\dual{S}\otimes\varphi_t^*\tangent{N}}_{\partial_t}\right|_{t=0}\psi_t^{0,1}
			= \zeta^{0,1}
		\end{split}
	\end{equation}
	Here we have used~\ref{prop:ComponentsOfDJBar} and furthermore \(\targetACI\left(\nabla_\xi\targetACI\right)\DJBar\Phi_0=0\).
	For the second summand we obtain by the same arguments:
	\begin{equation}
		\begin{split}
			\MoveEqLeft
			\left.\frac{\d}{\d{t}}\right|_{t=0} \frac12\Tr_{\dual{g}_S}i^*\nablabar^{\cotangent{M}\otimes\Phi_C^*\tangent{N}}{\left(\id_{\dual{\cD}}\otimes P_{\exp_{\Phi_C} tY}^{\nablabar}\right)}^{-1}\DJBar \exp_{\Phi_C} tY \\
			&= \frac12\left.\nabla_{\partial_t}\right|_{t=0} \Tr_{\dual{g}_S}i^*\nablabar^{\cotangent{M}\otimes\Phi_C^*\tangent{N}}\DJBar \Phi_t  \\
			&= \left.\nabla_{\partial_t}\right|_{t=0}\left(\frac14F_t + \frac18\Tr_{\dual{g}_S}\left(\id_{\dual{S}}\otimes\mathfrak{j}_t\targetACI + \ACI\otimes\mathfrak{j}_t\right)\psi_t\right) \\
			&= \frac14\sigma
		\end{split}
	\end{equation}
	Here we have used that \(i^*\nablabar_{\partial_t}\nablabar_{F_\alpha} - i^*\nablabar{F_\alpha}\nablabar_{\partial_t} = \overline{R}(\xi, \psi_0) = 0\).
	For the third summand of \(\underline{\mathcal{F}}'_\phi(0, \xi, \zeta, \sigma)\) we obtain:
	\begin{equation}
		\begin{split}
			\MoveEqLeft
			\left.\frac{\d}{\d{t}}\right|_{t=0} \frac12\Tr_{\dual{g}_S}\left(\id_{\dual{S}}\otimes\gamma\otimes\id_{\varphi_C^*\tangent{N}}\right)i^*\nablabar^{\cotangent{M}\otimes\Phi^*\tangent{N}}{\left(\id_{\dual{\cD}}\otimes P_{\exp_{\Phi_C} tY}^{\nablabar}\right)}^{-1}\DJBar \exp_{\Phi_C} tY \\
			&= \frac12\left.\nabla_{\partial_t}\right|_{t=0} \Tr_{\dual{g}_S}\left(\id_{\dual{S}}\otimes\gamma\otimes\id_{\varphi_C^*\tangent{N}}\right)i^*\nablabar^{\cotangent{M}\otimes\Phi^*\tangent{N}}\DJBar \Phi_t \\
			&= \left.\nabla_{\partial_t}\right|_{t=0} \left(-\DelJBar\varphi_t + \frac18\left(1+\ACI\otimes\targetACI\right)\Tr_{\dual{g}_S}\left(\gamma\otimes\mathfrak{j}_t\targetACI\right)\psi_t\right) \\
			&= -D_\varphi \xi
		\end{split}
	\end{equation}
	Notice that for \(\Phi_0=\Phi_C\) the endomorphisms \(\mathfrak{j}_0\) and \(\tensor[_{34}]{\mathfrak{j}}{_0}\) vanish.
	\begin{equation}
		\begin{split}
			\MoveEqLeftR
			\left.\frac{\d}{\d{t}}\right|_{t=0} -\frac12i^*\DBarLaplace{\left(\id_{\dual{\cD}}\otimes P_{\exp_{\Phi_C} tX}^{\nablabar}\right)}^{-1}\DJBar \exp_{\Phi_C} Y \\
			={}& -\frac12\left.\nabla_{\partial_t}\right|_{t=0}i^*\DBarLaplace\DJBar \Phi_t \\
			={}&	-\frac12\left.\nabla_{\partial_t}\right|_{t=0}\left(1 + \ACI\otimes\targetACI\right)\left(\Dirac\psi_t - \frac16SR^N(\psi_t) \right.\\
				&+ \frac12\left(\id_{\dual{S}}\otimes\left(\targetACI\tensor[_{34}]{\mathfrak{j}}{_t} -\frac14\Tr_{\dual{g}_S}\mathfrak{j}_t\mathfrak{j}_t\right)\right)\psi_t \\
				&+ \left.\frac12\Tr_g\left<\left(\gamma\otimes\id_{\varphi_t^*\tangent{N}}\right)\psi_t, \targetACI\varphi_t^*\nabla\targetACI\right>\differential{\varphi_t} - \frac12\mathfrak{j}_t\targetACI F_t\right) \\
				={}& -\frac12\left(1+\ACI\otimes\targetACI\right)\left(\Dirac\zeta + \frac12\Tr_g\left<\left(\gamma\otimes\id_{\varphi_C^*\tangent{N}}\right)\zeta, \targetACI\varphi_C^*\nabla\targetACI\right>\differential{\varphi_C}\right) \\
				={}& -\Dirac^{1,0}\zeta^{1,0} + \frac12\left(\ACI \gamma^k\otimes\left(\nabla_k\targetACI\right)\zeta^{0,1} - \Tr_g\left<\left(\gamma\otimes\id_{\varphi_C^*\tangent{N}}\right)\zeta^{0,1}, \targetACI\varphi_C^*\nabla\targetACI\right>\differential{\varphi_C}\right)
		\end{split}
	\end{equation}

	It remains to calculate \(\underline{\mathcal{F}}'_\phi(C)(\rho, 0, 0, 0)\).
	This is obtained by taking the time derivative of the component fields of \(\DJBar\Phi_{C,\chi}\) for \(\chi=t\rho\).
	Considering \(\psi_{C, \chi}=0\), the only contribution is \(2\left<\vee Q\rho, \differential{\varphi}\right>\) for the fourth component field.
\end{proof}
\begin{rem}
	Instead of deriving the expressions for \(\underline{\mathcal{F}}'_\phi\) by varying the component expressions of \(\DJBar\Phi\), one can also directly compute component expressions of
	\begin{equation}
		\begin{split}
			\MoveEqLeft
			\left.\frac{\d}{\d{t}}\right|_{t=0}{\left(\id_{\dual{\cD}}\otimes P^{\nablabar}_{\exp_{\Phi} tY}\right)}^{-1}\DJBar\exp_{\Phi} tY
			= \left.\nablabar_{\partial_t} \DJBar\exp_{\Phi} tY\right|_{t=0} \\
			&= \left.\frac12\left(1+\ACI\otimes\targetACI\right)\left(\nabla Y - \frac12 \id\otimes \left(\targetACI\left(\nabla_{Y}\targetACI\right)\right)\differential{\Phi}\right)\right|_{\cD}.
		\end{split}
	\end{equation}
	This expression is the linearization of \(\DJBar\) at a super \(\targetACI\)-holomorphic curve \(\Phi\) in analogy to~\cite[Proposition~3.1.1]{McDS-JHCST}.
\end{rem}

\begin{prop}\label{prop:TransversalityIsAnalysis}
	The map \(\underline{\mathcal{F}}_\phi'(C)\) is surjective if the reduced operators
	\begin{equation}
		D_\phi\colon\VSec{\phi^*\tangent{N}}\to {\VSec{\cotangent{\Red{M}}\otimes\phi^*\tangent{N}}}^{0,1}
	\end{equation}
	and
	\begin{equation}
		\Red{\Dirac}^{1,0}\colon \VSec{\dual{\Red{S}}\otimes_\C\phi^*\tangent{N}} \to {\VSec{\dual{\Red{S}}\otimes\phi^*\tangent{N}}}^{0,1}
	\end{equation}
	are surjective.
\end{prop}
\begin{proof}
	The main point here is that for \(\lambda^{\kappa}\) coordinates of \(C\) and
	\begin{align}
		\xi &= \sum_{\underline{\kappa}\text{ even}}\lambda^{\underline{\kappa}}\tensor[_{\underline{\kappa}}]{\xi}{}
		\in\VSec{\phi^*\tangent{N}}\otimes{\left(\cO_C\right)}_0 \\
		\zeta^{1,0} &= \sum_{\underline{\kappa}\text{ odd}}\lambda^{\underline{\kappa}}\tensor[_{\underline{\kappa}}]{\zeta}{^{1,0}}
		\in\VSec{\dual{\Red{S}}\otimes_\C\phi^*\tangent{N}}\otimes{\left(\cO_C\right)}_1
	\end{align}
	we have
	\begin{align}
		D_{\varphi_C}\xi
		&= \sum_{\underline{\kappa}\text{ even}}\lambda^{\underline{\kappa}}D_\phi\tensor[_{\underline{\kappa}}]{\xi}{} &
		\Dirac^{1,0}\zeta^{1,0}
		&= \sum_{\underline{\kappa}\text{ even}}\lambda^{\underline{\kappa}}\Red{\Dirac}^{1,0}\tensor[_{\underline{\kappa}}]{\zeta}{^{1,0}}
	\end{align}
	because \(\varphi_C=\phi\times\id_C\).
	Now let \((t, u, v, w)\in\underline{F}_\phi(C)\).
	The equations
	\begin{equation}\label{eq:ComponentsF'}
		\begin{aligned}
			t ={}& \zeta^{0,1} \\
			u ={}& \frac14\sigma \\
			v ={}& -D_{\varphi_C}\xi \\
			w ={}& -\Dirac^{1,0}\zeta^{1,0} + \frac12\Tr_g\left<\left(\gamma\otimes\id_{\varphi_C^*\tangent{N}}\right)\zeta^{0,1}, \targetACI\varphi_C^*\nabla\targetACI\right>\differential{\varphi_C} \\
				&- \frac12\ACI \gamma^k\otimes\left(\nabla_k\targetACI\right)\zeta^{0,1} + 2\left<\vee Q\rho, \differential\varphi_C\right>
		\end{aligned}
	\end{equation}
	can be solved for \(\xi\), \(\zeta=(\zeta^{1,0},\zeta^{0,1})\) and \(\sigma\) if \(D_\phi\) and \(\Red{\Dirac^{1,0}}\) are surjective.
\end{proof}
It is an important result that for a generic almost complex structure \(\targetACI\) the operator~\(D_\phi\) is surjective for all simple \(\phi\colon \Red{M}\to N\), see~\cite[Theorem~3.1.5(ii)]{McDS-JHCST}.
We expect that similarly \(\Red{\Dirac}^{1,0}\) is surjective for a large open set of maps~\(\phi\) and almost complex structures~\(\targetACI\).
The Dirac operator \(\Red{\Dirac}^{1,0}\) depends on the almost complex structure through the Levi-Civita connection of the target metric \(n(X, Y)=\omega(\targetACI X, Y)\).
In Section~\ref{Sec:BochnersMethod} we will derive conditions for the surjectivity of \(\Red{\Dirac}^{1,0}\).

\subsection{The index of \texorpdfstring{\(D_\phi\)}{Dphi} and \texorpdfstring{\(\Red{\Dirac}^{1,0}\)}{DiracRed1,0}}
In this section we calculate the Fredholm-index of the operators \(D_\phi\) and \(\Red{\Dirac}^{1,0}\) in appropriate Sobolev spaces of sections.
For the analytical details we mainly use the Riemann--Roch Theorem for real linear Cauchy--Riemann operators presented in~\cite[Appendix~C]{McDS-JHCST}.
If the operators \(D_\phi\) and \(\Red{\Dirac}^{1,0}\) are surjective their index corresponds with the dimension of their kernel.
This will allow to calculate the dimension of the moduli space of super \(\targetACI\)-holomorphic curves in the next section.

Recall that a complex linear Cauchy--Riemann operator on a vector bundle \(E\to\Red{M}\) with almost complex structure \(J^E\) is a complex linear operator \(\mathfrak{D}\colon \VSec{E}\to \VSec{{\cotangent{\Red{M}}\otimes E}}^{0,1}\) satisfying the Leibniz rule \(\mathfrak{D}fY = \left(\bar\partial f\right)Y + f \mathfrak{D}Y\) for complex functions \(f\) and \(Y\in\VSec{E}\).
A smooth real linear Cauchy--Riemann operator is an operator of the form \(\mathfrak{D}=\mathfrak{D}_0 + \alpha\), where \(\mathfrak{D}_0\) is a smooth complex linear Cauchy--Riemann operator and \(\alpha\in{\VSec{\cotangent{\Red{M}}\otimes\End(E)}}^{0,1}\).
Operators of the form \(\frac12\left(1+\ACI\otimes J^E\right)\nabla^E\) are always real linear Cauchy--Riemann operators, in particular also the operator
\begin{equation}
	D_\phi\xi = \frac12\left(1+\ACI\otimes\targetACI\right)\left(\nabla\xi - \frac12\id_{\cotangent{\Smooth{M}}}\otimes\left(\targetACI\left(\nabla_\xi \targetACI\right)\right)\differential\phi\right).
\end{equation}
The Riemann--Roch theorem for Cauchy--Riemann operator~\cite[Theorem~C.1.10]{McDS-JHCST} states in particular that for every smooth real linear Cauchy--Riemann operator \(\mathfrak{D}\) the completion
\begin{equation}
	\mathfrak{D}\colon W^{k,q}(E)\to W^{k-1,q}{(\cotangent{\Red{M}}\otimes E)}^{0,1}
\end{equation}
is Fredholm for all integers \(k>0\), \(q>1\) and its real index is given by
\begin{equation}
	\ind \mathfrak{D}= n(2-2p)+ 2\left<c_1(E), [\Red{M}]\right>.
\end{equation}
Here \(p\) denotes the genus of \(\Red{M}\) and \(n\) is the complex rank of \(E\).
Furthermore the kernel of \(\mathfrak{D}\) is independent of \(k\) and \(q\).

This result is applied in~\cite{McDS-JHCST} mainly to
\begin{equation}
	D_\phi\colon W^{k,q}(\Red{\phi}^*\tangent{N})\to W^{k-1, q}{(\cotangent{\Red{M}}\otimes\phi^*\tangent{N})}^{0,1}
\end{equation}
whose index is given by
\begin{equation}
	\ind D_\phi = n(2-2p) + \left<c_1(\tangent{N}), [\im\phi]\right>.
\end{equation}
In order to apply the Riemann--Roch Theorem for real linear Cauchy--Riemann operators also to \(\Red{\Dirac}^{1,0}\), we need the following lemmata:
\begin{lemma}
	For
	\begin{equation}
		\begin{split}
			\delta_\gamma\colon\cotangent{\Red{M}}\otimes \dual{\Red{S}}\otimes_\C\phi^*\tangent{N} &\to {\left(\dual{\Red{S}}\otimes\phi^*\tangent{N}\right)}^{0,1} \\
			\alpha\otimes s\otimes_\C Y &\mapsto \gamma(\alpha)s\otimes Y
		\end{split}
	\end{equation}
	the short exact sequence of real vector bundles
	\begin{diag}\label{seq:twistedSpinors}
		\matrix[mat](m)
		{
			0 & \cotangent{\Red{M}}\otimes_\C\dual{\Red{S}}\otimes_\C\phi^*\tangent{N} & \cotangent{\Red{M}}\otimes \dual{\Red{S}}\otimes_\C\phi^*\tangent{N} &\\
				& \vphantom{0} & {\left(\dual{\Red{S}}\otimes\phi^*\tangent{N}\right)}^{0,1} & 0 \\
		} ;
		\path[pf]
		{
			(m-1-1) edge (m-1-2)
			(m-1-2) edge (m-1-3)
			(m-2-2) edge node[auto]{\(\delta_\gamma\)} (m-2-3)
			(m-2-3) edge (m-2-4)
		} ;
	\end{diag}
	is split by
	\begin{equation}
		{\ker\delta_\gamma}^\perp
		= {\left(\cotangent{\Red{M}}\otimes\dual{\Red{S}}\otimes_\C\phi^*\tangent{N}\right)}^{0,1}.
	\end{equation}
\end{lemma}
\begin{proof}
	Notice that \(\delta_\gamma\) is indeed complex antilinear and hence
	\begin{equation}
		\delta_\gamma(\alpha\otimes\left(s \otimes Y - \ACI s\otimes \targetACI Y\right))
		= \left(1+\ACI\otimes\targetACI\right)\left(\gamma(\alpha)s\otimes Y\right).
	\end{equation}
	The sequence~\eqref{seq:twistedSpinors} is split by the complex anti-linear map
	\begin{equation}
		\begin{split}
			\frac12\gamma\colon {\left(\dual{\Red{S}}\otimes\phi^*\tangent{N}\right)}^{0,1} &\to \cotangent{\Red{M}}\otimes\dual{\Red{S}}\otimes_\C\phi^*\tangent{N} \\
			s^\alpha\otimes\tensor[_\alpha]{\psi}{} &\mapsto -\frac12f^k\otimes s^\beta\otimes \tensor{\gamma}{_k_\beta^\alpha}\tensor[_\alpha]{\psi}{}
			\qedhere
		\end{split}
	\end{equation}
\end{proof}

\begin{lemma}
	Denote by \(J^{\dual{\Red{S}}\otimes_\C\phi^*\tangent{N}}=\ACI\otimes_\C\id_{\phi^*\tangent{N}}=\id_{\dual{\Red{S}}}\otimes_\C\targetACI\) the almost complex structure on \(\dual{\Red{S}}\otimes_\C\phi^*\tangent{N}\) and by \(\nablabar^{\dual{\Red{S}}\otimes_\C\phi^*\tangent{N}}=\nabla^{\dual{\Red{S}}}\otimes_\C\nablabar\) the induced almost complex connection.
	Then for every \(\zeta^{1,0}\in\VSec{\dual{\Red{S}}\otimes\phi^*\tangent{N}}\)
	\begin{equation}
		\Red{\Dirac}^{1,0} \zeta^{1,0}
		= \delta_\gamma \circ \DelJBar^{\nablabar^{\dual{\Red{S}}\otimes_\C\phi^*\tangent{N}}} \zeta^{1,0}
		= \delta_\gamma\left(\frac12\left(1+\ACI\otimes J^{\dual{\Red{S}}\otimes\phi^*\tangent{N}}\right)\nablabar^{\dual{\Red{S}}\otimes\varphi^*\tangent{N}}\zeta^{1,0}\right).
	\end{equation}
\end{lemma}
\begin{proof}
By straightforward calculation
\begin{equation}
	\begin{split}
		\MoveEqLeft
		\DelJBar^{\nablabar^{\dual{\Red{S}}\otimes\phi^*\tangent{N}}} \zeta^{1,0}
		= \frac12\left(1+\ACI\otimes J^{\dual{\Red{S}}\otimes\phi^*\tangent{N}}\right)\nablabar^{\dual{\Red{S}}\otimes\varphi^*\tangent{N}}\zeta^{1,0} \\
		={}& \frac12f^k\otimes s^\kappa\left(\delta_k^l\delta_\kappa^\beta + \tensor{\ACI}{_k^l}\tensor{\ACI}{_\kappa^\lambda}\right)\left(-\frac12\omega_l^{LC}\tensor{\ACI}{_\lambda^\tau}\tensor[_\tau]{\zeta}{^{1,0}} + \nabla^{\phi^*\tangent{N}}_{f_l}\tensor[_\lambda]{\zeta}{^{1,0}} - \frac12 \left(\nabla_{f_l}\targetACI\right)\targetACI\tensor[_\lambda]{\zeta}{^{1,0}}\right)
	\end{split}
\end{equation}
The image \(\DelJBar^{\nablabar^{\dual{\Red{S}}\otimes\phi^*\tangent{N}}}\zeta^{1,0}\) is a section of \({\left(\cotangent{\Red{M}}\otimes\dual{\Red{S}}\otimes_\C\phi^*\tangent{N}\right)}^{0,1}\) and \(\delta_\gamma\) maps it to
\begin{equation}
	\begin{split}
		\delta_\gamma\circ \DelJBar^{\nablabar^{\dual{\Red{S}}\otimes\phi^*\tangent{N}}} \zeta^{1,0}
		&= -s^\kappa\tensor{\gamma}{^l_\kappa^\lambda}\left(-\frac12\omega_l^{LC}\tensor{\ACI}{_\lambda^\tau}\tensor[_\tau]{\zeta}{^{1,0}} + \nabla_{f_l}\tensor[_\lambda]{\zeta}{^{1,0}} - \frac12\left(\nabla_{f_l}\targetACI\right)\targetACI\tensor[_\lambda]{\zeta}{^{1,0}}\right) \\
		&= \Red{\Dirac}^{1,0}\zeta^{1,0}.
		\qedhere
	\end{split}
\end{equation}
\end{proof}

Notice that \(\delta_\gamma\) is an isomorphism of real vector bundles.
Consequently, the reduced operator
\begin{equation}
\Red{\Dirac}^{1,0}\colon W^{k,q}(\dual{\Red{S}}\otimes_\C\phi^*\tangent{N}) \to W^{k-1, q}{(\dual{\Red{S}}\otimes\phi^*\tangent{N})}^{0,1}
\end{equation}
is a Fredholm operator with index
\begin{equation}
	\begin{split}
		\ind \Red{\Dirac}^{1,0}
		&= n(2-2p) + 2\left<c_1(\dual{\Red{S}}\otimes\phi^*\tangent{N}), [\Red{M}]\right> \\
		&= 2\left<c_1(\tangent{N}), [\im\phi]\right>.
	\end{split}
\end{equation}

\subsection{Charts for \texorpdfstring{\(\underline{\mathcal{M}(\mathcal{X}, A)}}{M(X,A)}\)}
In this section we show that if \(\underline{\mathcal{S}}\) is transversal to the zero section at the \(\targetACI\)-holomorphic curve \(\phi\colon \Red{M}\to N\), there is a local chart for \(\underline{\mathcal{M}(\mathcal{X}, A)}\) around \(\phi\) of finite dimension
\begin{equation}
	\begin{split}
		\MoveEqLeft
		\dim \ker D_\phi|\dim\ker\Red{\Dirac}^{1,0} + \dim \mathcal{X}
		=\ind D_\phi|\ind\Red{\Dirac}^{1,0} + \dim\mathcal{X}\\
		&= 2n(1-p) + 2\left<c_1(\tangent{N}), A\right>|2\left<c_1(\tangent{N}), A\right> + \dim\mathcal{X}.
	\end{split}
\end{equation}
The union of all such charts gives the moduli space \(\underline{\mathcal{M}(\mathcal{X}, A)}\) of super \(\targetACI\)-holomorphic curves the structure of a supermanifold fibered over \(\underline{\mathcal{X}}\).
This implies the main Theorem~\ref{thm:MainThm}.

To obtain a local chart \(\underline{\mathcal{F}}_\phi^{-1}(0)\) for \(\underline{\mathcal{S}}^{-1}(0)\) we apply the implicit function theorem to \(\underline{\mathcal{F}}_\phi(C)\).
In order to apply the implicit function theorem and use Fredholm theory, we have to complete the spaces of smooth sections \(\underline{U}_\phi(C)\) and \(\underline{F}_\phi(C)\) to appropriate Sobolev spaces.
Then we have to show that \(0\) is a regular value of all \(\underline{\mathcal{F}}_\phi(C)\).
As solutions of \(\DJBar\Phi=0\) are always smooth, the charts for the moduli space of super \(\targetACI\)-holomorphic curves are actually smooth.

Recall that for any vector bundle \(E\to \Red{M}\), any positive integer \(k\) and any real number \(q>1\) the Sobolev space \(W^{k,q}(E)\) is the space of sections of \(E\) such that in any coordinate chart weak derivatives of order up to \(k\) exist and are in \(L^q\).
Let \(\norm{\cdot}_{k,q}\) be a norm on \(W^{k,q}(E)\) turning it into a Banach space.
Such a norm can be obtained by adding up the \(W^{k,p}\)-norms in coordinate charts using a partition of unity.
For the remainder of this section we choose \(k>1\), \(q>1\) such that \(k-\frac2q>1\).
This choice of \(k\), \(q\) guarantees that all functions are at least \(C^1\) by Sobolev embedding theorems and all products, such as norm scalar products, are of the same regularity.
\begin{defn}
	For every map \(\phi\colon \Red{M}\to N\) and every \(C\in\cat{SPoint}\) we define the Banach spaces
	\begin{align}
		\begin{split}
			\MoveEqLeftR
			\underline{W}^{k,q}_\phi(C)
			= W^{k,q}(\phi^*\tangent{N}){\otimes\left(\cO_C\right)}_0 \oplus W^{k,q}(\Red{\dual{S}}\otimes\phi^*\tangent{N})\otimes{\left(\cO_C\right)}_1 \\
				&\oplus W^{k,q}(\phi^*\tangent{N}){\otimes\left(\cO_C\right)}_0
		\end{split} \\
		\begin{split}
			\MoveEqLeftR
			\underline{F}^{k-1,q}_\phi(C)
			= W^{k,q}{(\dual{\Red{S}}\otimes\phi^*\tangent{N})}^{0,1}\otimes{\left(\cO_C\right)}_1 \oplus W^{k,q}(\phi^*\tangent{N})\otimes{\left(\cO_C\right)}_0 \\
				& \oplus W^{k-1,q}{(\cotangent{\Red{M}}\otimes\phi^*\tangent{N})}^{0,1}\otimes{\left(\cO_C\right)}_0\oplus W^{k-1,q}{(\dual{\Red{S}}\otimes\phi^*\tangent{N})}^{0,1}\otimes{\left(\cO_C\right)}_1.
		\end{split}
	\end{align}
	Here we define the Banach space structure on the tensor product as the finite direct sum of the Banach spaces \(W^{k,q}\) indexed by expressions \(\lambda^{\underline{\kappa}}\) for \(\lambda^{\kappa}\) coordinates on \(C\).
	That is, for \((\xi, \zeta, \sigma)\in\underline{W}^{k,q}_\phi(C)\) we define the \(W^{k,q}\)-norm as
	\begin{equation}
		\norm{(\xi, \zeta, \sigma)}_{k,q} = \sum_{\underline{\kappa}\text{ even}}\left(\norm{\tensor[_{\underline{\kappa}}]{\xi}{}}_{k,q} + \norm{\tensor[_{\underline{\kappa}}]{\sigma}{}}_{k,q}\right) + \sum_{\underline{\kappa}\text{ odd}}\norm{\tensor[_{\underline{\kappa}}]{\zeta}{}}_{k,q}.
	\end{equation}
	Here we have chosen norms \(\norm{\cdot}_{k,q}\) on \(W^{k,q}(\phi^*\tangent{N})\) and \(W^{k,q}(\dual{\Red{S}}\otimes\phi^*\tangent{N})\).
	The case \(\underline{F}_\phi^{k-1,q}\) is treated analogously.

	Furthermore, we define
	\begin{equation}
		\underline{U}_\phi^{k,q}(C)
		= \Set{(\xi, \zeta, \sigma)\in \underline{W}^{k,q}_\phi(C)\given \tensor[_0]{\xi}{}\in U},
	\end{equation}
	where \(U\subset \tangent{N}\) is the open neighbourhood of zero where the exponential map is bijective to \(V\subset N\times N\).
\end{defn}

Notice that the compactness of \(\Red{M}\) implies that \(\underline{W}_\phi(C)\) is dense in \(\underline{W}_\phi^{k,q}(C)\) with respect to \(\norm{\cdot}_{k,q}\), and \(\underline{F}_\phi(C)\) is dense in \(\underline{F}_\phi^{k-1,q}\) with respect to \(\norm{\cdot}_{k-1,q}\).
Hence:
\begin{lemma}
	For every \(\phi\colon \Red{M}\to N\) and every \(C\in\cat{SPoint}\) the smooth map \(\underline{\mathcal{F}}_\phi(C)\colon \underline{\mathcal{X}}(C)\times\underline{U}_\phi(C)\to \underline{F}_\phi(C)\) extends continuously to a smooth map
	\begin{equation}
		\underline{\mathcal{F}}_\phi^{k,q}(C)\colon \underline{\mathcal{X}}(C)\times\underline{U}_\Phi^{k,q}(C) \to \underline{F}_\Phi^{k-1,q}(C).
	\end{equation}
	Consequently, \(\underline{\mathcal{F}}_\phi^{k,q}\colon \underline{\mathcal{X}}\times\underline{U}_\phi^{k,q}\to\underline{F}_\phi^{k-1,q}\) is a smooth map between Banach super domains.
	The differential
	\begin{equation}
		{\underline{\mathcal{F}}'}_\phi^{k,q}(C)\colon \underline{\mathcal{X}}(C)\times \underline{W}_\phi^{k,q}(C) \to \underline{F}_\phi^{k-1,q}(C)
	\end{equation}
	of \(\underline{\mathcal{F}}_\phi^{k,q}(C)\) at zero is given by the formulas in Proposition~\ref{prop:DifferentialOfF}.
\end{lemma}
\begin{proof}
	We have argued in the proof of Proposition~\ref{prop:SmoothSection} that \(\underline{\mathcal{F}}_\phi(\chi, \xi, \zeta, \sigma)\) is given by a combination of differentials of the smooth maps \(\exp_{\Phi_C}\) and \(P_{\exp_{\Phi_C}}^{\nablabar}\) and at most one differential operator of order one acting on \((\xi, \zeta, \sigma)\).
	If \((\xi, \zeta, \sigma)\) are not smooth but rather an element of \(\underline{U}_\phi^{k,q}(C)\), the formulas show the regularity of the components of \(\underline{\mathcal{F}}_\phi(C)X\).
	Using the expressions from Proposition~\ref{prop:ComponentsOfDJBar}, we see that the first two component fields of \(\underline{\mathcal{F}}_\phi(C)X\) involve only algebraic expressions, smooth maps, and the component fields of \(X\).
	Consequently, they lie in \(W^{k,q}{(\Red{\dual{S}}\otimes\phi^*\tangent{N})}^{0,1}\otimes{\left(\cO_C\right)}_1\) and \(W^{k,q}(\phi^*\tangent{N})\otimes{\left(\cO_C\right)}_0\).
	The other two component fields of \(\underline{\mathcal{F}}_\phi(C)\) involve in addition first order differential operators, hence they are elements of \(W^{k-1,q}{(\cotangent{\Red{M}}\otimes\phi^*\tangent{N})}^{0,1}\otimes{\left(\cO_C\right)}_0\) and \(W^{k-1,q}{(\Red{\dual{S}}\otimes\phi^*\tangent{N})}^{0,1}\otimes{\left(\cO_C\right)}_1\) respectively.
	This shows the claim.
\end{proof}

For the remainder of this Chapter we assume that \(\phi\colon \Red{M}\to N\) is a \(\targetACI\)-holomorphic curve with \([\im \phi]=A\in H_2(N)\) such that both
\begin{align}
	D_\phi\colon W^{k,q}(\phi^*\tangent{N})&\to W^{k-1,q}{(\cotangent{\Red{M}}\otimes\phi^*\tangent{N})}^{0,1} \\
	\Red{\Dirac}^{1,0}\colon W^{k,q}(\dual{\Red{S}}\otimes_\C\phi^*\tangent{N})&\to W^{k-1,q}{(\dual{\Red{S}}\otimes\phi^*\tangent{N})}^{0,1}
\end{align}
are surjective.
Any surjective Fredholm operator has a (non-canonical) right-inverse, that is, there are linear operators
\begin{align}
	R_1 \colon W^{k-1, q}{(\cotangent{\Red{M}}\otimes\phi^*\tangent{N})}^{0,1} &\to W^{k,q}(\phi^*\tangent{N}) \\
	R_2 \colon  W^{k-1, q}{(\dual{\Red{S}}\otimes\phi^*\tangent{N})}^{0,1} &\to W^{k,q}(\dual{\Red{S}}\otimes_\C\phi^*\tangent{N})
\end{align}
such that
\begin{align}
	D_\phi\circ R_1 &= \id, &
	\Red{\Dirac}^{1,0}\circ R_2 &= \id.
\end{align}
The choice of such right-inverses gives splittings:
\begin{align}
	W^{k,q}(\phi^*\tangent{N}) &= \ker D_\phi\oplus W^{k-1, q}{(\cotangent{\Red{M}}\otimes\phi^*\tangent{N})}^{0,1} \\
	W^{k,q}(\dual{\Red{S}}\otimes_\C\phi^*\tangent{N}) &= \ker\Red{\Dirac}^{1,0}\oplus W^{k-1, q}{(\dual{\Red{S}}\otimes\phi^*\tangent{N})}^{0,1}
\end{align}
Notice, the Proposition~\ref{prop:TransversalityIsAnalysis} extends to all \(\underline{\mathcal{F}}_\phi^{k,q}(C)\) and a choice of right-inverses for \(D_\phi\) and \(\Red{\Dirac}^{1,0}\) yields a right-inverse for \({\underline{\mathcal{F}}'}_\phi^{k,q}(C)\) as follows:
\begin{prop}\label{prop:FkqCHasRightInverse}
	There exists a right-inverse to \({\underline{\mathcal{F}}'}_\phi^{k,q}(C)\), yielding a splitting
	\begin{equation}\label{eq:SplittingRightInverseF'}
		\underline{\mathcal{X}}(C)\times \underline{W}_\phi^{k,q}(C)
		= {\ker\underline{\mathcal{F}}'}_\phi^{k,q}(C)\oplus \underline{F}_\phi^{k-1,q}(C).
	\end{equation}
	Consequently, \(0\) is a regular value of \(\underline{\mathcal{F}}_\phi^{k,q}(C)\).
	The kernel \(\ker{\underline{\mathcal{F}}'}^{k,q}_\phi(C)\) is isomorphic to
	\begin{equation}
		\ker D_\phi\otimes{\left(\cO_C\right)}_0 \oplus \left(\ker \Red{\Dirac}^{1,0}\oplus \mathcal{X}\right)\otimes {\left(\cO_C\right)}_1.
	\end{equation}
	The splitting~\eqref{eq:SplittingRightInverseF'} is functorial in \(C\).
\end{prop}
\begin{proof}
	We pick right-inverses of \(R_1\) of \(D_\phi\) and \(R_2\) of \(\Red{\Dirac}^{1,0}\).
	As \(D_\phi\) and \(\Red{\Dirac}^{1,0}\) are surjective, also \({\underline{\mathcal{F}}'}_\phi^{k,q}(C)\) is surjective by Proposition~\ref{prop:TransversalityIsAnalysis}.
	We have to construct a linear operator
	\begin{equation}
		\underline{R}(C)\colon \underline{F}_\phi^{k-1,q}(C)\to \underline{\mathcal{X}}(C)\times\underline{W}_\phi^{k,q}(C)
	\end{equation}
	such that for any \((t, u, v, w)\in\underline{F}_\phi^{k-1,q}(C)\) the image \((\rho, \xi, \zeta, \sigma)=R(C)(t, u, v, w)\) satisfies
	\begin{equation}
		{\underline{\mathcal{F}}'}_\phi^{k,q}(C)(\rho, \xi, \zeta, \sigma) = (t, u, v, w).
	\end{equation}
	As \(D_{\varphi_C}=D_\phi\otimes\id_{\cO_C}\) and \(\Dirac^{1,0}=\Red{\Dirac}^{1,0}\otimes\id_{\cO_C}\) we obtain right-inverses \(R_{1,C}=R_1\otimes\id_{\cO_C}\) of \(D_{\varphi_C}\) and \(R_{2,C}=R_2\otimes\id_{\cO_C}\) of \(\Dirac^{1,0}\).
	We can read of a solution \((\rho, \xi, \zeta, \sigma)\) from the explicit formulas in Equation~\eqref{eq:ComponentsF'}:
	\begin{align}
		\rho &= 0, \\
		\xi &= - R_{1,C} v, \\
		\zeta^{0,1} &= t, \\
		\zeta^{1,0} &= -R_{2,C}\left(w + \frac12\left(\ACI \gamma^k\otimes\left(\nabla_k\targetACI\right)t + \Tr_g\left<\left(\gamma\otimes\id_{\varphi_C^*\tangent{N}}\right)t, \targetACI\varphi_C^*\nabla\targetACI\right>\differential{\varphi_C}\right)\right), \\
		\sigma &= 4u.
	\end{align}
	This defines a linear right inverse to \({\underline{\mathcal{F}}'}_\phi^{k,q}(C)\) in a functorial manner.

	From the formulas~\eqref{eq:ComponentsF'} one reads off that the kernel of \({\underline{\mathcal{F}}'}^{k,q}_\phi(C)\) is the direct sum of the linear spaces
	\begin{align}
		&\Set{(0, \xi, 0, 0)\in \underline{\mathcal{X}}(C)\times\underline{W}_\phi^{k,q}(C)\given \xi\in \ker D_{\varphi_C}}, \\
		&\Set{(0, 0, \zeta, 0)\in \underline{\mathcal{X}}(C)\times\underline{W}_\phi^{k,q}(C)\given \zeta = 0\oplus \zeta^{1,0}\text{ s.t. }\zeta^{1,0}\in\ker\Dirac^{1,0}}, \\
		&\Set{(\rho, 0, \zeta, 0)\in \underline{\mathcal{X}}(C)\times\underline{W}_\phi^{k,q}(C)\given \zeta=0\oplus\zeta^{1,0}\text{ s.t. }\zeta^{1,0}=2R_{2,C}\left<\vee Q\rho, \differential\varphi_C\right>}.
	\end{align}
	Again because \(D_{\varphi_C}=D_\phi\otimes\id_{\cO_C}\) and \(\Dirac^{1,0}=\Red{\Dirac}^{1,0}\otimes\id_{\cO_C}\), the first one is isomorphic to \(\ker D_\phi\otimes{\left(\cO_C\right)}_0\), the second one to \(\ker\Red{\Dirac}^{1,0}\otimes{\left(\cO_C\right)}_1\) and the third is isomorphic to~\(\underline{\mathcal{X}}(C)\).
	This shows the claim.
\end{proof}

The right-inverse to \({\underline{\mathcal{F}}'}_\phi^{k,q}\) allows us to apply the implicit function theorem to the smooth map \(\underline{\mathcal{F}}_\phi^{k,q}\colon \underline{\mathcal{X}}\times \underline{U}_\phi^{k,q}\to \underline{F}_\phi^{k-1,q}\) between Banach superdomains.
Hence, close to zero the domain of \(\underline{\mathcal{F}}_\phi^{k,q}\) can be written as a product of kernel and image of \(\underline{\mathcal{F}}_\phi^{k,q}\).

\begin{prop}\label{prop:ImplicitFunctionTheorem}
There exist open sets \(U_0\subset\underline{U}_\phi^{k,q}(\R^{0|0})\), \(K_0\subset\underline{\R^{\ind D_\phi|\ind \Red{\Dirac}^{1,0}}}(\R^{0|0})\) and \(F_0\subset\underline{F}_\phi^{k-1,q}(\R^{0|0})\) such that we have a smooth isomorphism
	\begin{equation}
		\underline{\mathcal{X}} \times \underline{U}^{k,q}_\phi|_{U_0}
		= \underline{\mathcal{X}} \times \underline{\R^{\ind D_\phi|\ind \Red{\Dirac}^{1,0}}}|_{K_0}\times \underline{F}_\phi^{k-1,q}|_{F_0}.
	\end{equation}
	With respect to this product, for all \((\chi, \mathfrak{k}, \mathfrak{f})\in \underline{\mathcal{X}}(C)\times \underline{U}_\phi^{k,q}(C)\) it holds \(\underline{\mathcal{F}}_\phi^{k,q}(\chi, \mathfrak{k}, \mathfrak{f}) = \mathfrak{f}\).
	In particular, for all \((\chi, \mathfrak{k}, \mathfrak{f})\in\underline{\mathcal{X}}(C)\times\underline{U}_\phi^{k,q}(C)\) the equality \(\DJBar\exp_{\Phi_{C,\chi}} (\chi, \mathfrak{k}, \mathfrak{f})=0\) is equivalent to \(\mathfrak{f}=0\).
\end{prop}
\begin{proof}
	We apply the implicit function theorem for Banach supermanifolds, see~\cite[Proposition~6.3.2]{M-IDCSM}, twice.
	First, there exist open sets \(\tilde{U}_0\subset\underline{U}_\phi^{k,q}(\R^{0|0})\), \(\tilde{F}_0\subset\underline{F}_\phi^{k-1,q}(\R^{0|0})\), \(\tilde{K}_0\subset\underline{\R^{\ind D_\phi|\ind \Red{\Dirac}^{1,0} + \dim \mathcal{X}}}(\R^{0|0})\) and a smooth isomorphism
	\begin{equation}
		\underline{\tilde{f}}\colon \underline{\mathcal{X}} \times \underline{U}^{k,q}_\phi|_{\tilde{U}_0}
		\to \underline{\R^{\ind D_\phi|\ind \Red{\Dirac}^{1,0}+\dim \mathcal{X}}}|_{\tilde{K}_0}\times \underline{F}_\phi^{k-1,q}|_{\tilde{F}_0},
	\end{equation}
	such that the map \(\underline{\mathcal{F}}_\phi^{k,q}\) is given by projection on the second factor.
	The tangent space of \(\underline{\R^{\ind D_\phi|\ind\Red{\Dirac}^{1,0}+\dim\mathcal{X}}}|_{\tilde{K}_0}\) at zero is given by \(\ker{\underline{\mathcal{F}}'}_\phi^{k,q}\), that is elements \((\rho, \xi, \zeta, \sigma)\in \underline{\mathcal{X}}(C)\times \underline{W}_\phi^{k,q}\) such that
	\begin{align}
		\sigma &= 0, &
		\zeta^{0,1} &= 0, \\
		D_\phi \xi &= 0, &
		\Dirac^{1,0}\zeta^{1,0} - 2\left<\vee Q\rho, \differential{\varphi}_C\right> &= 0.
	\end{align}
	The differential of the projection \(\underline{\mathcal{X}}\times\underline{U}_\phi^{k,q}\to \underline{\mathcal{X}}\) at zero is given by \((\rho, \xi, \zeta, \sigma)\mapsto \rho\).
	As \(\Red{\Dirac}^{1,0}\) and hence \(\Dirac^{1,0}\) are surjective, the projection \(\ker{\underline{\mathcal{F}}'}_\phi^{k,q}\to\underline{\mathcal{X}}\) is also surjective.
	Hence we can apply the implicit function theorem again to split of \(\underline{\mathcal{X}}\) from \(\underline{\R^{\ind D_\phi|\ind\Red{\Dirac}^{1,0}+\dim\mathcal{X}}}|_{\tilde{K}_0}\).
	That is, there are open subsets \(U_0\subset\tilde{U}_0\), \(K_0\subset \tilde{K}_0\) and \(F_0\subset\tilde{F}_0\) and a smooth isomorphism
	\begin{equation}
		\underline{f}\colon \underline{\mathcal{X}} \times \underline{U}^{k,q}_\phi|_{U_0}
		\to \underline{\mathcal{X}} \times \underline{\R^{\ind D_\phi|\ind \Red{\Dirac}^{1,0}}}|_{K_0}\times \underline{F}_\phi^{k-1,q}|_{F_0}.
	\end{equation}
	such that the map \(\underline{\mathcal{F}}_\phi^{k,q}\) is given by projection on the last factor.
\end{proof}

Let us denote the superdomain \(\underline{\R^{\ind D_\phi|\ind\Red{\Dirac}^{1,0}}}|_{K_0}\) constructed in Proposition~\ref{prop:ImplicitFunctionTheorem} by \(\underline{K}_\phi\).
The superdomain \(\underline{\mathcal{X}}\times\underline{K}_\phi\) is immersed in \(\underline{\mathcal{X}}\times\underline{U}_\phi^{k,q}\) and represents the preimage \({\left(\underline{\mathcal{F}}_\phi^{k,q}\right)}^{-1}(0)\).
We now show that its elements are smooth.

\begin{prop}
	Let \((\chi, \xi, \zeta, \sigma)\in\underline{\mathcal{X}}(C)\times \underline{U}_\phi^{k,q}(C)\) such that \(\underline{\mathcal{F}}_\phi^{k,q}(C)(\chi, \xi, \zeta, \sigma)=0\).
	Then \((\xi, \zeta, \sigma)\) is smooth.
	Consequently, the embedding \(\underline{\mathcal{X}}\times\underline{K}_\phi\hookrightarrow\underline{\mathcal{X}}\times\underline{U}_\phi^{k,q}\) constructed in Proposition~\ref{prop:ImplicitFunctionTheorem} factors over \(\underline{U}_\phi\).
\end{prop}
\begin{proof}
	As \((\xi, \zeta, \sigma)\in\underline{U}_\phi^{k,q}(C)\) with \(k-\frac2q>1\), we obtain by Sobolev embedding theorem that all coefficients of \((\xi, \zeta, \sigma)\) are at least once continuously differentiable.
	Let us denote the component fields of \(\exp_{\Phi_{C,\chi}}(\xi, \zeta, \sigma)\) by \(\tilde{\varphi}\), \(\tilde{\psi}\) and \(\tilde{F}\).
	All coefficients in \(\lambda\)-expansions of the components are likewise at least \(C^1\) because \(\exp_{\Phi_{C,\chi}}\) is smooth.
	The component fields \(\tilde{\varphi}\), \(\tilde{\psi}\) and \(\tilde{F}\) satisfy
	\begin{equation}\label{eq:EqComponentsTilde}
		\begin{aligned}
			0 &= \left(1+\ACI\otimes\targetACI\right)\tilde{\psi}, \\
			0 &= \tilde{F}, \\
			0 &= \frac12\left(1 + \ACI\otimes\targetACI\right)\left(\differential{\tilde{\varphi}} + \left<\chi, \tilde{\psi}\right> +\frac14\Tr_{\dual{g}_S}\left(\gamma\otimes\tilde{\mathfrak{j}}\targetACI\right)\tilde{\psi}\right), \\
			0 &= \frac12\left(1 + \ACI\otimes\targetACI\right)\left(\Dirac\tilde{\psi} - 2\left<\vee Q\chi,\differential{\tilde{\varphi}}\right> - \frac13SR^N(\tilde{\psi})\right).
		\end{aligned}
	\end{equation}
	The components \(\tilde{\psi}^{0,1}=0\) and \(F=0\) are obviously smooth.
	The reduced map \(\Red{\tilde{\varphi}}\) satisfies \(\DelJBar\Red{\tilde{\varphi}}=0\) and hence is also smooth, see, for example,~\cite[Appendix~B.4]{McDS-JHCST}.
	To show that the nilpotent part of \(\tilde{\varphi}\) and \(\tilde{\psi}\) are also smooth, we restrict to coordinate neighbourhoods on \(\Red{M}\) with coordinates \(x^1, x^2\) and \(y^a\) on \(N\).
	The expansion of~\(\tilde{\varphi}\) and~\(\tilde{\psi}\) in terms of the generators \(\lambda^\sigma\) of \(\cO_C\) is then given by
	\begin{align}
		\tilde{\varphi}^\# y^a &= \sum_{\underline{\sigma}\text{ even}} \lambda^{\underline{\sigma}}\tensor[_{\underline{\sigma}}]{\tilde{\varphi}}{^a}(x) &
		\tilde{\psi} = \sum_{\underline{\sigma}\text{ odd}} \lambda^{\underline{\sigma}}\tensor[_{\underline{\sigma}}]{\tilde{\psi}}{}
	\end{align}
	where \(\tensor[_{\underline{\sigma}}]{\tilde{\varphi}}{^a}(x)\in C^1(\R^2)\) and \(\tensor[_{\underline{\sigma}}]{\tilde{\psi}}{}\) are \(C^1\)-sectinos of \(\Red{S}\otimes\Red{\tilde{\varphi}}^*\tangent{N}\).
	The coefficients of \(\lambda^{\underline{\sigma}}\) in the differential equations in~\eqref{eq:EqComponentsTilde} are of the form
	\begin{align}
		\frac12\left(1 + \ACI\otimes\targetACI\right)\d\tensor[_{\underline{\sigma}}]{\tilde{\varphi}}{}
		&= \tilde{r}_1(\tensor[_{\underline{\nu}}]{\tilde{\varphi}}{}, \tensor[_{\underline{\nu}}]{\tilde{\psi}}{}; \abs{\underline{\nu}}<\abs{\underline{\sigma}}), \\
		\Red{\Dirac}^{1,0}\tensor[_{\underline{\sigma}}]{\tilde{\psi}}{}
		&= \tilde{r}_2(\tensor[_{\underline{\nu}}]{\tilde{\varphi}}{}, \tensor[_{\underline{\nu}}]{\tilde{\psi}}{}; \abs{\underline{\nu}}<\abs{\underline{\sigma}}),
	\end{align}
	where the right hand side depends smoothly only on the terms \(\tensor[_{\underline{\nu}}]{\tilde{\varphi}}{}\) and \(\tensor[_{\underline{\nu}}]{\tilde{\psi}}{}\) of degree lower than \(\abs{\underline{\sigma}}\).
	As the differential operator on the left-hand side is elliptic, the functions \(\tensor[_{\underline{\sigma}}]{\tilde{\varphi}}{}\) and \(\tensor[_{\underline{\sigma}}]{\tilde{\psi}}{}\) are smooth if the right hand side is.
	Consequently, one can conclude inductively that all \(\tensor[_{\underline{\sigma}}]{\tilde{\varphi}}{}\) and \(\tensor[_{\underline{\sigma}}]{\tilde{\psi}}{}\) are smooth.
	As the map \(\exp_{\Phi_{C,\chi}}\) is smooth and invertible also \((\xi, \zeta, \sigma)\) is smooth.
\end{proof}

The superdomain \(\underline{\mathcal{X}}\times \underline{K}_\phi\) forms a local coordinate patch for the moduli space \(\underline{\mathcal{M}(\mathcal{X}, A)}\) around \(\phi\).
For any \(\phi'= \exp_\phi Y\) for \(Y\in\underline{K}_\phi(\R^{0|0})\), the change of coordinates from \(\underline{\mathcal{X}}\times\underline{K}_\phi\) to \(\underline{\mathcal{X}}\times\underline{K}_{\phi'}\) is smooth as it is given by a composition of the exponential map and its inverse at a different point.
Consequently, in this section we have proven the following main theorem:

\MainThm

By definition, a \(C\)-point of \(\underline{\mathcal{M}(\mathcal{X}, A)}\) is a pair \((\chi, \Phi)\) consisting of a gravitino \(\chi\in\underline{\mathcal{X}}(C)\) and a super \(\targetACI\)-holomorphic curve \(\Phi\colon M_\chi\to N\) such that \([\im \Red{\Phi}]=A\in H_2(N)\).
Here \(M_\chi\) is the super Riemann surface determined by fixed \(\Red{g}\) and \(\Red{S}\) as well as \(\chi\) which can vary in \(\underline{\mathcal{X}}\).

In order to obtain the moduli space of super \(\targetACI\)-holomorphic curves for a fixed super Riemann surface \(M\) determined by a fixed triple \((\Red{g}, \Red{S}, \chi)\) we have to take the preimage of \(\chi\) under the projection \(\underline{\pi}_{\mathcal{X}}\colon\underline{\mathcal{M}(\mathcal{X}, A)}\to \underline{\mathcal{X}}\).
This is done as follows:
For \(\chi\in\underline{\mathcal{X}}(B)\) define
\begin{equation}
	\begin{split}
		\underline{\mathcal{M}(\chi, A)}\colon \cat{SPoint}^{op}_B&\to \cat{Set} \\
		\left(b\colon C\to B\right) &\mapsto \Set{(b^*\chi, \Phi)\in \underline{\mathcal{M}(\mathcal{X}, A)}(C)}
	\end{split}
\end{equation}
The charts \(\underline{\mathcal{X}}\times \underline{K}_\phi\) of \(\underline{\mathcal{M}(\mathcal{X}, A)}\) restrict to charts \(\Set{\chi}\times \underline{K}_\phi\) of \(\underline{\mathcal{M}(\chi, A)}\) and the change of charts \(\underline{K}_\phi\to\underline{K}_{\phi'}\) is a smooth map between superdomains depending on \(\chi\), that is, relative to \(B\).
Hence, the charts \(\Set{\underline{K}_\phi}\) equip \(\underline{\mathcal{M}(\mathcal{X}, A)}\) with the structure of a supermanifold over \(B\) of relative dimension
\begin{equation}
	2n(1-p) + 2\left<c_1(\tangent{N}), A\right>|2\left<c_1(\tangent{N}), A\right>.
\end{equation}

Given that we already allowed the gravitino to vary in a finite-dimensional superdomain it would be a natural generalization to consider families of super Riemann surfaces where all \((g, S, \chi)\) are allowed to vary and the corresponding super \(\targetACI\)-holomorphic curves.
The resulting moduli space would be a fiber bundle over the moduli space of super Riemann surfaces where the fibers are given by the moduli spaces constructed here.
We will postpone this until a theory of the moduli space of super Riemann surfaces in terms of \((g, S, \chi)\) is available.
However, notice that the only closed compact super Riemann surface of genus zero is \(\ProjectiveSpace[\C]{1|1}\), see~\cites{CR-SRSUTT}[Section~9.4]{EK-SGSRSSCA}.
In this case and for \(\mathcal{X}=\Set{0}\), the moduli space \(\underline{\mathcal{M}(\mathcal{X}, A)}\) is the moduli space of all super \(\targetACI\)-holomorphic curves of genus zero.
In analogy to the theory of \(\targetACI\)-holomorphic curves we hope to construct a compactification of the moduli space of super \(\targetACI\)-holomorphic curves and to obtain geometric invariants from the compactified moduli space.

\subsection{Surjectivity of \texorpdfstring{\(\Red{\Dirac}^{1,0}\)}{Dred10psi} via Bochner method}\label{Sec:BochnersMethod}
In this section we apply the Bochner method to obtain that under suitable positivity assumptions on the curvature tensor \(R^N\) on \(\phi^*\tangent{N}\) the Dirac-operator \(\Red{\Dirac}^{1,0}\) is surjective.
Combining those surjectivity results with the transversality theory for \(\targetACI\)-holomorphic curves we obtain examples of maps satisfying the transversality condition for super \(\targetACI\)-holomorphic curves as needed for Theorem~\ref{thm:MainThm}.

In order to fix notation and signs, let us recall Bochner's method for the twisted Dirac operator \(\Red{\Dirac}\).
For the Bochner method in the general setting see, for example,~\cite[Chapter~II.8]{LM-SG}.
In this section and in contrast to before, we use the positive definite metric \(\overline{g}_{\Red{S}}\) on \(\Red{S}\) induced by the \(\UGL(1)\)-structure of the spin structure.
With respect to this metric the operator \(\Red{\Dirac}\colon \Red{\dual{S}}\otimes\phi^*\tangent{N}\to \Red{S}\otimes\phi^*\tangent{N}\) is formally anti-self-adjoint.
That is, for all \(\psi\), \(\Psi\in\VSec{\dual{\Red{S}}\otimes\phi^*\tangent{N}}\)
\begin{equation}
	\int_{\Red{M}} \dual{\overline{g}_{\Red{S}}}\otimes \phi^*n\left(\Dirac\psi, \Psi\right)\d{vol}_g
	= -\int_{\Red{M}} \dual{\overline{g}_{\Red{S}}}\otimes \varphi^*n\left(\psi, \Dirac\Psi\right)\d{vol}_g
\end{equation}
It follows that the index of all Sobolev completions of \(\Red{\Dirac}\) is zero.
The Bochner identity is
\begin{equation}\label{eq:DiracBochnerIdentity}
	\begin{split}
		\Red{\Dirac}^2 \psi
		={}& \gamma(f^k)\nabla^{\dual{\Red{S}}\otimes\phi^*\tangent{N}}_{f_k}\left(\gamma(f^l)\nabla^{\dual{\Red{S}}\otimes\phi^*\tangent{N}}_{f_l}\psi\right) \\
		={}& \gamma(f^k)\gamma(\nabla^{\cotangent{\Red{M}}}_{f_k}f^l)\nabla^{\dual{\Red{S}}\otimes\phi^*\tangent{N}}_{f_l}\psi + \gamma(f^k)\gamma(f^l)\nabla^{\dual{\Red{S}}\otimes\phi^*\tangent{N}}_{f_k}\nabla^{\dual{\Red{S}}\otimes\phi^*\tangent{N}}_{f_l}\psi \\
		={}&  \delta^{kl}\left(\nabla^{\dual{\Red{S}}\otimes\phi^*\tangent{N}}_{f_k}\nabla^{\dual{\Red{S}}\otimes\phi^*\tangent{N}}_{f_l}\psi - \nabla^{\dual{\Red{S}}\otimes\phi^*\tangent{N}}_{\nabla^{\cotangent{\Smooth{M}}}_{f_k}f_l}\psi\right) \\
			& + \gamma^1\gamma^2\left(\nabla^{\dual{\Red{S}}\otimes\phi^*\tangent{N}}_{f^1}\nabla^{\dual{\Red{S}}\otimes\phi^*\tangent{N}}_{f_2}\psi - \nabla^{\dual{\Red{S}}\otimes\phi^*\tangent{N}}_{f^2}\nabla^{\dual{\Red{S}}\otimes\phi^*\tangent{N}}_{f_1}\psi \right.\\
			&- \left.f^l\left(\nabla^{\tangent{\Red{M}}}_{f_1}f_2 - \nabla^{\tangent{\Red{M}}}_{f_2}f_1\right)\nabla^{\dual{\Red{S}}\otimes\phi^*\tangent{N}}_{f_l}\psi\right) \\
		={}& \Tr_g \nabla^{\dual{\Red{S}}\otimes\phi^*\tangent{N}}\nabla^{\dual{\Red{S}}\otimes\phi^*\tangent{N}}\psi + \gamma^1\gamma^2R^{\dual{\Red{S}}\otimes\phi^*\tangent{N}}(f_1, f_2)\psi
	\end{split}
\end{equation}
If we assume that \(\psi\) is in the kernel of \(\Red{\Dirac}\), we have
\begin{equation}
	\begin{split}
		0
		={}& \int_{\Red{M}} \norm{\Red{\Dirac}\psi}^2 \d{vol}_g \\
		={}& -\int_{\Red{M}} \dual{\overline{g}_{\Red{S}}}\otimes \phi^*n\left(\Red{\Dirac}\Red{\Dirac}\psi, \psi\right) \d{vol}_g \\
		={}& -\int_{\Red{M}} \dual{\overline{g}_{\Red{S}}}\otimes\phi^*n\left(\Tr_g \nabla^{\dual{\Red{S}}\otimes\phi^*\tangent{N}}\nabla^{\dual{\Red{S}}\otimes\phi^*\tangent{N}}\psi, \psi\right) \\
			&\quad + \dual{\overline{g}_{\Red{S}}}\otimes\phi^*n\left(\gamma^1\gamma^2R^{\dual{\Red{S}}\otimes\phi^*\tangent{N}}(f_1, f_2)\psi, \psi\right) \d{vol}_g \\
		={}& \int_{\Red{M}} \norm{\nabla^{\dual{\Red{S}}\otimes\phi^*\tangent{N}}\psi}^2 - \dual{\overline{g}_{\Red{S}}}\otimes\phi^*n\left(\gamma^1\gamma^2R^{\dual{\Red{S}}\otimes\phi^*\tangent{N}}(f_1, f_2)\psi, \psi\right) \d{vol}_g.
	\end{split}
\end{equation}
In particular if the curvature term is everywhere non-negative and strictly positive at one point, we obtain a contradiction.
In that case, the twisted spinor \(\psi\) has to vanish; the Dirac operator \(\Red{\Dirac}\) has trivial kernel.

Let us now turn to the operator \(\Red{\Dirac}^{1,0}\).
For \(\psi\in\VSec{\dual{\Red{S}}\otimes_\C\varphi^*\tangent{N}}\) and \(\Psi\in\VSec{{\left(\dual{\Red{S}}\otimes\varphi^*\tangent{N}\right)}^{0,1}}\) we have
\begin{equation}
	\begin{split}
		\MoveEqLeft
		\int_{\Red{M}}\dual{\overline{g}_{\Red{S}}}\otimes \phi^*n\left(\Red{\Dirac}^{1,0}\psi, \Psi\right)\d{vol}_g \\
		&= \frac14\int_{\Red{M}}\dual{\overline{g}_{\Red{S}}}\otimes \phi^*n\left(\Red{\Dirac}\left(1-\ACI\otimes\targetACI\right)\psi, \left(1+\ACI\otimes\targetACI\right)\Psi\right)\d{vol}_g \\
		&= -\frac14\int_{\Red{M}}\dual{\overline{g}_{\Red{S}}}\otimes \phi^*n\left(\left(1-\ACI\otimes\targetACI\right)\psi, \Red{\Dirac}\left(1+\ACI\otimes\targetACI\right)\Psi\right)\d{vol}_g \\
		&= -\int_{\Red{M}}\dual{\overline{g}_{\Red{S}}}\otimes \phi^*n\left(\psi, \Red{\Dirac}^{0,1}\Psi\right)\d{vol}_g
	\end{split}
\end{equation}
That is, \(-\Red{\Dirac}^{0,1}\) is the formal adjoint of \(\Red{\Dirac}^{1,0}\) and hence,
\begin{equation}
	0
	= \ind \Red{\Dirac}
	= \ind \Red{\Dirac}^{1,0} + \ind\Red{\Dirac}^{0,1}.
\end{equation}
By definition of the operators \(\Red{\Dirac}^{1,0}\) and \(\Red{\Dirac}^{0,1}\), see also Equation~\eqref{eq:DiracOperatorsProjectors}, we read of that for \(\psi\in\VSec{\dual{S}\otimes_\C\varphi^*\tangent{N}}\) and \(\Psi\in\VSec{{\left(\dual{S}\otimes\varphi^*\tangent{N}\right)}^{0,1}}\) we have
\begin{align}
	\Dirac^{1,0}\psi &= \gamma(f^k)\overline{\nabla}_{f_k}\psi, &
	\Dirac^{0,1}\Psi &= \gamma(f^k)\overline{\nabla}_{f_k}\Psi.
\end{align}
Analogously to~\eqref{eq:DiracBochnerIdentity}, we obtain:
\begin{align}
	\Dirac^{1,0}\Dirac^{0,1}\Psi
	&= \Tr_g\overline{\nabla}^{\dual{\Red{S}}\otimes\phi^*\tangent{N}}\overline{\nabla}^{\dual{\Red{S}}\otimes\phi^*\tangent{N}}\Psi + \gamma^1\gamma^2\overline{R}^{\dual{\Red{S}}\otimes\phi^*\tangent{N}}(f_1, f_2)\Psi, \\
	\Dirac^{0,1}\Dirac^{1,0}\psi
	&= \Tr_g\overline{\nabla}^{\dual{\Red{S}}\otimes\phi^*\tangent{N}}\overline{\nabla}^{\dual{\Red{S}}\otimes\phi^*\tangent{N}}\psi + \gamma^1\gamma^2\overline{R}^{\dual{\Red{S}}\otimes\phi^*\tangent{N}}(f_1, f_2)\psi,
\end{align}
where \(\overline{R}^{\dual{\Red{S}}\otimes\phi^*\tangent{N}}\) is the curvature of \(\overline{\nabla}^{\dual{\Red{S}}\otimes\phi^*\tangent{N}}\).
For \(\Psi\in\ker\Red{\Dirac}^{0,1}\) we obtain
\begin{equation}
	\begin{split}
		0
		&= \int_{\Red{M}}\norm{\Red{\Dirac}^{0,1}\Psi}^2 \d{vol}_g \\
		&= \int_{\Red{M}} \norm{\overline{\nabla}^{\dual{\Red{S}}\otimes\phi^*\tangent{N}}\Psi}^2 - \dual{\overline{g}_{\Red{S}}}\otimes\phi^*n\left(\gamma^1\gamma^2\overline{R}^{\dual{\Red{S}}\otimes\phi^*\tangent{N}}(f_1, f_2)\Psi, \Psi\right) \d{vol}_g.
	\end{split}
\end{equation}
This proves the following proposition:
\begin{prop}
	\begin{enumerate}
		\item
			Assume that for any \(\Psi\in {\VSec{\dual{\Red{S}}\otimes\phi^*\tangent{N}}}^{0,1}\) the curvature expression
			\begin{equation}
				- \dual{\overline{g}_{\Red{S}}}\otimes\phi^*n\left(\gamma^1\gamma^2\overline{R}^{\dual{\Red{S}}\otimes\phi^*\tangent{N}}(f_1, f_2)\Psi, \Psi\right)
			\end{equation}
			is non-negative and strictly positive at one point.
			Then the kernel of \(\Red{\Dirac}^{0,1}\) consists only of the zero-section.
			Consequently, \(\Red{\Dirac}^{1,0}\) is surjective.
		\item
			Assume that for any \(\psi\in \VSec{\dual{\Red{S}}\otimes_C\phi^*\tangent{N}}\) the curvature expression
			\begin{equation}
				- \dual{\overline{g}_{\Red{S}}}\otimes\phi^*n\left(\gamma^1\gamma^2\overline{R}^{\dual{\Red{S}}\otimes\phi^*\tangent{N}}(f_1, f_2)\psi, \psi\right)
			\end{equation}
			is non-negative and strictly positive at one point.
			Then the kernel of \(\Red{\Dirac}^{1,0}\) consists only of the zero-section.
	\end{enumerate}
\end{prop}

\begin{ex}[Maps from a surface with constant curvature to {\(\C^n\)}]\label{ex:SurfaceToCn}
	By Gauß--Bonnet Theorem, a closed surface with constant scalar curvature \(\mathfrak{s}\), has positive scalar curvature in the case of genus zero, vanishing scalar curvature in the case of genus one, and negative scalar curvature for genus two and higher.
	Up to a conformal rescaling we can assume that the scalar curvature is either \(\mathfrak{s}=+2\), \(\mathfrak{s}=0\) or \(\mathfrak{s}=-2\).
	Let us use the orthonormal frame \(f_a\) on \(\tangent{\Red{M}}\) and \(s_\alpha\) on \(S\).
	With respect to those frames, we have \(\gamma(f^1)\gamma(f^2)=\ACI\), the curvature tensor on \(\dual{S}\) is given by \(R^{\dual{S}}(f_1, f_2)=\frac14\mathfrak{s}\ACI\) and \(\Psi=s^\alpha\otimes \tensor[_\alpha]{\Psi}{}\).
	Then
	\begin{equation}
		- \dual{\overline{g}_{\Red{S}}}\otimes\phi^*n\left(\gamma^1\gamma^2\overline{R}^{\dual{\Red{S}}\otimes\phi^*\tangent{N}}(f_1, f_2)\Psi, \Psi\right)
		= \frac14\mathfrak{s}\norm{\Psi}^2.
	\end{equation}
	In the case of genus zero, that is \(\mathfrak{s}=2\), this expression is non-negative and strictly positive at any point where \(\Psi\) does not vanish.
	Consequently, both \(\Red{\Dirac}^{1,0}\) and \(\Red{\Dirac}^{0,1}\) are bijective for maps \(\ProjectiveSpace[\C]{1}\to \C^n\).

	In the case of maps \(\phi\colon \ProjectiveSpace[_\C]{1}\to\C^n\) the operator \(D_\phi\) coincides with several direct copies of \(\overline{\partial}\) and can be shown to be surjective via Dolbeault cohomology, see~\cite[Lemma~3.3.1]{McDS-JHCST}.
	Hence, by Theorem~\ref{thm:MainThm}, the moduli space of super \(\targetACI\)-holomorphic curves \(\Phi\colon \ProjectiveSpace[\C]{1|1}\to \C^n\) can be constructed and is a smooth supermanifold of real dimension \(2n|0\).
	As holomorphic maps from a compact surface to \(\C^n\) are constant, the moduli space of super \(\targetACI\)-holomorphic curves \(\Phi\colon \ProjectiveSpace[\C]{1|1}\to \C^n\) coincides with \(\C^n\).
\end{ex}

\begin{ex}[Maps from a surface with constant curvature to spaces with constant holomorphic sectional curvature]\label{ex:SurfaceToConstantHolomorphicBisectionalCurvature}
	Recall that a Kähler manifold is said to be of constant holomorphic sectional curvature \(\sigma\) if
	\begin{equation}
		\sigma(X) = n\left(R^N(X, \targetACI X) \targetACI X, X\right)
	\end{equation}
	is independent of the vector field \(X\) where \(\norm{X}=1\), see, for example~\cite{GK-HBC}.
	Examples of manifolds with constant holomorphic sectional curvature are \(\ProjectiveSpace[\C]{n}\) with the Fubini--Study metric having \(\sigma=4\), hyperbolic space with \(\sigma=-4\) and euclidean \(\C^n\) with \(\sigma=0\).
	In our notation the curvature tensor of a manifold with constant holomorphic curvature for \(X\), \(Y\), \(Z\), \(W\in\VSec{\tangent{N}}\) is given by
	\begin{equation}
		\begin{split}
			n(R^N(X, Y)Z, W)
			={}& \frac\sigma4\left(n(X, W)n(Y, Z) - n(X, Z)n(Y, W) + n(Z, \targetACI X)n(Y, \targetACI W)  \right.\\
			&-\left. n(X, \targetACI W)n(\targetACI Y, Z) - 2n(X, \targetACI Y)n(Z, \targetACI W)\right),
		\end{split}
	\end{equation}
	and hence, using the notation from Example~\ref{ex:SurfaceToCn}
	\begin{equation}
		\begin{split}
			\MoveEqLeftR
			- \dual{\overline{g}_{\Red{S}}}\otimes\phi^*n\left(\gamma^1\gamma^2\overline{R}^{\dual{\Red{S}}\otimes\phi^*\tangent{N}}(f_1, f_2)\Psi, \Psi\right) \\
			={}& \frac14\mathfrak{s}\norm{\Psi}^2 - \delta^{\alpha\beta}\phi^*n\left(R^N(\differential{\phi}(f_1), \differential{\phi}(f_2))\tensor{\ACI}{_\alpha^\kappa}\tensor[_\kappa]{\Psi}{}, \tensor[_\beta]{\Psi}{}\right) \\
			={}& \frac14\mathfrak{s}\norm{\Psi}^2 - \frac\sigma4\delta^{\alpha\beta}\left(\phi^*n(\differential{\phi}(f_1), \tensor[_\beta]{\Psi}{})\phi^*n(\differential{\phi}(f_2), \tensor{\ACI}{_\alpha^\kappa}\tensor[_\kappa]{\Psi}{}) \right.\\
				&- \phi^*n(\differential{\phi}(f_1), \tensor{\ACI}{_\alpha^\kappa}\tensor[_\kappa]{\Psi}{})\phi^*n(\differential{\phi}(f_2),\tensor[_\beta]{\Psi}{}) \\
				&+ \phi^*n(\tensor{\ACI}{_\alpha^\kappa}\tensor[_\kappa]{\Psi}{}, \targetACI\differential{\phi}(f_1))\phi^*n(\targetACI\differential{\phi}(f_2), \tensor[_\beta]{\Psi}{}) \\
				&- \phi^*n(\differential{\phi}(f_1), \targetACI\tensor[_\beta]{\psi}{})\phi^*n(\targetACI\differential{\phi}(f_2), \tensor{\ACI}{_\alpha^\kappa}\tensor[_\kappa]{\Psi}{}) \\
				&- \left.2\phi^*n(\differential{\phi}(f_1), \targetACI\differential{\phi}(f_2))\phi^*n(\tensor{\ACI}{_\alpha^\kappa}\tensor[_\kappa]{\Psi}{}, \targetACI\tensor[_\beta]{\Psi}{})\right) \\
			={}& \frac14\mathfrak{s}\norm{\Psi}^2 + \frac\sigma4\delta^{\alpha\beta}\delta^{kl}\phi^*n(\differential{\phi}(f_k), \tensor[_\alpha]{\Psi}{})\phi^*n(\differential{\phi}(f_l), \targetACI\tensor{\ACI}{_\beta^\kappa}\tensor[_\kappa]{\Psi}{}) \\
				&+ \frac\sigma4\delta^{\alpha\beta}\norm{\differential{\phi}}^2\phi^*n(\tensor[_\beta]{\Psi}{}, \targetACI\tensor{\ACI}{_\alpha^\kappa}\tensor[_\kappa]{\Psi}{}).
		\end{split}
	\end{equation}
	Consequently, for \(\left(1\pm\ACI\otimes\targetACI\right)\Psi=0\), we have
	\begin{equation}
		\begin{split}
			\MoveEqLeftR
			- \dual{\overline{g}_{\Red{S}}}\otimes\phi^*n\left(\gamma^1\gamma^2\overline{R}^{\dual{\Red{S}}\otimes\phi^*\tangent{N}}(f_1, f_2)\Psi, \Psi\right) \\
			={}& \frac14\mathfrak{s}\norm{\Psi}^2 \mp \frac\sigma4\left(\norm{\phi^*n(\differential{\phi}, \Psi)}^2 + \norm{\differential{\phi}}^2\norm{\Psi}^2\right).
		\end{split}
	\end{equation}
	This yields the following cases for \(\sigma>0\):
	\begin{itemize}
		\item\emph{genus \(p=0\)}: \(\Red{\Dirac}^{0,1}\) is injective \(\Red{\Dirac}^{1,0}\) is surjective for all \(\phi\); \(\Red{\Dirac}^{1,0}\) is injective, if \(\norm{\differential{\phi}}^2\leq\frac{\mathfrak{s}}{2\sigma}\) with strict inequality at one point.
		\item\emph{genus \(p=1\)}: \(\Red{\Dirac}^{0,1}\) is injective and \(\Red{\Dirac}^{1,0}\) is surjective as long as \(\differential{\phi}\neq0\) at one point.
		\item\emph{genus \(p>1\)}: \(\Red{\Dirac}^{0,1}\) is injective and \(\Red{\Dirac}^{1,0}\) is surjective if \(\norm{\differential{\phi}}^2\geq\frac{\mathfrak{s}}{2\sigma}\) with strict inequality at one point.
	\end{itemize}
	For \(\sigma<0\) the roles of \(\Red{\Dirac}^{1,0}\) and \(\Red{\Dirac}^{0,1}\) are reversed.

	It is shown in~\cite[Theorem~3.1.5]{McDS-JHCST} that \(D_\phi\) is surjective for all simple \(\phi\) and generic almost complex structure~\(\targetACI\).
	If one uses in the above enumeration the strict inequalities, the surjectivity of \(\Red{\Dirac}^{1,0}\) is preserved under small perturbations of \(\targetACI\).
	Consequently, for an almost complex structure~\(\targetACI'\) close to \(\targetACI\) the moduli space of simple super \(\targetACI'\)-holomorphic curves can be constructed.
\end{ex}
\begin{ex}[{Maps from \(\ProjectiveSpace[\C]{1}\) to \(\ProjectiveSpace[\C]{n}\)}]
	Every homology class \(A\in H_2(\ProjectiveSpace[\C]{n}, \Z)\) can be written as a product of \(k\in\Z\) and the Poincaré dual of the hyperplane class.
	One can check using the Euler sequence that \(\left<c_1(\tangent{\ProjectiveSpace[\C]{n}}), A\right>=k(n+1)\).
	The dimension formula for the super moduli space yields
	\begin{equation}
		2n + 2k(n+1)|2k(n+1).
	\end{equation}
	The bundle \(\phi^*\tangent{N}\) is globally generated by holomorphic sections and splits into a direct sum of line bundles whose first Chern number is non-negative.
	Hence, by~\cite[Lemma~3.3.1]{McDS-JHCST} the operator \(D_\phi\) is surjective.
	The previous Example~\ref{ex:SurfaceToConstantHolomorphicBisectionalCurvature} shows that \(\Red{\Dirac}^{1,0}\) is surjective for all maps \(\phi\colon \ProjectiveSpace[\C]{1}\to \ProjectiveSpace[\C]{n}\).
	Consequently, the moduli space of super \(\targetACI\)-holomorphic curves \(\ProjectiveSpace[\C]{1|1}\to \ProjectiveSpace[\C]{n}\) can be constructed using Theorem~\ref{thm:MainThm}.

	If we use the metric of constant scalar curvature \(\mathfrak{s}=2\) on \(\ProjectiveSpace[\C]{1}\) and the Fubini--Study metric on \(\ProjectiveSpace[\C]{n}\), that is \(\sigma=4\), we obtain in addition the following statement:
	If \(\norm{\differential{\phi}}^2\leq\frac14\), the operator \(\Red{\Dirac}^{1,0}\) is bijective and hence \([\im\phi]\) is the trivial homology class (\(k=0\)).
\end{ex}

\counterwithin{lemma}{section}
\counterwithin{equation}{section}
\appendix
\section{Appendix: Connections on the target}
In this appendix we prove a Lemma on symmetry properties of a contraction of the curvature tensor \(R^N\) of the target with coefficients of the twisted spinor \(\psi\) that has been used at different places in the main text.
\begin{lemma}\label{lemma:SRIdentities}
	For the pullback of the curvature of \(N\) along \(\varphi\) and the components of \(\psi=s^\mu\otimes\tensor[_\mu]{\psi}{}\) it holds
	\begin{equation}
		6R^N(\tensor[_\mu]{\psi}{}, \tensor[_\nu]{\psi}{})\tensor[_\sigma]{\psi}{}
		= 2\left(\Gamma_{\mu\nu}^t\tensor{\gamma}{_t_\sigma^\tau} - \delta_{\mu\nu}\tensor{\ACI}{_\sigma^\tau}\right)\tensor[_\tau]{{SR^N(\psi)}}{}
		= \left(\delta_{\nu\sigma}\tensor{\ACI}{_\nu^\tau} - \Gamma_{\nu\sigma}^t\tensor{\gamma}{_t_\mu^\tau}\right)\tensor[_\tau]{{SR^N(\psi)}}{}.
	\end{equation}
	Analogously, for the derivative of the curvature tensor it holds
	\begin{equation}
		\begin{split}
			6\left(\nabla_{\tensor[_\rho]{\psi}{}}R^N\right)\left(\tensor[_\mu]{\psi}{}, \tensor[_\nu]{\psi}{}\right)\tensor[_\sigma]{\psi}{}
			&= 2\left(\Gamma_{\mu\nu}^t\tensor{\gamma}{_t_\sigma^\tau} - \delta_{\mu\nu}\tensor{\ACI}{_\sigma^\tau}\right)\varepsilon^{\kappa\lambda}\left(\nabla_{\tensor[_\rho]{\psi}{}}R^N\right)\left(\tensor[_\tau]{\psi}{}, \tensor[_\kappa]{\psi}{}\right)\tensor[_\lambda]{\psi}{} \\
			&= \left(\delta_{\nu\sigma}\tensor{\ACI}{_\nu^\tau} - \Gamma_{\nu\sigma}^t\tensor{\gamma}{_t_\mu^\tau}\right)\varepsilon^{\kappa\lambda}\left(\nabla_{\tensor[_\rho]{\psi}{}}R^N\right)\left(\tensor[_\tau]{\psi}{}, \tensor[_\kappa]{\psi}{}\right)\tensor[_\lambda]{\psi}{}.
		\end{split}
	\end{equation}
\end{lemma}
\begin{proof}
	We use the anti-symmetry
	\begin{equation}
		R^N(X, Y)=-{(-1)}^{\p{X}\p{Y}}R^N(Y,X)
	\end{equation}
	and the Bianchi identity of the curvature tensor \(R^N\)
	\begin{equation}
		R^N(X,Y)Z+{\left(-1\right)}^{\p{Z}(\p{X}+\p{Y})}R^N(Z,X)Y + {\left(-1\right)}^{\p{X}(\p{Y}+\p{Z})} R^N(Y,Z)X=0,
	\end{equation}
	for vector fields \(X\), \(Y\) and \(Z\in\VSec{\varphi^*\tangent{N}}\) of parity \(\p{X}\), \(\p{Y}\) and \(\p{Z}\).
	For the odd coefficients \(\tensor[_\alpha]{\psi}{}\), \(\alpha=3,4\) of \(\psi\) we have
	\begin{align}
		R^N\left(\tensor[_3]{\psi}{}, \tensor[_3]{\psi}{}\right)\tensor[_3]{\psi}{}
		&= R^N\left(\tensor[_4]{\psi}{}, \tensor[_4]{\psi}{}\right)\tensor[_4]{\psi}{}
		= 0, \\
		R^N\left(\tensor[_4]{\psi}{}, \tensor[_3]{\psi}{}\right)\tensor[_3]{\psi}{}
		&= R^N\left(\tensor[_3]{\psi}{}, \tensor[_4]{\psi}{}\right)\tensor[_3]{\psi}{}
		= -\frac12R^N\left(\tensor[_3]{\psi}{}, \tensor[_3]{\psi}{}\right)\tensor[_4]{\psi}{}, \\
		R^N\left(\tensor[_4]{\psi}{}, \tensor[_3]{\psi}{}\right)\tensor[_4]{\psi}{}
		&= R^N\left(\tensor[_3]{\psi}{}, \tensor[_4]{\psi}{}\right)\tensor[_4]{\psi}{}
		= -\frac12R^N\left(\tensor[_4]{\psi}{}, \tensor[_4]{\psi}{}\right)\tensor[_3]{\psi}{}.
	\end{align}
	It follows
	\begin{align}
		\begin{split}
			\tensor[_3]{{SR^N(\psi)}}{}
			&= \varepsilon^{\kappa\lambda}R^N(\tensor[_3]{\psi}{}, \tensor[_\kappa]{\psi}{})\tensor[_\lambda]{\psi}{}
			= R^N(\tensor[_3]{\psi}{}, \tensor[_3]{\psi}{})\tensor[_4]{\psi}{} - R^N(\tensor[_3]{\psi}{}, \tensor[_4]{\psi}{})\tensor[_3]{\psi}{} \\
			&= -3R\left(\tensor[_3]{\psi}{}, \tensor[_4]{\psi}{}\right)\tensor[_3]{\psi}{},
		\end{split} \\
		\begin{split}
			\tensor[_4]{{SR^N(\psi)}}{}
			&= \varepsilon^{\kappa\lambda}R^N(\tensor[_4]{\psi}{}, \tensor[_\kappa]{\psi}{})\tensor[_\lambda]{\psi}{}
			= R^N(\tensor[_4]{\psi}{}, \tensor[_3]{\psi}{})\tensor[_4]{\psi}{} - R^N(\tensor[_4]{\psi}{}, \tensor[_4]{\psi}{})\tensor[_3]{\psi}{} \\
			&= 3R\left(\tensor[_3]{\psi}{}, \tensor[_4]{\psi}{}\right)\tensor[_4]{\psi}{}.
		\end{split}
	\end{align}
	We now use a Fierz-type identity to obtain
	\begin{equation}\label{eq:PullbackCurvatureFierz}
			2R^N(\tensor[_\mu]{\psi}{}, \tensor[_\nu]{\psi}{})\tensor[_\sigma]{\psi}{}
			=\delta_{\mu\nu}\delta^{\kappa\lambda}R^N(\tensor[_\kappa]{\psi}{}, \tensor[_\lambda]{\psi}{})\tensor[_\sigma]{\psi}{} + \Gamma_{\mu\nu}^t\varepsilon^{\kappa\tau}\tensor{\gamma}{_t_\tau^\lambda}R^N(\tensor[_\kappa]{\psi}{}, \tensor[_\lambda]{\psi}{})\tensor[_\sigma]{\psi}{}.
	\end{equation}
	In order to further simplify the summands, we calculate
	\begin{align}
		\begin{split}
			\delta^{\kappa\lambda}R^N(\tensor[_\kappa]{\psi}{}, \tensor[_\lambda]{\psi}{})\tensor[_3]{\psi}{}
			&= R^N(\tensor[_3]{\psi}{}, \tensor[_3]{\psi}{})\tensor[_3]{\psi}{} + R^N(\tensor[_4]{\psi}{}, \tensor[_4]{\psi}{})\tensor[_3]{\psi}{} \\
			&= -2R^N(\tensor[_3]{\psi}{}, \tensor[_4]{\psi}{})\tensor[_4]{\psi}{}
			= -\frac23\tensor[_4]{{SR^N(\psi)}}{},
		\end{split}\\
		\begin{split}
			\delta^{\kappa\lambda}R^N(\tensor[_\kappa]{\psi}{}, \tensor[_\lambda]{\psi}{})\tensor[_4]{\psi}{}
			&= R^N(\tensor[_3]{\psi}{}, \tensor[_3]{\psi}{})\tensor[_4]{\psi}{} + R^N(\tensor[_4]{\psi}{}, \tensor[_4]{\psi}{})\tensor[_4]{\psi}{} \\
			&= -2R^N(\tensor[_3]{\psi}{}, \tensor[_4]{\psi}{})\tensor[_3]{\psi}{}
			= \frac23\tensor[_3]{{SR^N(\psi)}}{},
		\end{split}
	\end{align}
	and hence for the first summand in~\eqref{eq:PullbackCurvatureFierz}
	\begin{equation}
		\delta^{\kappa\lambda}R^N(\tensor[_\kappa]{\psi}{}, \tensor[_\lambda]{\psi}{})\tensor[_\sigma]{\psi}{}
		= -\frac23\tensor{\ACI}{_\sigma^\tau}\tensor[_\tau]{{SR^N(\psi)}}{}.
	\end{equation}
	For the second summand verify
	\begin{align}
		\begin{split}
			\varepsilon^{\kappa\tau}\tensor{\gamma}{_1_\tau^\lambda}R^N(\tensor[_\kappa]{\psi}{}, \tensor[_\lambda]{\psi}{})\tensor[_3]{\psi}{}
			&= R^N(\tensor[_3]{\psi}{}, \tensor[_3]{\psi}{})\tensor[_3]{\psi}{} - R^N(\tensor[_4]{\psi}{}, \tensor[_4]{\psi}{})\tensor[_3]{\psi}{} \\
			&= 2R^N(\tensor[_3]{\psi}{}, \tensor[_4]{\psi}{})\tensor[_4]{\psi}{}
			= \frac23\tensor[_4]{{SR^N(\psi)}}{},
		\end{split}\\
		\begin{split}
			\varepsilon^{\kappa\tau}\tensor{\gamma}{_1_\tau^\lambda}R^N(\tensor[_\kappa]{\psi}{}, \tensor[_\lambda]{\psi}{})\tensor[_4]{\psi}{}
			&= R^N(\tensor[_3]{\psi}{}, \tensor[_3]{\psi}{})\tensor[_4]{\psi}{} - R^N(\tensor[_4]{\psi}{}, \tensor[_4]{\psi}{})\tensor[_4]{\psi}{} \\
			&= -2R^N(\tensor[_3]{\psi}{}, \tensor[_4]{\psi}{})\tensor[_3]{\psi}{}
			= \frac23\tensor[_3]{{SR^N(\psi)}}{},
		\end{split}\\
		\begin{split}
			\varepsilon^{\kappa\tau}\tensor{\gamma}{_2_\tau^\lambda}R^N(\tensor[_\kappa]{\psi}{}, \tensor[_\lambda]{\psi}{})\tensor[_3]{\psi}{}
			&= R^N(\tensor[_3]{\psi}{}, \tensor[_4]{\psi}{})\tensor[_3]{\psi}{} + R^N(\tensor[_4]{\psi}{}, \tensor[_3]{\psi}{})\tensor[_3]{\psi}{} \\
			&= 2R^N(\tensor[_3]{\psi}{}, \tensor[_4]{\psi}{})\tensor[_3]{\psi}{}
			= -\frac23\tensor[_3]{{SR^N(\psi)}}{},
		\end{split}\\
		\begin{split}
			\varepsilon^{\kappa\tau}\tensor{\gamma}{_2_\tau^\lambda}R^N(\tensor[_\kappa]{\psi}{}, \tensor[_\lambda]{\psi}{})\tensor[_4]{\psi}{}
			&= R^N(\tensor[_3]{\psi}{}, \tensor[_4]{\psi}{})\tensor[_4]{\psi}{} + R^N(\tensor[_4]{\psi}{}, \tensor[_3]{\psi}{})\tensor[_4]{\psi}{} \\
			&= 2R^N(\tensor[_3]{\psi}{}, \tensor[_4]{\psi}{})\tensor[_4]{\psi}{}
			= \frac23\tensor[_4]{{SR^N(\psi)}}{},
		\end{split}
	\end{align}
	or shorter
	\begin{equation}
		\varepsilon^{\kappa\tau}\tensor{\gamma}{_t_\tau^\lambda}R^N(\tensor[_\kappa]{\psi}{}, \tensor[_\lambda]{\psi}{})\tensor[_\sigma]{\psi}{}
		= \frac23\tensor{\gamma}{_t_\sigma^\tau}\tensor[_\tau]{{SR^N(\psi)}}{}.
	\end{equation}
	This shows the first claim of the Lemma.
	The second equality can be derived analogously.
	First, the application of Fierz identity yields
	\begin{equation}
			2R^N(\tensor[_\mu]{\psi}{}, \tensor[_\nu]{\psi}{})\tensor[_\sigma]{\psi}{}
			=\delta_{\nu\sigma}\delta^{\kappa\lambda}R^N(\tensor[_\mu]{\psi}{}, \tensor[_\kappa]{\psi}{})\tensor[_\lambda]{\psi}{} + \Gamma_{\nu\sigma}^t\varepsilon^{\kappa\tau}\tensor{\gamma}{_t_\tau^\lambda}R^N(\tensor[_\mu]{\psi}{}, \tensor[_\kappa]{\psi}{})\tensor[_\lambda]{\psi}{}.
	\end{equation}
	Here for the first summand verify
	\begin{align}
		\delta^{\kappa\lambda}R^N(\tensor[_3]{\psi}{}, \tensor[_\kappa]{\psi}{})\tensor[_\lambda]{\psi}{}
		&= R^N(\tensor[_3]{\psi}{}, \tensor[_3]{\psi}{})\tensor[_3]{\psi}{} + R^N(\tensor[_3]{\psi}{}, \tensor[_4]{\psi}{})\tensor[_4]{\psi}{}
		= \frac13\tensor[_4]{{SR^N(\psi)}}{}, \\
		\delta^{\kappa\lambda}R^N(\tensor[_4]{\psi}{}, \tensor[_\kappa]{\psi}{})\tensor[_\lambda]{\psi}{}
		&= R^N(\tensor[_4]{\psi}{}, \tensor[_3]{\psi}{})\tensor[_3]{\psi}{} + R^N(\tensor[_4]{\psi}{}, \tensor[_4]{\psi}{})\tensor[_4]{\psi}{}
		= -\frac13\tensor[_3]{{SR^N(\psi)}}{},
	\end{align}
	that is,
	\begin{equation}
		\delta^{\kappa\lambda}R^N(\tensor[_\sigma]{\psi}{}, \tensor[_\kappa]{\psi}{})\tensor[_\lambda]{\psi}{}
		= \frac13\tensor{\ACI}{_\sigma^\tau}\tensor[_\tau]{{SR^N(\psi)}}{}.
	\end{equation}
	For the second summand,
	\begin{align}
		\begin{split}
			\varepsilon^{\kappa\tau}\tensor{\gamma}{_1_\tau^\lambda}R^N(\tensor[_3]{\psi}{}, \tensor[_\kappa]{\psi}{})\tensor[_\lambda]{\psi}{}
			&= R^N(\tensor[_3]{\psi}{}, \tensor[_3]{\psi}{})\tensor[_3]{\psi}{} - R^N(\tensor[_3]{\psi}{}, \tensor[_4]{\psi}{})\tensor[_4]{\psi}{} \\
			&= -R^N(\tensor[_3]{\psi}{}, \tensor[_4]{\psi}{})\tensor[_4]{\psi}{}
			= -\frac13\tensor[_4]{{SR^N(\psi)}}{},
		\end{split}\\
		\begin{split}
			\varepsilon^{\kappa\tau}\tensor{\gamma}{_1_\tau^\lambda}R^N(\tensor[_4]{\psi}{}, \tensor[_\kappa]{\psi}{})\tensor[_\lambda]{\psi}{}
			&= R^N(\tensor[_4]{\psi}{}, \tensor[_3]{\psi}{})\tensor[_3]{\psi}{} - R^N(\tensor[_4]{\psi}{}, \tensor[_4]{\psi}{})\tensor[_4]{\psi}{} \\
			&= R^N(\tensor[_3]{\psi}{}, \tensor[_4]{\psi}{})\tensor[_3]{\psi}{}
			= -\frac13\tensor[_3]{{SR^N(\psi)}}{},
		\end{split} \\
		\begin{split}
			\varepsilon^{\kappa\tau}\tensor{\gamma}{_2_\tau^\lambda}R^N(\tensor[_3]{\psi}{}, \tensor[_\kappa]{\psi}{})\tensor[_\lambda]{\psi}{}
			&= R^N(\tensor[_3]{\psi}{}, \tensor[_3]{\psi}{})\tensor[_4]{\psi}{} + R^N(\tensor[_3]{\psi}{}, \tensor[_4]{\psi}{})\tensor[_3]{\psi}{} \\
			&= -R^N(\tensor[_3]{\psi}{}, \tensor[_4]{\psi}{})\tensor[_3]{\psi}{}
			= \frac13\tensor[_3]{{SR^N(\psi)}}{},
		\end{split}\\
		\begin{split}
			\varepsilon^{\kappa\tau}\tensor{\gamma}{_2_\tau^\lambda}R^N(\tensor[_4]{\psi}{}, \tensor[_\kappa]{\psi}{})\tensor[_\lambda]{\psi}{}
			&= R^N(\tensor[_4]{\psi}{}, \tensor[_3]{\psi}{})\tensor[_4]{\psi}{} + R^N(\tensor[_4]{\psi}{}, \tensor[_4]{\psi}{})\tensor[_3]{\psi}{} \\
			&= R^N(\tensor[_3]{\psi}{}, \tensor[_4]{\psi}{})\tensor[_3]{\psi}{}
			= -\frac13\tensor[_4]{{SR^N(\psi)}}{},
		\end{split}
	\end{align}
	yields
	\begin{equation}
		\varepsilon^{\kappa\tau}\tensor{\gamma}{_t_\tau^\lambda}R^N(\tensor[_\sigma]{\psi}{}, \tensor[_\kappa]{\psi}{})\tensor[_\lambda]{\psi}{}
		= -\frac13\tensor{\gamma}{_t_\sigma^\tau}\tensor[_\tau]{{SR^N(\psi)}}{}.
	\end{equation}
	This shows the first claim.

	The identities for the derivatives of the curvature tensor follow analogously, because the tensor \(\nabla_{\tensor[_\rho]{\psi}{}}R^N\) has the same symmetry properties as \(R^N\).
\end{proof}

\printbibliography

\textsc{Enno Keßler\\
Center of Mathematical Sciences and Applications,
Harvard University,
20~Garden Street,
Cambridge, MA 02138,
USA}\\
\texttt{ek@cmsa.fas.harvard.edu} \\

\textsc{Artan Sheshmani\\
Center of Mathematical Sciences and Applications,
Harvard University,
20~Garden Street,
Cambridge, MA 02138,
USA\\[.5em]
Aarhus University,
Department of Mathematics, Ny Munkegade 118,
8000, Aarhus C,
Denmark\\[.5em]
National Research University
Higher School of Economics, Russian Federation,
Laboratory of Mirror Symmetry,
NRU HSE,
6 Usacheva Street,
Moscow, Russia, 119048}\\
\texttt{artan@cmsa.fas.harvard.edu} \\

\textsc{Shing-Tung Yau\\
Center of Mathematical Sciences and Applications,
Harvard University,
20~Garden Street,
Cambridge, MA 02138,
USA\\[.5em]
Department of Mathematics,
Harvard University,
Cambridge, MA 02138,
USA
}\\
\texttt{yau@math.harvard.edu}

\end{document}